\documentclass[11pt]{article}
\usepackage[utf8]{inputenc}
\usepackage[british]{babel}
\usepackage{lmodern}
\usepackage{comment}
\usepackage{amssymb,amsmath}
\usepackage{amsthm}
\usepackage{mathtools}
\usepackage{graphicx,algorithmic,algorithm}%
\usepackage{enumerate}
\usepackage{tikz,tikz-cd}
\usepackage{circuitikz}
\tikzcdset{scale cd/.style={every label/.append style={scale=#1},
    cells={nodes={scale=#1}}}}%
\usetikzlibrary{decorations.markings}
\usetikzlibrary{shapes}
\usepackage[letterpaper, margin=1in]{geometry}
\usepackage{hyperref}
\usepackage[nameinlink,capitalize]{cleveref}
\usepackage{todonotes}
\usepackage{dsfont}
\usepackage{soul}
\usepackage[normalem]{ulem}
\hypersetup{
	pdftitle=   {},
	pdfauthor=  {}
}
\usepackage{authblk}
\usepackage{caption}
\usepackage{subcaption}
\usepackage{commath}

\newtheorem{theorem}{Theorem}[section]
\newtheorem{lemma}[theorem]{Lemma}

\newtheorem{observation}[theorem]{Observation}
\newtheorem{claim}[theorem]{Claim}
\newtheorem{question}{Question}

\newenvironment{clproof}{\begin{list}{}{%
			\setlength{\leftmargin}{3mm}%
		} \item {\it Proof.} }{\hfill$\lozenge$\end{list}}

\newcommand\extrafootertext[1]{%
    \bgroup
    \renewcommand\thefootnote{\fnsymbol{footnote}}%
    \renewcommand\thempfootnote{\fnsymbol{mpfootnote}}%
    \footnotetext[0]{#1}%
    \egroup
}
\newcommand{\circum}{\mathsf{circ}}
\newcommand{\crank}{\mathsf{cr}}

\newcommand{\dtw}{\mathsf{dtw}}

\newcommand{\dist}{\mathsf{dist}}

\newcommand{\invertex}{\mathsf{in}\text{-}\mathsf{vertex}}
\newcommand{\outvertex}{\mathsf{out}\text{-}\mathsf{vertex}}

\newcommand{\inarbor}{\mathsf{in}\text{-}\mathsf{arbor}}
\newcommand{\outarbor}{\mathsf{out}\text{-}\mathsf{arbor}}

\newcommand{\onewc}{\mathrm{wcol}^{\rightarrow}}
\newcommand{\onesc}{\mathrm{scol}^{\rightarrow}}
\newcommand{\onewre}{\mathrm{WReach}^{\rightarrow}}
\newcommand{\onesre}{\mathrm{SReach}^{\rightarrow}}
\newcommand{\twowc}{\mathrm{wcol}^{\leftrightarrow}}
\newcommand{\twowre}{\mathrm{WReach}^{\leftrightarrow}}

\newcommand{\ladder}{\mathrm{L}}
\newcommand{\cycleCh}{\mathrm{CC}}
\newcommand{\treeCh}{\mathrm{TC}}
\newcommand{\treeChFam}{\mathcal{TC}}

\newcommand{\ord}{\mathsf{ord}}
\newcommand{\weight}{\mathsf{weight}}
\newcommand{\inarb}{\mathsf{Inarb}}
\newcommand{\outarb}{\mathsf{Outarb}}
\newcommand{\inext}{\mathsf{Inext}}
\newcommand{\outext}{\mathsf{Outext}}
\newcommand{\head}{\mathsf{head}}
\newcommand{\tail}{\mathsf{tail}}

\newcommand{\outdeg}{\mathsf{deg}^{+}}

\newcommand{\lex}{\textrm{lex}}

\hypersetup{
	colorlinks=true,
	linkcolor=green!50!black,
	citecolor=violet,
	urlcolor=blue!80!white,
	bookmarksopen=true,
	bookmarksnumbered,
	bookmarksopenlevel=2,
	bookmarksdepth=3
}

\def\multiset#1#2{\ensuremath{\left(\kern-.3em\left(\genfrac{}{}{0pt}{}{#1}{#2}\right)\kern-.3em\right)}}

\begin{document}

	\title{Unavoidable butterfly minors in digraphs of large cycle rank}
	
	\author[1]{Meike Hatzel}
	\author[1,2]{O-joung Kwon}
	\author[2]{Myounghwan Lee}
	\author[3]{Sebastian Wiederrecht}

	\affil[1]{Discrete Mathematics Group, Institute for Basic Science (IBS), Daejeon, South Korea}
	\affil[2]{Department of Mathematics, Hanyang University, Seoul, South Korea.}
    \affil[3]{School of Computing, KAIST, Daejeon, South Korea}
	
	\date{}
	\maketitle

	\extrafootertext{O.~Kwon and M.~Lee are supported by the National Research Foundation of Korea (NRF) grant funded by the Ministry of Science and ICT (No.~RS-2023-00211670). M.~Hatzel and O.~Kwon are supported by the Institute for Basic Science (IBS-R029-C1).}

	\extrafootertext{E-mail addresses: 
  \texttt{meikehatzel@ibs.re.kr} (M.~Hatzel),
  \texttt{ojoungkwon@hanyang.ac.kr} (O.~Kwon), \texttt{sycuel@hanyang.ac.kr} (M.~Lee), and 
  \texttt{wiederrecht@kaist.ac.kr}
  (S.~Wiederrecht)}

\begin{abstract}
    Cycle rank is one of the depth parameters for digraphs introduced by Eggan in 1963. 
    We show that there exists a function $f:\mathds{N}\to \mathds{N}$ such that every digraph of cycle rank at least $f(k)$ contains a \textsl{directed cycle chain}, a \textsl{directed ladder}, or a \textsl{directed tree chain} of order $k$ as a butterfly minor.
    We also investigate a new connection between cycle rank and a directed analogue of the weak coloring number of graphs.

 \end{abstract}

\section{Introduction}

The graph parameter \textsl{treedepth} measures how far a graph is from being a star.
It was in its current form introduced by Ne\v{s}et\v{r}il and Ossona de Mendez~\cite{NesetrilO2006} in 2006; however, equivalent or similar notions have been known and worked on before this~\cite{oldTreedepth1,oldTreedepth2,oldTreedepth3,oldTreedepth4}.
Treedepth plays an important role in the study of sparse graph classes~\cite{NesetrilO2012}.
For example, classes of bounded expansion are characterized as classes of graphs admitting low \textsl{treedepth colorings}. 
It is well known that a graph has large treedepth if and only if it contains a long path, see~\cite[Section 6.2]{NesetrilO2012}.

A natural analogue of treedepth for digraphs is the parameter \textsl{cycle rank}, which was introduced much earlier than treedepth (or any equivalent notion) in 1963 by Eggan~\cite{Eggan1963} in the context of studying the star height of regular languages.
It is however hard to determine the cycle rank of a given digraph.
In fact, Gruber~\cite{Gruber2012} proved that the problem of computing the cycle rank of a digraph is NP-complete even for digraphs of maximum out-degree~$2$.
Later, Giannopoulou, Hunter, and Thilikos~\cite{GIANNOPOULOU2012searchinggame} observed that for every fixed integer $k$, one can decide whether an $n$-vertex digraph has cycle rank at most $k$ in time $\mathcal{O}(n^k)$.
Additionally, they characterized classes of bounded cycle rank using several types of obstacles, such as \textsl{LIFO-havens}.  
In 2014, Ganian et al.~\cite{Ganian2014} considered cycle rank as a parameter and analyzed the complexity of several digraph problems with respect to it.

Despite several results on cycle rank, little is known about the structure in digraphs of large cycle rank.
Are there certain substructures that every digraph of large cycle rank has to contain? 
Wiederrecht~\cite{cabello202210th}
posed the problem of finding unavoidable butterfly minors in digraphs of large cycle rank at GROW 2022. Note that the cycle rank of a digraph does not increase when taking a butterfly minor, see \cref{lem:butterflyminor}.

Our main result is providing three obstruction families for cycle rank in terms of butterfly minors.
Characterizing a graph width parameter in terms of such obstructions, that is, substructures witnessing high width, not only yields a deeper understanding of it but also offers a way to certify that the given parameter is large. Finding unavoidable structures for undirected width parameters is well understood for many parameters, for example, for treewidth~\cite{RobertsonS1986,ChuzhoyT2019}, pathwidth~\cite{BienstockRST1991}, treedepth~\cite[Chapter 6]{NesetrilO2012}, tree-cut width~\cite{Wollan2015} or rank-width~\cite{GeelenKMW2023}.
However, finding such obstructions for digraph width parameters turns out to be very difficult, and results proving finitely many obstruction families are extremely rare.
In fact, there are essentially three examples of such results:
\begin{itemize}
    \item Younger~\cite{younger} conjectured in 1973 that a digraph contains many disjoint cycles if and only if it does not contain a small \textsl{feedback vertex set}, and this conjecture was confirmed by Reed, Robertson, Seymour, and Thomas~\cite{ReedRST1996} in 1996 and
    \item Johnson, Robertson, Seymour, and Thomas~\cite{JohnsonRST2001} conjectured in 2001 that a digraph has large \textsl{directed treewidth} if and only if it contains a large cylindrical grid as a butterfly minor, and this conjecture was confirmed by Kawarabayashi and Kreutzer~\cite{kawarabayashi2015directed} in 2015 (called the directed grid theorem),  
    \item Ganian et al.~\cite{Ganian2014} showed that a digraph has large DAG-depth if and only if it contains a long path as a subdigraph.
\end{itemize}
The first two statements remained open for a significant period before being solved, and the proofs are quite complicated and involved. Several width parameters inspired by treewidth have been introduced for digraphs (we refer to the survey in~\cite{Kreutzer2018}), for many of which, like DAG-width, Kelly-width, and directed pathwidth, unavoidable structures are still unknown, see~\cref{fig:parameters} for an overview.
This makes it all the more surprising that the families we present are rather simple, and that the proof does not reach the technical complexity of the two results above.

\begin{figure}[t]
\captionsetup[subfigure]{width=0.66\textwidth}
\;
\begin{minipage}{0.3\linewidth}
\centering
\includegraphics[width=\textwidth]{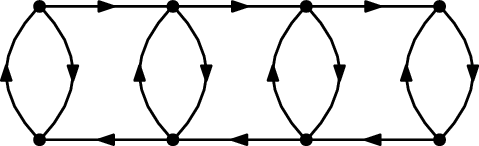}
\subcaption*{$\ladder_4$}
\end{minipage}
\quad
\begin{minipage}{0.3\linewidth}
    \centering
    \vskip 0.4cm 
\includegraphics[width=0.9\textwidth]{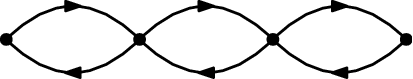}
    \vskip 0.4cm 
\subcaption*{$\cycleCh_4$}
    \end{minipage}
    \quad
\begin{minipage}{0.3\linewidth}
    \centering
    \includegraphics[scale=0.5]{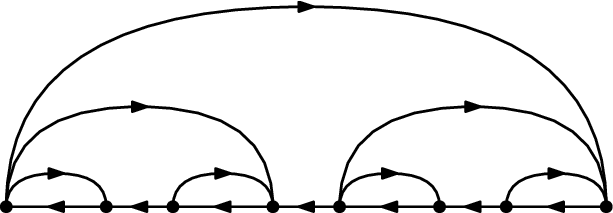}
    \subcaption*{$\treeCh_3$}
    \end{minipage}
\caption{The directed ladder $\ladder_4$ of order $4$, the directed cycle chain $\cycleCh_4$ of order $4$, and the directed tree chain $\treeCh_3$ of order $3$.}
\label{fig:ladderandchain}
\end{figure}

More specifically, we prove that every digraph of sufficiently large cycle rank contains a butterfly minor isomorphic to one of the following three types: \textsl{directed ladders} $\ladder_k$, \textsl{directed cycle chains} $\cycleCh_k$, and \textsl{directed tree chains} $\treeCh_{k}$, illustrated in~\cref{fig:ladderandchain}. 
The most involved of these are the directed tree chains, which are defined recursively in a tree-like fashion.
Briefly speaking, $\treeCh_{0}$ is the one-vertex graph, and for $k\ge 1$, $\treeCh_{k}$ is obtained from the disjoint union of two copies of $\treeCh_{k-1}$ by linking them using two edges as illustrated in the right-most digraph in~\cref{fig:ladderandchain}.
We present formal definitions in \cref{sec:ladderchain}, and prove that all three families have unbounded cycle rank.
We also prove, in \cref{sec:independent}, that these families are pairwise independent; for example, all digraphs in $\{\ladder_k:k\in \mathds{N}\}$ do not contain a fixed $\cycleCh_t$ as a butterfly minor, and so on.

\begin{theorem}\label{thm:cyclerankmainthm}
There is a function $f_{\ref{thm:cyclerankmainthm}}:\mathds{N}\to \mathds{N}$ satisfying the following. For every positive integer $k$, every digraph of cycle rank at least $f_{\ref{thm:cyclerankmainthm}}(k)$ contains $\ladder_k$, $\cycleCh_k$, or $\treeCh_k$ as a butterfly minor.  
\end{theorem}

 \cref{thm:cyclerankmainthm} provides a characterization of all butterfly-minor-closed classes of digraphs of bounded cycle rank.

\begin{theorem}\label{thm:bfclosedClassesCycleRank}
A butterfly-minor-closed class $\mathcal{C}$ of digraphs has bounded cycle rank if and only if there exists $k\in\mathds{N}$ such that $\mathcal{C}$ does not contain $\ladder_k$, $\cycleCh_k$, or $\treeCh_k$.
\end{theorem}

Our function $f_{\ref{thm:cyclerankmainthm}}$ in~\cref{thm:cyclerankmainthm} is of the form
\begin{align*}
f_{\ref{thm:cyclerankmainthm}}(k)=2^{2^{2k+2}\cdot(2k+2)(4k^2+k)+(2k+2)}\cdot(2k+2)(4k^2+k)^2\cdot f_{\mathsf{dtw}}(k-1)+2,
\end{align*}
where $f_{\mathsf{dtw}}$ is the function from the directed grid theorem (\cref{thm:directedgrid}).
Due to the work by Hatzel, Kreutzer, Milani, and Muzi \cite{HatzelKMM2024}, we know that $f_{\mathsf{dtw}}(k)$ is ``only'' $22$-fold exponential in $k$ as opposed to the original non-elementary bound of Kawarabayashi and Kreutzer \cite{kawarabayashi2015directed}.
It is an interesting question whether~\cref{thm:cyclerankmainthm} can be proved without using the directed grid theorem. A positive answer to this question may lead to a significantly improved function for~\cref{thm:cyclerankmainthm}.
For the case of planar digraphs, Hatzel, Kawarabayashi, and Kreutzer~\cite{HatzelKK2019} obtained a polynomial bound for the grid theorem, and by applying their result, we obtain a double-exponential bound for planar digraphs.

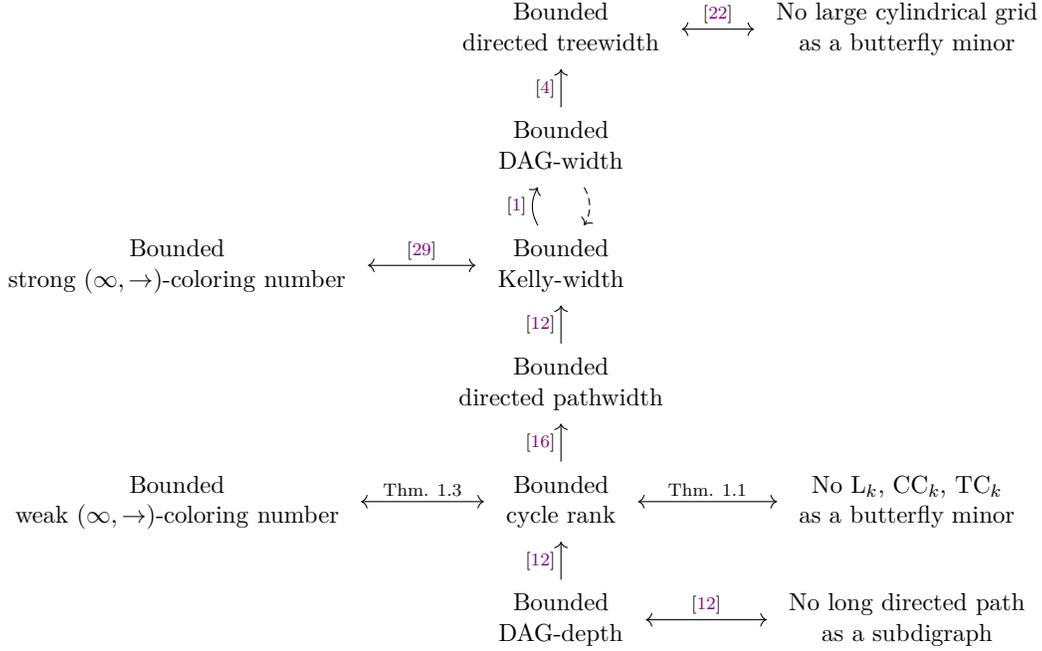
\begin{figure}[t]
    \centering
    \[\begin{tikzcd}[column sep=normal, scale cd=0.85, row sep=scriptsize]
    &\begin{array}{c}\text{Bounded}\\
    \text{directed treewidth}\end{array}\arrow[r,leftrightarrow,"\text{\cite{kawarabayashi2015directed}}"] & \begin{array}{c}\text{No large cylindrical grid}\\ \text{as a butterfly minor}\end{array}\\
    & \begin{array}{c}\text{Bounded}\\ \text{DAG-width}\end{array} \arrow[u,"\text{\cite{berwanger2012dag}}"]\arrow[d, bend left, dashed] &\\
    \begin{array}{c}\text{Bounded}\\ \text{strong $(\infty,\to)$-coloring number}\end{array}\arrow[r,leftrightarrow,"\text{\cite{meister2010recognizing}}"] &\begin{array}{c}\text{Bounded}\\ \text{Kelly-width}\end{array} \arrow[u, bend left,"\text{\cite{AmiriKRS2015}}"]&\\
    &\begin{array}{c}\text{Bounded}\\ \text{directed pathwidth}\end{array} \arrow[u,"\text{\cite{Ganian2014}}"] &\\
    \begin{array}{c}\text{Bounded}\\ \text{weak $(\infty,\to)$-coloring number}\end{array}\arrow[r,leftrightarrow,"\text{Thm. 1.3}"]&\begin{array}{c}\text{Bounded}\\ \text{cycle rank}\end{array} \arrow[u,"\text{\cite{Gruber2012}}"]\arrow[r,leftrightarrow,"\text{Thm. 1.1}"] & \begin{array}{c}\text{No $\ladder_k$, $\cycleCh_k$, $\treeCh_k$}\\ \text{as a butterfly minor}\end{array}\\
    &\begin{array}{c}\text{Bounded}\\ \text{DAG-depth}\end{array}\arrow[u,"\text{\cite{Ganian2014}}"]\arrow[r,leftrightarrow,"\text{\cite{Ganian2014}}"] & \begin{array}{c}\text{No long directed path}\\ \text{as a subdigraph}\end{array}
\end{tikzcd}\]
    \caption{The hierarchy of the mentioned digraph classes.
    We write $A\rightarrow B$ if every class with property $A$ satisfies the property $B$.
    The relationships between cycle rank, Kelly-width and weak $(\infty,\to)$-coloring number, strong $(\infty,\to)$-coloring number, respectively, are proved in \cref{sec:wcol}.
    Whether or not every class of bounded DAG-width has bounded Kelly-width is an open problem.
    }
    \label{fig:parameters}
\end{figure}

We investigate a new connection between cycle rank and a directed analogue of the weak coloring number~\cite{KiersteadY2003}. Weak coloring numbers play an important role in the theory of sparse graph classes developed by Ne\v{s}et\v{r}il and Ossona de Mendez~\cite{NesetrilO2008,NesetrilO2008-2,NesetrilO2008-3,NesetrilO2012}. In particular, classes of undirected graphs of bounded expansion are characterized in terms of weak coloring numbers~\cite{Zhu2009}. Given a linear ordering $L$ of an undirected graph $G$, a vertex $u$ that appears before a vertex $v$ is called weakly $k$-reachable from $v$ if there is a path of length at most $k$ from $u$ to $v$, where all internal vertices of the path appear after $u$ in the ordering. When there is no restriction on the length $k$, we say that $u$ is weakly $\infty$-reachable from $v$. By replacing the reachability condition with `there is a directed path of length at most $k$ from $v$ to $u$' in digraphs, we obtain the weak $(k, \rightarrow)$-coloring number of a digraph. See~\cref{sec:wcol} for its formal definition. We prove the following relationship, which is analogous to the one between treedepth and weak $\infty$-coloring number~\cite[Lemma 6.5]{NesetrilO2012}. 

\begin{theorem}\label{thm:CyclerankWcol}
    The cycle rank of a digraph is equal to its weak $(\infty, \rightarrow)$-coloring number minus one.
\end{theorem}

Note that the minus one is natural, as acyclic digraphs are defined to have cycle rank zero.

We may consider another reachability condition that `there is a directed path of length at most $k$ from $u$ to $v$, or from $v$ to $u$'. This leads to the definition of the weak $(k, \leftrightarrow)$-coloring number, introduced by Kreutzer, Rabinovich, Siebertz, and Weberst\"{a}dt~\cite{KreutzerRSW2017}. It was used to characterize classes of bounded directed expansion. In the full version of Kreutzer, Muzi, Ossona de Mendez, Rabinovich, and Siebertz~\cite{KORS2017}, they introduced the directed treedepth of a digraph as its weak $(\infty, \leftrightarrow)$-coloring number.
In their theory of sparsity, Ne\v{s}et\v{r}il and Ossona de Mendez relate the weak coloring numbers of undirected graphs also to strong coloring numbers.
While for undirected graphs weak coloring numbers with $k=\infty$ describe the parameter treedepth, the strong coloring numbers capture treewidth in a similar way.
 The strong variant with a fixed direction is the strong~$(k, \rightarrow)$-coloring number of a digraph, and Meister, Telle, and Vatshelle~\cite[Lemma 3.3]{meister2010recognizing} already showed that the Kelly-width~\cite{HUNTER2008206} of a digraph is equal to the strong~$(\infty,\rightarrow)$-coloring number, although they did not use this terminology. 

As mentioned before, 
Giannopoulou, Hunter, and Thilikos~\cite{GIANNOPOULOU2012searchinggame} gave an XP algorithm for testing whether the cycle rank of a given digraph is at most $k$. They posed an open question about whether there exists an FPT algorithm for this problem, which was also mentioned by Gruber~\cite{Gruber2012}. Even the possibility of having an FPT-time factor-$g(k)$ approximation algorithm remains an interesting open problem. Our proof is mostly constructive and comes close to yielding such an algorithm, but there is a part involving an Erd\H{o}s-P\'osa-type argument (\cref{thm:inductionmainstar}) for which it is unclear whether it can be turned into an FPT algorithm; see the sketch of the proof below for more details.

\medskip 

We now sketch the proof of~\cref{thm:cyclerankmainthm}.

We use the directed grid theorem~\cite{kawarabayashi2015directed} (\cref{thm:directedgrid}) by Kawarabayashi and Kreutzer to argue that we can assume that a given digraph has small directed treewidth. The directed grid theorem states that every digraph of sufficiently large directed treewidth contains a cylindrical grid of order~$k$ as a butterfly minor. See~\cref{fig:grid} for an illustration of a cylindrical grid. 
In~\cref{sec:ladderchain}, we show that the cycle chain of order $k+1$ is a butterfly minor of a cylindrical grid of order $k$.
Thus, if a given digraph has large directed treewidth, then the desired structure already appears. Therefore, we may assume that a given digraph has bounded directed treewidth. Note that we prove in~\cref{lem:FnotembeddableinGrid} that $\treeCh_k$ with $k\ge 3$ is not a butterfly minor of a cylindrical grid. 

Next, we use an Erd\H{o}s-P\'osa type argument for strongly connected graphs on digraphs of bounded directed treewidth.
Let $\mathcal{F}$ be a set of strongly connected digraphs. In~\cref{lem:erdosposa}, we show that for every digraph $G$ of bounded directed treewidth and integer $k$, we can find either $k$ vertex-disjoint subdigraphs, each of which is isomorphic to a digraph in $\mathcal{F}$, or a small vertex set hitting all subdigraphs isomorphic to a digraph in $\mathcal{F}$. 

Let $\mathcal{M}_1$ be the set of all strongly connected digraphs, and for each integer $i\ge 2$, let $\mathcal{M}_i$ be the set of strongly connected digraphs that can be obtained from the disjoint union of two digraphs $H_1$ and $H_2$ in $\mathcal{M}_{i-1}$ by linking them using a directed path from $H_1$ to $H_2$ and a directed path from $H_2$ to $H_1$.
So while getting smaller and smaller, the classes $\mathcal{M}_i$ also get more and more structured.
Using the aforementioned Erd\H{o}s-P\'osa type argument, we deduce that one can obtain a digraph in $\mathcal{M}_t$ with large $t$ as a subdigraph.
The construction of a digraph in $\mathcal{M}_i$ is represented using a decomposition tree, called a chain decomposition, introduced in~\cref{sec:chaindecomp}.

Conceptually, digraphs in $\mathcal{M}_i$ have similar structure to $\treeCh_k$. However, extracting $\treeCh_k$ as a butterfly minor from a digraph in $\mathcal{M}_i$ with large $i$ is a highly involved procedure. In particular, there are several cases in which the two paths $P$ and $Q$ link the two digraphs $H_1$ and $H_2$ in $\mathcal{M}_{i-1}$. 
We want the paths $P$ and $Q$ to be linked in a way similar to how they are in $\treeCh_k$. 
For this argument, we introduce special types of chain decompositions with two extra properties called being \textsl{clean} and \textsl{spotless}, which are introduced in \cref{sec:chaindecomp}.

In~\cref{sec:mainproof}, we show that from a digraph in $\mathcal{M}_i$ with sufficiently large $i$, one can either find a directed ladder or a directed cycle chain of large order, or obtain a digraph in $\mathcal{M}_{i'}$ with large $i'$ that admits a spotless chain decomposition. To prove this, we introduce an intermediate structure called a mixed chain in~\cref{sec:relaxedversion}, which combines relaxations of directed chains and directed ladders. The complexity of a mixed chain is measured by a weight (see~\cref{fig:mixedchain} for an illustration). We show that a mixed chain of large weight always contains a directed chain or a directed ladder of order $k$. Accordingly, instead of directly linking two digraphs by two paths, we consider them to be linked by a mixed chain. Whenever a structure in a chain decomposition violates the desired properties, we can construct a mixed chain with increased weight. Finally, when we obtain a spotless chain decomposition of sufficiently large height, we can extract $\treeCh_k$ as a butterfly minor, thereby completing the proof of the theorem.

\medskip 

The paper is organized as follows. In~\cref{sec:prelim}, we introduce some preliminary concepts on digraphs including butterfly minors, cycle rank, and directed treewidth. In~\cref{sec:ladderchain}, we formally define the three basic obstructions that are directed ladders, directed cycle chains, and directed tree chains. We show that they have large cycle rank, and directed ladders and directed cycle chains are butterfly minors of some cylindrical grids. In~\cref{sec:relaxedversion}, we introduce intermediate structures called relaxed ladders, relaxed chains, mixed chains, and relaxed tree chains. We prove some lemmas on mixed chains in~\cref{sec:mixedchain} that are used to refine a chain decomposition of a digraph.
In~\cref{sec:chaindecomp}, we define chain decompositions of digraphs, which capture recursive structures of digraphs, and introduce related notions such as clean and spotless decompositions. We prove~\cref{thm:cyclerankmainthm} in~\cref{sec:mainproof}
and show in~\cref{sec:independent} that the three types of digraphs in~\cref{thm:cyclerankmainthm} are pairwise independent. 
We prove that cycle rank is equal to weak $(\infty,\rightarrow)$-coloring number minus one in~\cref{sec:wcol}. In~\cref{sec:conclusion}, we conclude our paper with a
discussion on future research directions.

\section{Preliminaries}\label{sec:prelim}

We denote by $\mathds{N}$ the set of all positive integers. For an integer $n$, let $[n]$ denote the set of positive integers at most $n$.
All logarithms in this paper are taken at base $2$.

In this paper, all undirected graphs and digraphs are finite. 
Let $G$ be a digraph. 
We denote by $V(G)$ and $E(G)$ the vertex set and the edge set of $G$, respectively.
If $(v,w)$ is an edge in $G$, then $v$ is its \emph{tail} and $w$ is its \emph{head}.
We also say that $v$ is an \emph{in-neighbor} of $w$ and $w$ is an \emph{out-neighbor} of $v$.
For a vertex $v$ of $G$, let $N_G^+(v)$ denote the set of out-neighbors of $v$, and $N_G^-(v)$ denote the set of in-neighbors of $v$.
The \emph{in-degree} of $v$, denoted by $\deg_G^{-}(v)$, is the number of in-neighbors of~$v$, and the \emph{out-degree} of~$v$, denoted by $\deg_G^{+}(v)$, is the number of out-neighbors of~$v$.
For a set $A$ of vertices in a digraph $G$, we denote by $G-A$ the digraph obtained from $G$ by removing all the vertices in $A$, and denote by $G[A]$ the subdigraph of $G$ induced by $A$.
For a vertex $v$ of a digraph $G$, we write $G-v\coloneqq G-\{v\}$.
For two digraphs $G$ and $H$, let $G\cup H\coloneqq (V(G)\cup V(H), E(G)\cup E(H))$ and $G\cap H\coloneqq (V(G)\cap V(H), E(G)\cap E(H))$.  For a digraph $G$, its \emph{underlying graph} is the undirected graph obtained from $G$ by replacing every directed edge in $G$ with an undirected edge with the same endpoints, and then removing parallel edges.

For a directed path $P=v_1v_2 \cdots v_n$, we call $v_1$ and $v_n$ the \emph{tail} and \emph{head} of $P$, respectively. We write $\tail(P)=v_1$ and $\head(P)=v_n$.
For $v_i, v_j\in V(P)$ with $i\le j$, we denote by $v_iPv_j$ the subpath of $P$ from $v_i$ to $v_j$.

A digraph is \emph{acyclic} if it has no directed cycles.
A digraph $G$ is \emph{strongly connected} if for every two vertices $u$ and $v$ of $G$, there exist a directed path from $u$ to $v$ and a directed path from $v$ to $u$ in $G$.
A maximal strongly connected subdigraph of a digraph is called its \emph{strongly connected component}.
A digraph is \emph{weakly connected} if its underlying graph is connected. 
A maximal weakly connected subdigraph of a digraph is called its \emph{weakly connected component}.

For sets $A$ and $B$ of vertices in a digraph $G$, a directed path in $G$ is an \emph{$(A, B)$-path} if it starts in $A$ and ends in $B$, and all its internal vertices are not in $A\cup B$.
We say that an $(A, B)$-path \emph{leaves} $A$ and \emph{enters} $B$.
If $A$ or $B$ consists of a single vertex $v$, we simplify notation by writing $v$ instead of $\{v\}$.
For instance, if $A=\{v\}$, then a $(v, B)$-path is a $(\{v\}, B)$-path.
For two subdigraphs $F$ and $H$ of $G$, a $(V(F), V(H))$-path is shortly denoted as an $(F, H)$-path.

A \emph{tree} is a connected undirected graph that has no cycle.
A \emph{rooted tree} is a pair $(T,r)$ where $T$ is a tree and $r$ is a node of $T$, called the root of $T$. A disjoint union of rooted trees is called a \emph{rooted forest}.
Let $(T,r)$ be a rooted tree.
A node $t$ with $t\neq r$ of $(T,r)$ is called a \emph{leaf} if it is of degree exactly $1$.
Otherwise, we call it an \emph{internal node}.
A node $u$ of $T$ is a \emph{descendant} of a node $v$ if the unique path between $u$ and $r$ in $T$ contains $v$.
In this case, we also say that $v$ is an \emph{ancestor} of $u$. 
Additionally, a descendant or an ancestor $u$ of $v$ is \emph{proper} if $u\neq v$.
A node $u$ of $T$ is a \emph{child} of a node $v$ if it is a proper descendant of $v$ that is adjacent to $v$.
In this case, we also say that $v$ is the \emph{parent} of $u$.
For a node $t$ in $T$, let us define $T_t$ as the rooted tree $(T',t)$ where $T'$ is the subgraph of $T$ induced by the set of all descendants of $t$ in $T$.
The \emph{height} of a rooted tree $(T,r)$ is the number of vertices of a longest path from $r$ to a leaf.
For two nodes $v,w$ of a tree $T$, we denote by $\dist_T(v,w)$ the length of the unique path from $v$ to $w$ in $T$.

An \emph{out-arborescence} is an orientation of a rooted tree $(T, r)$ such that for every node $t$ of $T$, there is a directed path from $r$ to $t$.
An \emph{in-arborescence} is an orientation of a rooted tree $(T, r)$ such that for every node $t$ of $T$, there is a directed path from $t$ to $r$.

For a digraph $G$ and $v\in V(G)$, let $\outarb(G,v)$ be the maximal out-arborescence $H$ of $G$ rooted at $v$ such that every vertex in $V(H)\setminus \{v\}$ has in-degree exactly $1$ in $G$.
For a digraph $G$ and $v\in V(G)$, let $\inarb(G,v)$ be the maximal in-arborescence $H$ of $G$ rooted at $v$ such that every vertex in $V(H)\setminus \{v\}$ has out-degree exactly $1$ in $G$.

For a vector $\mathsf{v}$, we write $\mathsf{v}(i)$ for the $i$-th entry of $\mathsf{v}$.
For integers $i\le j$, we write $\mathsf{v}[i,j]$ for the vector consisting of the $i$-th entry to the $j$-th entry of $\mathsf{v}$.
For example, if $\mathsf{v}=(2,7,1,5,4,8)$, then $\mathsf{v}[3,5]=(1,5,4)$.
We write $\norm{\mathsf{v}}_1$ for the sum of all entries of $\mathsf{v}$.
For two vectors $\mathsf{v}$ and $\mathsf{w}$ of integers, we write $\mathsf{v}<_{\lex} \mathsf{w}$ if there exists a positive integer $t$ such that
\begin{itemize}
    \item $\mathsf{v}(i)=\mathsf{w}(i)$ for all integers $i\in [t-1]$, and 
    \item $\mathsf{v}(t)<\mathsf{w}(t)$.
\end{itemize}

\begin{figure}
    \centering
    \includegraphics[scale=0.5]{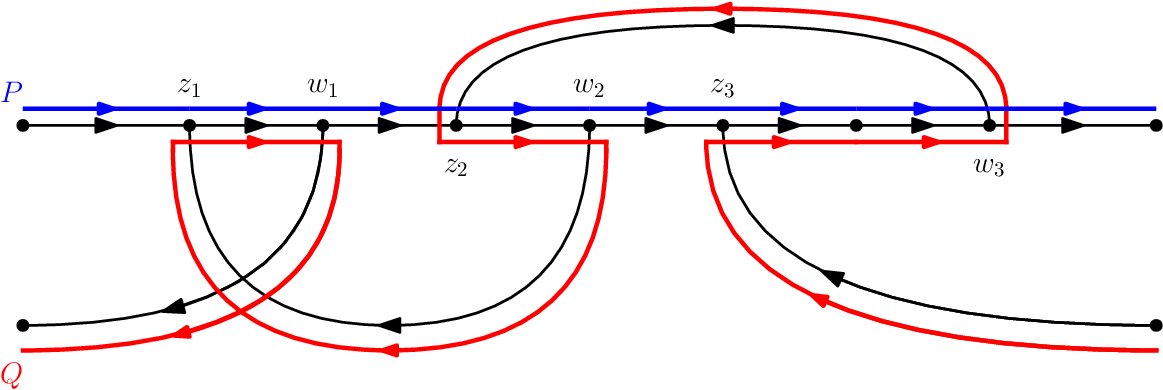}
    \caption{An example of laced paths $P$ and $Q$.}
    \label{fig:lacedpaths}
\end{figure}

Two directed paths $P$ and $Q$ in a digraph are \emph{laced} if either $P$ and $Q$ are disjoint or there exist $z_1,w_1,\dots,z_n,w_n \in V(P)\cap V(Q)$ such that 
\begin{itemize}
    \item  $\{z_1Pw_1,\ldots,z_nPw_n\}$ is the set of all weakly connected components of $P\cap Q$,
    \item $z_1Pw_1, z_2Pw_2, \ldots,z_nPw_n$ appear in this order along $P$, and 
    \item $z_nPw_n, z_{n-1}Pw_{n-1}, \ldots, z_1Pw_1$ appear in this order along $Q$.
\end{itemize}
See \cref{fig:lacedpaths} for an illustration.
The following lemmas allow us to untangle a pair of directed paths into a laced pair. 
\begin{lemma}[Lemma~4.3 of \cite{Hatzel2024Generating}]\label{lem:laced1}
  Let $P$ be an $(a,b)$-path and $Q$ be a $(c,d)$-path in a digraph $G$.
  Then, there exists a $(c,d)$-path $Q'$ of $G$ such that $Q'$ is a subdigraph of $P \cup Q$, and $P$ and $Q'$ are laced.
\end{lemma}

Later, we need a slightly different version of it.

\begin{lemma}\label{lem:laced2}
  Let $P$ be an $(a,b)$-path and $Q$ be a $(c,d)$-path in a digraph $G$ such that 
  $(V(P)\cap V(Q))\setminus \{a,d\}\neq\emptyset$, $a$ is not an internal vertex of $Q$, and $d$ is not an internal vertex of $P$.
  Then, there exists a $(c,d)$-path $Q'$ of $G$ such that 
  \begin{itemize}
      \item $Q'$ is a subdigraph of $P \cup Q$,
      \item $P$ and $Q'$ are laced, and 
      \item $(V(P)\cap V(Q'))\setminus \{a,d\}\neq\emptyset$.
  \end{itemize} 
\end{lemma}
\begin{proof}
    Let $P=p_1p_2 \cdots p_n$ and $Q=q_1q_2 \cdots q_m$ with $p_1=a$, $p_n=b$, $q_1=c$, and $q_m=d$.

    We prove the statement by induction on $\abs{V(P)}$.
    
    Let $i$ be the minimum integer in $[m]\setminus \{1\}$ such that $q_i\in V(P)$. 
    Assume that $q_i=p_1$.
    Since $p_1$ is not an internal vertex of $Q$, in this case, $q_i$ is an endpoint of $Q$.
    If $q_1=p_n$, then $P$ and $Q$ are laced and $(V(P)\cap V(Q))\setminus \{a,d\}=\{p_n\}\neq \emptyset$. 
    If $q_1\neq p_n$, then $(V(P)\cap V(Q))\setminus \{a,d\}= \emptyset$, contradicting the assumption. Thus, we may assume that $q_i\neq p_1$. 
    
    Let $j\in [n]$ be the integer such that $p_j=q_i$.
    Let $i'$ be the maximum integer in $[m]$ such that $q_{i'}\in V(p_jPp_m)$.
    Note that $i'<m$ as $d$ is not an internal vertex of $P$.
    Let $P^*=p_1Pp_j$ and $Q^*=q_{i'}Qq_m$. 

    Observe that $p_1$ is not an internal vertex of $Q$, so it is not an internal vertex of $Q^*$.
    Also, $q_m$ is not an internal vertex of $P'$.
    By construction, $q_{i'}$ is not an internal vertex of $P^*$, and $p_j$ is not an internal vertex of $Q^*$.

    Suppose that $(V(P^*)\cap V(Q^*))\setminus \{p_1, q_m\}=\emptyset$.
    In this case, let $Q'=q_1Qq_i\cup p_jPq_{i'}\cup q_{i'}Qq_m$. Then $Q'$ is a desired path. 

    So, we assume that $(V(P^*)\cap V(Q^*))\setminus \{p_1, q_m\}\neq \emptyset$.
    Then by the induction hypothesis, there exists a $(q_{i'}, q_m)$-path $(Q^*)'$ of $G$ such that 
    \begin{itemize}
      \item $(Q^*)'$ is a subdigraph of $P^* \cup Q^*$,
      \item $P^*$ and $(Q^*)'$ are laced, and 
      \item $(V(P^*)\cap V((Q^*)'))\setminus \{p_1, q_m\}\neq\emptyset$.
  \end{itemize}
  Let $Q'=q_1Qq_i \cup p_jPq_{i'} \cup (Q^*)'$.
  Then $Q'$ is a desired path.    
\end{proof}

\subsection{Butterfly minors}
For a digraph $G$, we say that an edge $(u,v)$ of $G$ is \emph{butterfly contractible} if $(u,v)$ is the only edge whose tail is $u$, or $(u,v)$ is the only edge whose head is $v$.
In this case, the butterfly contraction of $(u,v)$ in $G$ is the operation which identifies $u$ and $v$ into one vertex, say $x_{u,v}$, deletes $(u,v)$, replaces all edges $(z, u)$ with $(z, x_{u,v})$, replaces all edges $(v,z)$ with $(x_{u,v}, z)$, and then deletes all multiple edges.
We denote it by $G/(u,v)$.
We say that a digraph $H$ is a \emph{butterfly minor} of a digraph $G$ if $H$ can be obtained from $G$ by a sequence of deleting vertices, deleting edges, and butterfly contractions.

For two digraphs $H$ and $G$, a \emph{butterfly minor model} of $H$ in $G$ is a function $\mu$ with the domain $V(H)\cup E(H)$ such that 
\begin{itemize} 
    \item for every $v\in V(H)$, $\mu(v)$ is a subdigraph of $G$ with vertex partition $(\{r_v\}, I_v, O_v)$ such that 
    \begin{itemize}
        \item $\mu(v)$ is an orientation of a tree, 
        \item $\mu(v)[\{r_v\}\cup O_v]$ is an out-arborescence rooted in $r_v$, and 
        \item $\mu(v)[\{r_v\}\cup I_v]$ is an in-arborescence rooted in $r_v$, 
    \end{itemize}
    \item for two distinct $v,w\in V(H)$, $\mu(v)$ and $\mu(w)$ are vertex-disjoint,
    \item for every $(x,y)\in E(H)$, $\mu((x,y))$ is an edge in $G$ whose tail is in $\{r_x\}\cup O_x$ and head is in $\{r_y\}\cup I_y$.
\end{itemize}
We say that the collection $\{(\{r_v\}, I_v, O_v)\}_{v\in V(H)}$ is the \emph{decomposition of $\mu(H)$ into arborescences}.
For every subdigraph $H'$ of $H$, let $\mu(H')$ be the subdigraph of $G$ obtained from $\bigcup_{v\in V(H')}\mu(v)$ by adding edges in $\{\mu(e):e\in E(H')\}$.
Amiri et al.~\cite[Lemma 3.2]{AmiriKKW2016} showed that $H$ is a butterfly minor of $G$ if and only if there is a butterfly minor model of $H$ in $G$.

We say that a butterfly minor model $\mu$ of $H$ in $G$ is \emph{minimal} if there is no butterfly minor model $\mu'$ such that $\mu'(H)$ is a proper subdigraph of $\mu(H)$. 
We make use of the following observation.

\begin{lemma}\label{lem:minimalmodel}
    Let $H$ be a strongly connected digraph, and let $\mu$ be a minimal butterfly minor model of $H$ in a digraph $G$.
    Then $\mu(H)$ is strongly connected.
\end{lemma}
\begin{proof}
    Let $v,w$ be two vertices in $\mu(H)$. Let $x,y\in V(H)$ such that $v\in V(\mu(x))$ and $w\in V(\mu(y))$.
    As $\mu$ is a minimal model, there is a directed path in $\mu(x)$ from $v$ to the vertex incident with $\mu((x,x'))$ for some $x'$. Similarly, there is a directed path in $\mu(y)$ from the vertex incident with $\mu((y',y))$ for some $y'$ to $w$. Since $H$ is strongly connected, there is a directed path from $x'$ to $y'$ in $H$. Following the models of the vertices of this path, we can find a directed path from $\mu((x,x'))$ to $\mu((y',y))$ in $\mu(H)$. Thus, there is a directed path from $v$ to $w$ in $\mu(H)$.
\end{proof}

\subsection{Cycle rank}\label{subsec:cyclerank}
We define the \emph{cycle rank} of a digraph $G$, denoted by $\crank(G)$, as follows.
\begin{itemize}
    \item If $G$ is acyclic, then $\crank(G)=0$.
    \item If $G$ is strongly connected and $\abs{E(G)}\ge 1$, then $\crank(G)=1+\min_{v\in V(G)}(\crank(G-v))$.
    \item If $G$ is not strongly connected, then \[\crank(G)=\max\{\crank(H):\text{$H$ is a strongly connected component of $G$}\}.\]
\end{itemize}

The cycle rank can be equivalently defined via a decomposition scheme.
Let $G$ be a digraph and let $G_1,\ldots, G_n$ be all strongly connected components of $G$.
A \emph{cycle rank decomposition} of $G$ is a rooted forest $T$ such that
\begin{enumerate}[(1)]
    \item $T$ is the disjoint union of rooted trees in $\set{(T_i,r_i):i\in [n]}$ with $V(T_i)=V(G_i)$,
    \item for every internal node $t$ and its child $s$, $G[V(T_s)]$ is a strongly connected component of $G[V(T_t)]-t$.
\end{enumerate}

The \textit{height} of a cycle rank decomposition is the maximum height among all its strongly connected components. Note that McNaughton~\cite{MCNAUGHTON1969CC} defined a similar concept called an `analysis forest', while Gruber~\cite{Gruber2012} called it a `directed elimination forest'. In their definitions, each leaf node corresponds to a minimal strongly connected subdigraph with at least two vertices, which results in a one-level difference in height compared to our decomposition.
The following is directly obtained from the definition of cycle rank, and 
also follows from an observation for analysis forests by McNaughton~\cite{MCNAUGHTON1969CC}.

\begin{lemma}\label{lem:crdecomposition}
The cycle rank of a digraph $G$ is the minimum height minus one over all cycle rank decompositions of $G$.
\end{lemma} 

McNaughton~\cite[Theorem 2.4]{MCNAUGHTON1969CC}
observed that removing edges or vertices does not increase the cycle rank of a digraph.
McNaughton~\cite{McNaughton1967} also introduced a notion of a \emph{pathwise homomorphism} in connection with cycle rank, and proved that if there is a pathwise homomorphism from $G$ to $H$, then $\crank(H)\le \crank(G)$. From this, we can deduce that if $H$ is a butterfly minor of $G$, then $\crank(H)\le \crank(G)$. However, the proof in~\cite{McNaughton1967} is rather involved and we provide a direct proof here.  

\begin{lemma}\label{lem:butterflyminor}
    If a digraph $H$ is a butterfly minor of a digraph $G$, then $\crank(H)\le \crank(G)$.\footnote{In Section 6.6 of the book~\cite{NesetrilO2012}, the authors mentioned that `As noticed by Gelade~\cite{Gelade2010}, the cycle rank [...] is monotone by minor. [...] Here, minors of digraphs are interpreted in the weakest sense as minors of underlying multigraphs with proper orientation.' Gelade~\cite{Gelade2010} mentioned that `cycle rank does not increase when contracting any edge (due to McNaughton~\cite{McNaughton1967})'. But this is not true; for example, the digraph obtained from a directed cycle by reversing one edge has cycle rank $0$, but contracting this reversed edge results in a directed cycle, which has cycle rank $1$. We warn the reader that these references may easily be misunderstood.}
\end{lemma}
\begin{proof}
    It is straightforward to verify that removing edges or vertices does not increase the cycle rank of a digraph; see also McNaughton~\cite[Theorem 2.4]{MCNAUGHTON1969CC}. We show that for every butterfly contractible edge $e$ in a digraph $G$, we have $\crank(G/e)\le \crank(G)$.
    Let $k=\crank(G)$.
    
    Let $e=(u,v)$ be a butterfly contractible edge in $G$, and let $x_e$ be the contraction vertex in~$G/e$.
    Let $G_1, \ldots, G_n$ be the strongly connected components of $G$ and $T$ be a cycle rank decomposition of $G$ of height at most $k+1$ such that for each $i\in [n]$, $(T_i, r_i)$ is the component of $T$ with $V(G_i)=V(T_i)$.

    We say that two nodes in $T$ are \emph{comparable} if one of them is a descendant of the other in $T$.

    First, assume that $u$ and $v$ are not comparable.     
    Assume that $v$ has no in-neighbor in $G$ other than $u$. 
    Let $b$ be the node in the component $T'$ of $T$ containing $v$ such that 
    \begin{enumerate}[(1)]
        \item $T_b$ contains $v$ but does not contain $u$, and 
        \item subject to (1), $b$ is closest to the root of $T'$.
    \end{enumerate}
    As $u$ is the only in-neighbor of $v$, we have $V(T_b)=\{v\}$ because $G[V(T_b)]$ is strongly connected and $v\in V(T_b)$.
    We replace $u$ with $x_e$ and remove $v$ from $T$.
    It is not difficult to verify that the resulting forest is a cycle rank decomposition of $G/e$ of height at most $k+1$.
    A similar argument applies when $u$ has out-degree $1$ in~$G$.

    Now, assume that $u$ and $v$ are comparable.
    We assume that $u$ is a descendant of $v$.
    The case when $v$ is a descendant of $u$ is similar. 
    Let $a$ be the child of $v$ such that $T_a$ contains $u$. 
    Let $q$ be the height of the rooted tree $(T_a, a)$. Clearly, $G[V(T_a)]$ has cycle rank at most $q-1$.
    Let $C_1, C_2, \ldots, C_m$ be the strongly connected components of $G[V(T_a)]-u$.
    For each $i\in [m]$, the component $C_i$ is a subdigraph of $G[V(T_a)]$ and therefore has cycle rank at most $q-1$.
    Let $(F_i, s_i)$ be a cycle rank decomposition of $C_i$ of height at most $q$.

    From $T$, we remove $T_a$ and, for all $i\in [m]$, add $F_i$ and an edge between $s_i$ and $v$.
    Let $T^*$ be the resulting forest.
    We claim that $T^*$ is a cycle rank decomposition of $G/e$. 

    It is clear that $T^*$ has height at most $k+1$, as each $F_i$ has height at most $q$.
    Note that for every strongly connected digraph $H$ and a butterfly contractable edge $f$ in $H$, $H/f$ is strongly connected.  
    Thus, the first condition is satisfied.

    To check the second condition, let $t$ be an internal node in $T^*$ and $s$ be a child of $t$.
    If $t$ is some node in $\bigcup_{i\in [m]}F_i$, then it follows from the fact that $(F_i, s_i)$ is a cycle rank decomposition of $C_i$. If $t$ is a descendant of $v$ that is not in $\bigcup_{i\in [m]}F_i$, then $(T^*)_t$ is the same as $T_t$. Thus, it follows from the fact that $T$ is a cycle rank decomposition of $G$. If $t$ is a proper ancestor of $v$ and $T_s$ contains $u$, then $G[V(T_s)]$ is obtained from the strongly connected component of $G[V(T)]-t$ by contracting $(u,v)$. So, it is a strongly connected component of $G[V((T^*)_t)]-t$ containing $v$. If $s$ is another node, then it follows from $T$ being a cycle rank decomposition of $G$.
    So, we may assume that $t=v$.

    If $s=s_i$ for some $i\in [m]$, then 
    $(T^*)_s$ is a strongly connected component of $G[V(T_a)]-u$, where $T_a$ is a strongly connected component of $G[V(T_v)]-v$. Thus, $(T^*)_s$ is a strongly connected component of $G[(T^*)_v]-v$.
    If $s\notin \{s_i:i\in [m]\}$, then $(T^*)_s$ is a strongly connected component of $G[V(T_v)]-v$, which is also a strongly connected component of $G[V((T^*)_v)]-v$.

   Therefore, $T^*$ is a cycle rank decomposition of $G/e$ of height at most $k+1$ and $\crank(G/e)\le k$.
\end{proof}

\subsection{Directed treewidth}\label{sec:treewidth}

Given a digraph $G$ and sets $X,Y\subseteq V(G)$ we say that $X$ \emph{strongly guards} $Y$ in $G$ if $X$ contains a vertex of every directed cycle $C$ of $G$ with $V(C)\cap Y\neq\emptyset$ and $V(C)\cap (V(G)\setminus Y)\neq\emptyset$.
For a digraph~$G$, a tuple $(T,\beta,\gamma)$ of an out-arborescence $T$, a function $\beta\colon V(T)\to 2^{V(G)}$, and a function $\gamma\colon E(T)\to 2^{V(G)}$ is a \emph{directed tree decomposition} of $G$ if 
	\begin{enumerate}
        \item $\{\beta(t):t\in V(T)\}$ is a partition of $V(G)$ into non-empty sets, and
		\item for every $(d,t)\in E(T)$, $\gamma(d,t)$ strongly guards $\bigcup_{t'\in B}\beta(t')$, where $B$ is the set of nodes $b$ where there is a directed path from $t$ to $b$ in $T$.
		\end{enumerate}
The sets $\beta(t)$ are called the \emph{bags}, and the sets $\gamma(e)$ are called the \emph{guards}.

For every $t\in V(T)$ let $\Gamma(t)\coloneqq\beta(t)\cup\bigcup_{t\sim e}\gamma(e)$ where $t\sim e$ means that $t$ is an endpoint of $e$.
The \emph{width} of $(T,\beta,\gamma)$ is defined as
	\begin{align*}
		\mathsf{width}(T,\beta,\gamma)\coloneqq\max_{t\in V(T)}|\Gamma(t)|-1.
	\end{align*}
The \emph{directed treewidth} of $G$, denoted by $\mathsf{dtw}(G)$, is the minimum width over all directed tree decompositions for $G$.

The \emph{cylindrical grid} of \emph{order k} is the digraph obtained from the disjoint union of $k$ directed cycles $v^i_1v^i_2\dots v^i_{2k-1}v^i_{2k}v^i_1$, $i\in[k]$, by adding the path $v_j^1v_j^2\dots v^{k-1}_jv^k_j$ for every odd $j\in [2k]$ and the path $v_j^kv_j^{k-1}\dots v_j^2v_j^1$ for every even $j\in[2k]$. See \cref{fig:grid} for an illustration of the cylindrical grid of order $4$.

\begin{figure}[t]
    \centering
    \includegraphics[scale=0.9]{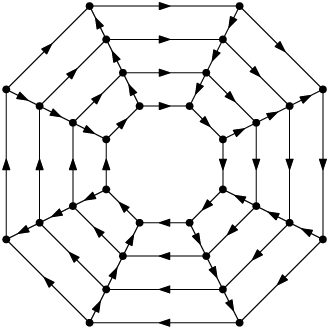}
    \caption{The cylindrical grid of order $4$}\label{fig:grid}
\end{figure}

 We use the directed grid theorem by 
Kawarabayashi and Kreutzer~\cite{kawarabayashi2015directed}. The required function in the directed grid theorem has been recently improved by Hatzel, Kreutzer, Milani, and Muzi~\cite{HatzelKMM2024}.

\begin{theorem}[Kawarabayashi and Kreutzer~\cite{kawarabayashi2015directed}]\label{thm:directedgrid}
There exists a function $f_{\mathsf{dtw}}\colon\mathds{N}\to\mathds{N}$ such that for every $k\in\mathds{N}$ every digraph $D$ with directed treewidth at least $f_{\mathsf{dtw}}(k)$ contains the cylindrical grid of order $k$ as a butterfly minor.
\end{theorem}

\section{Ladders, cycle chains, and tree chains}\label{sec:ladderchain}

Now we define the three types of butterfly minors for which every digraph of large cycle rank has to contain one of them.

For every positive integer $k$, let $\ladder_k$ be the digraph obtained from the disjoint union of two directed paths on $k$ vertices, say $p_1p_2 \cdots p_{k}$ and $q_1q_2 \cdots q_{k}$, by adding 
edges $(p_{i}, q_{k+1-i})$ and $(q_{k+1-i}, p_{i})$ for every $i\in [k]$.
We call $\ladder_k$ a \emph{(directed) ladder} of order $k$. 
For every positive integer $k$, let $\cycleCh_k$ be the digraph obtained from the undirected path $v_1v_2 \cdots v_k$ on $k$ vertices by replacing every edge with a directed cycle of length~$2$.
We call $\cycleCh_k$ a \emph{(directed) cycle chain} of order $k$ and refer to $v_1$ and $v_k$ as its endpoints. See \cref{fig:ladderandchain} for illustrations.

For a digraph $G$ and two vertices $v,w\in V(G)$, not necessarily distinct, we say that $(G, v, w)$ is a \emph{$2$-terminal digraph}. 
Now, we define the digraph $\treeCh_k$ as follows.
Let $(\treeCh_{0}, s_0, t_0)$ be a $2$-terminal digraph where $\treeCh_{0}$ is a digraph on the vertex set $\{v\}$ with $s_0=t_0=v$. For every positive integer $i$, let $(\treeCh_{i}, s_i, t_i)$ be the $2$-terminal digraph obtained from the disjoint union of two copies $(B^1_{i-1}, s^1_{i-1}, t^1_{i-1})$ and $(B^2_{i-1}, s^2_{i-1}, t^2_{i-1})$ of $(\treeCh_{i-1}, s_{i-1}, t_{i-1})$ 
by adding edges $(s^1_{i-1}, t^2_{i-1})$ and $(s^2_{i-1}, t^1_{i-1})$,
and assigning $s_i\coloneqq s^1_{i-1}$ and $t_i\coloneqq t^2_{i-1}$.
We call $\treeCh_k$ a \emph{tree chain} of order $k$.
See \cref{fig:ladderandchain} for an illustration of $\treeCh_3$.

We argue that ladders, cycle chains, and tree chains have large cycle rank. 
McNaughton~\cite[Section 4]{MCNAUGHTON1969CC} showed that $\cycleCh_k$ has cycle rank exactly $\lfloor\log k\rfloor$.
\begin{theorem}[McNaughton~\cite{MCNAUGHTON1969CC}]\label{thm:crankcc} 
    For every positive integer $k$, $\crank(\cycleCh_k)=\lfloor\log k\rfloor$.
\end{theorem}

We show that $\ladder_k$ has cycle rank at least $\lfloor\log k\rfloor+1$.

\begin{lemma}\label{lem:Lk}
    For every positive integer $k$, $\crank(\ladder_k)\ge \lfloor\log k\rfloor+1$.
\end{lemma}
\begin{proof}
    If $k=1$, then $\ladder_1$ is a directed cycle of length $2$ and we have $\crank(\ladder_1)=1$.
    Assume that $k>1$. 
    
    Note that for every vertex $v$ of $\ladder_k$, $\ladder_k-v$ has $\ladder_{\lfloor k/2 \rfloor}$ as a subdigraph.
    So, we have
    \[\crank(\ladder_k)=1+\min_{v\in V(\ladder_k)}(\crank(\ladder_k-v))\ge 1+\crank(\ladder_{\lfloor k/2\rfloor}).\]
    This inequality holds for every $k>1$. As $\crank(\ladder_1)=1$, we have $\crank(\ladder_k)\ge \lfloor\log k\rfloor+1$.
\end{proof}

Next, we prove that $\treeCh_k$ has cycle rank at least $k$. 

\begin{lemma}\label{lem:Fk}
    For every positive integer $k$, $\crank(\treeCh_k)\ge k$.
\end{lemma}
\begin{proof}
When $k=1$, $\treeCh_k$ is a cycle of length $2$ and $\crank(\treeCh_k)=1$.
We assume that $k\ge 2$.

Note that $\treeCh_k$ is strongly connected.
Observe that for every vertex $v$ in $\treeCh_k$, $\treeCh_k-v$ contains $\treeCh_{k-1}$ as a subdigraph.
By induction, we have that
\begin{equation*}
    \crank(\treeCh_k)\ge \min_{v\in V(G)}(\crank(\treeCh_k-v))+1
    \ge \crank(\treeCh_{k-1})+1\ge (k-1)+1=k. \qedhere
\end{equation*}
\end{proof}

In the next lemmas, we show that every large cylindrical grid contains a cycle chain and a ladder as butterfly minors.
On the other hand, we prove in \cref{sec:independent} that $\treeCh_k$ is not a butterfly minor of a cylindrical grid.

\begin{lemma}\label{lem:chainfromgrid}
    Let $k$ be a positive integer. Every cylindrical grid of order $k$ contains a cycle chain of order $k+1$ as a butterfly minor.
\end{lemma}
\begin{proof}
    Let $G$ be a cylindrical grid of order $k$ whose vertex set and edge set are given as in the definition. 

    If $k=1$, then $G$ is a directed cycle of length $2$. So, it is a cycle chain of order $2$.
    Assume $k=2$. Let $F$ be the subdigraph of $G$ obtained from 
    the disjoint union of
    $v^1_1v^1_2v^1_3v^1_4v^1_1$ and $v^2_1v^2_2v^2_3v^2_4v^2_1$
    by removing $(v^2_2, v^2_3)$ and adding $(v^2_2, v^1_2), (v^1_3, v^2_3)$. Then the digraph obtained from $F$ by butterfly contracting $(v^1_2, v^1_3)$ and 
    recursively butterfly contracting all remaining edges contained in $v^1_1v^1_2v^1_3v^1_4v^1_1$ and $v^2_1v^2_2v^2_3v^2_4v^2_1$ except $(v^i_1, v^i_2)$ for each $i\in [2]$ is a cycle chain of order $3$.

     So, we may assume that $k\ge 3$.
    Let $H$ be the subdigraph of $G$ obtained from the union of the $k$ directed cycles $v^i_1v^i_2\ldots v^i_{2k}v^i_1$ for $i\in [k]$ by 
    \begin{itemize}
        \item removing $(v^i_2, v^i_3)$ for all even $i\in [k]$, and removing $(v^i_4, v^i_5)$ for all odd $i\in [k]\setminus \{1\}$, and
        \item adding $(v^{i+1}_2, v^i_2), (v^i_3, v^{i+1}_3)$ for all odd $i\in [k-1]$, and adding $(v^{i+1}_4, v^i_4), (v^i_5, v^{i+1}_5)$ for all even $i\in [k-1]$. 
    \end{itemize}
    See the second digraph in \cref{fig:Chainfromgrid.eps} for an illustration of $H$.

    Now, we obtain a digraph $H^{*}$ from $H$ by 
   \begin{itemize}
       \item butterfly contracting $(v^i_2, v^i_3)$ for all odd $i\in [k-1]$, and butterfly contracting $(v^i_4, v^i_5)$ for all even $i\in [k-1]$, and 
       \item recursively butterfly contracting all remaining edges in $H$ except $(v^i_1,v^i_2), (v^i_3,v^i_4)$ for $i\in [k]$.
   \end{itemize}
    Then $H^{*}$ is a cycle chain of order $k+1$, where the two edges between $i$-th vertex and $(i+1)$-vertex are $(v^{i+1}_2, v^i_2), (v^i_3, v^{i+1}_3)$ for all odd $i\in [k-1]$ and 
    $(v^{i+1}_4, v^i_4), (v^i_5, v^{i+1}_5)$ for all even $i\in [k-1]$.
\end{proof}

\begin{figure}
    \centering
    \includegraphics[scale=0.75]{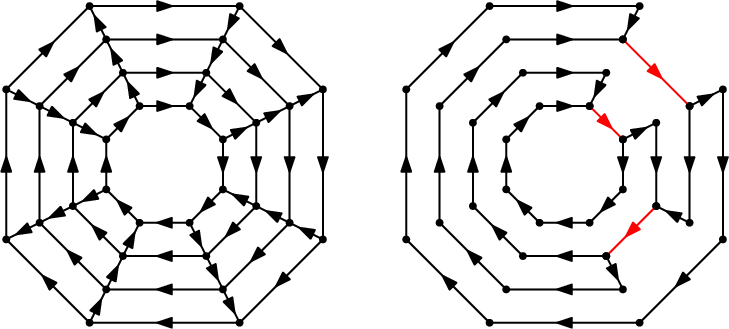}
    \caption{To obtain a cycle chain as a butterfly minor, we first take a subdigraph as right.
    Every red edge in the right digraph is butterfly contractible. After contracting all red edges, the remaining subdigraph has a cycle chain of order $5$ as a butterfly minor.}
    \label{fig:Chainfromgrid.eps}
\end{figure}

\begin{lemma}\label{lem:ladderfromgrid}
    Every cylindrical grid of order $2k$ contains a directed ladder of order $k$ as a butterfly minor.
\end{lemma}
\begin{proof}
    Let $G$ be a cylindrical grid of order $2k$ whose vertex set and edge set are given by the definition.
    Let $H$ be the digraph obtained from $G$ by 
    \begin{itemize}
        \item removing $(v^i_j, v^{i+1}_j)$ for all $i\in [2k-1]$ and $j\in [4k-1]\setminus\{1\}$,
        \item removing $(v^i_j, v^i_{j+1})$ for all odd $i\in [2k]$ and $j\in [4k-1]$, and
        \item removing $(v^i_{4k}, v^i_{1})$ for all even $i\in [2k]$.
    \end{itemize}
   Let $H^*$ be the digraph obtained from $H$ by
    \begin{itemize}
        \item butterfly contracting $(v^i_j, v^i_{j+1})$ for all even $i\in [2k]$ and $j\in [4k-2]$,
        \item butterfly contracting $(v^i_1, v^{i+1}_1)$ for all odd $i\in [2k]$, and
        \item butterfly contracting $(v^{i+1}_{4k}, v^{i}_{4k})$ for all odd $i\in [2k]$.
    \end{itemize}
    Then $H^*$ is a ladder of order $k$.
\end{proof}

\section{Relaxed ladders, relaxed chains, mixed chains and relaxed tree chains}\label{sec:relaxedversion}

In this section, we introduce several intermediate structures that we use: \textsl{relaxed ladders}, \textsl{relaxed chains}, \textsl{mixed chains} and \textsl{relaxed tree chains}.
Roughly, these relate to chains, ladders and tree chains as follows.
A \textsl{relaxed ladder} is a slight generalization of a ladder, while a \textsl{relaxed chain} arises naturally from two laced paths. A \textsl{mixed chain} is formed by combining relaxed ladders and relaxed chains, and serves as a central concept in our chain decomposition framework. A \textsl{relaxed tree chain} is a minor variation of a tree chain.
We formally define these structures and provide some basic properties.

\paragraph{Relaxed ladders.}
Let $p,q,p',q'$ be vertices of a digraph $G$ such that $\{p,q\}\cap \{p',q'\}=\emptyset$.
For a non-negative integer $k$, 
a subdigraph $H$ of a digraph $G$ is a \emph{relaxed ladder of order $k$} in $G$ with the pair $(p,q)$ of left endpoints and the pair $(p',q')$ of right endpoints if there is a tuple 
\[\left(P, Q, (P_i:i\in [k]), (Q_i:i\in [k]), (X_i:i\in [k]), (Y_i:i\in [k])\right)\]
such that 
\begin{itemize}
    \item $P$ is a $(p,p')$-path in $H$, $Q$ is a $(q',q)$-path in $H$, and 
    $P$ and $Q$ are internally vertex-disjoint,
    \item $P_1, \ldots, P_k$ are internally vertex-disjoint subpaths of $P$ of length at least $1$ that appear in this order in~$P$,
    \item $Q_1, \ldots, Q_k$ are internally vertex-disjoint subpaths of $Q$ of length at least $1$ that appear in this order in~$Q$, 
    \item for each $i\in [k]$, $X_i$ is a $(P, Q)$-path whose endpoints are in $P_i\cup Q_{k+1-i}$ and $Y_i$ is a $(Q,P)$-path whose endpoints are in $P_i\cup Q_{k+1-i}$ such that $X_i$ and $Y_i$ are laced, and
    \item for distinct $i_1, i_2\in [k]$, $(X_{i_1}\cup Y_{i_1})-(V(P)\cup V(Q))$ and $(X_{i_2}\cup Y_{i_2})-(V(P)\cup V(Q))$ are vertex-disjoint, 
    \item $p\notin V(X_1)$, $q\notin V(Y_1)$, $p'\notin V(X_k)$, and $q'\notin V(Y_k)$, and
        \item $H=P\cup Q\cup \left(\bigcup_{i\in [k]}(X_i\cup Y_i)\right)$.
\end{itemize}
Observe that for $i_1, i_2\in [k]$ with $\abs{i_2-i_1}\ge 2$, $X_{i_1}\cup Y_{i_1}$ and $X_{i_2}\cup Y_{i_2}$ are vertex-disjoint, because $P_{i_1}\cup Q_{k+1-i_1}$ and $P_{i_2}\cup Q_{k+1-i_2}$ are vertex-disjoint.
We call each $X_i\cup Y_i$ a \emph{rung} of $H$, and call the vertices in $V(X_i\cup Y_i)\cap V(P\cup Q)$ the \emph{endpoints} of $X_i\cup Y_i$. 
We call $(P,Q)$ the pair of \emph{boundary-paths}.
We denote the order of $H$ as $\ord(H)$.
See \cref{fig:ladderchain} for an illustration.
Note that a relaxed ladder of order $0$ is simply the union of two internally vertex-disjoint paths $P$ and $Q$ of length at least $1$.

\begin{figure}
    \centering
    \begin{subfigure}[b]{0.52\textwidth}
        \centering
         \resizebox{\textwidth}{!}{%
         \includegraphics[width=\textwidth]{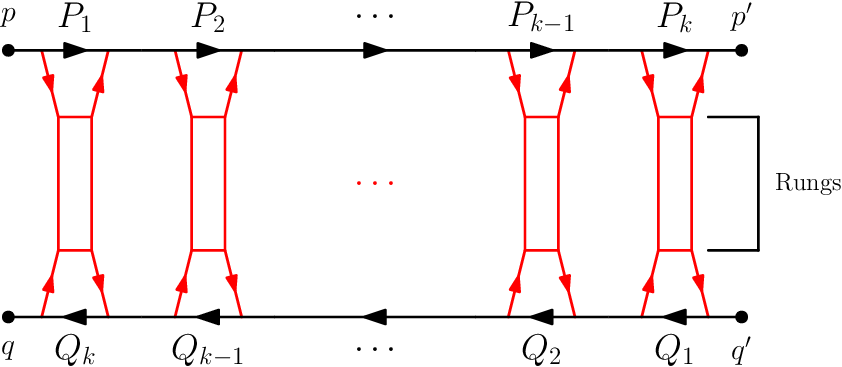}
         }
        \subcaption{A relaxed ladder of order $k$.}
    \end{subfigure}
    \hfill
    \begin{subfigure}[b]{0.47\textwidth}
        \centering
         \resizebox{\textwidth}{!}{%
        \includegraphics[width=\textwidth]{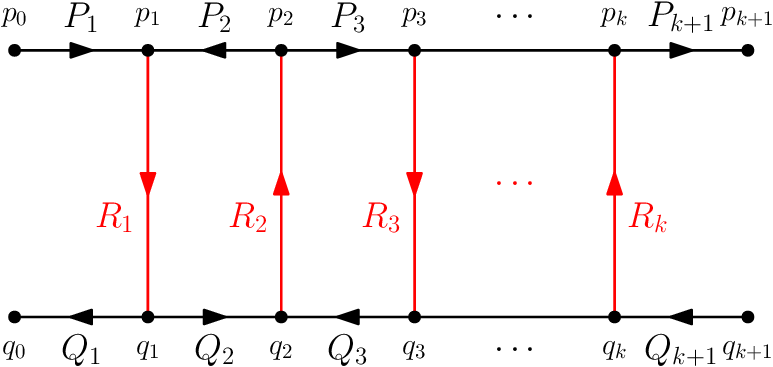}
        }
        \subcaption{A relaxed chain of length $k$ when $k$ is even.}
    \end{subfigure}
    \caption{Notice that each $R_i$ in a relaxed chain might be of length $0$.
    In both constructions, it is possible that the left endpoints are the same, and the right endpoints are the same.}
    \label{fig:ladderchain}
\end{figure}

\begin{lemma}\label{lem:relaxedladder}
    Let $k$ be a positive integer. A relaxed ladder of order $4k$ in a digraph contains a ladder of order $k$ as a butterfly minor.
\end{lemma}
\begin{proof}
    Let $H$ be a relaxed ladder of order $4k$ with tuple 
    \[\left(P, Q, (P_i:i\in [4k]), (Q_i:i\in [4k]), (X_i:i\in [4k]), (Y_i:i\in [4k])\right).\]
    We take a subdigraph \[H'=P\cup Q\cup \left(\bigcup_{i\in [k]}(Y_{4i-3}\cup X_{4i-1})\right).\]
    Note that for each $i\in [k]$, the edges on $P$ between the endpoints of $Y_{4i-3}$ and $X_{4i-1}$ are butterfly contractible, and the edges on $Q$ between the endpoints of  $Y_{4i-3}$ and $X_{4i-1}$ are butterfly contractible. 
    We first contract all these edges. Then in the resulting digraph, the set of endpoints of $Y_{4i-3}$ is the same as the set of endpoints of $X_{4i-1}$.    
    Now, by recursively butterfly contracting all possible edges, we get a ladder of order $k$.
\end{proof}

\paragraph{Relaxed chains.} Let $p,q,p',q'$ be vertices of a digraph $G$ such that $\{p,q\}\cap \{p',q'\}=\emptyset$.
For a non-negative integer~$k$, a subdigraph $H$ of $G$ is a \emph{relaxed chain of length $k$} in $G$ with the pair $(p,q)$ of left endpoints and the pair $(p',q')$ of right endpoints if there is a tuple 
\[\left((P_i:i\in [k+1]), (Q_i:i\in [k+1]), (R_i:i\in [k]) \right)\] 
and there are vertices $p_0, p_1, \ldots, p_{k+1}$ and $q_0, q_1, \ldots, q_{k+1}$ of $G$
such that 
\begin{itemize}
     \item for distinct $i,j\in \{0\}\cup [k+1]$, $\{p_i, q_i\}\cap \{p_j,q_j\}=\emptyset$,
        \item for each odd $i\in[k+1]$, $P_i$ is a directed path from $p_{i-1}$ to $p_i$ and for every even $i\in[k+1]$, $P_i$ is a directed path from $p_i$ to $p_{i-1}$,
        \item for each odd $i\in [k+1]$, $Q_i$ is a directed path from $q_i$ to $q_{i-1}$, and for each even $i\in [k+1]$, $Q_i$ is a directed path from $q_{i-1}$ to $q_i$, 
        \item for each odd $i\in [k]$, $R_i$ is a directed path from $p_i$ to $q_i$, and for each even $i\in [k]$, $R_i$ is a directed path from $q_i$ to $p_i$,
    \item $P_1, \ldots, P_{k+1}, Q_1, \ldots, Q_{k+1}, R_1, \ldots, R_k$ are pairwise internally vertex-disjoint, 
    \item $(p,q)=(p_0,q_0)$, and if $k$ is odd, then $(p',q')=(p_{k+1},q_{k+1})$, and if $k$ is even, then $(p',q')=(q_{k+1},p_{k+1})$, and
    \item $H=\left(\bigcup_{i\in [k+1]}(P_i\cup Q_i)\right)\cup \left(\bigcup_{i\in [k]}R_i\right)$.
\end{itemize}
Note that $p$ and $p'$ are vertices of out-degree $1$ in $H$, and $q$ and $q'$ are vertices of in-degree $1$ in $H$.
See \cref{fig:ladderchain} for an illustration. A relaxed chain is \emph{closed} if $p=q$ and $p'=q'$.

One can observe that a relaxed chain $H$ of length $k$ consists of two directed paths $P_1\cup R_1\cup Q_2\cup \cdots \cup R_k\cup P_{k+1}$ and $Q_{k+1}\cup R_k\cup P_{k}\cup \cdots \cup Q_1$ that are laced if $k$ is even.
If $k$ is odd, the two directed paths $P_1\cup R_1\cup Q_2\cup \cdots \cup R_{k}\cup Q_{k+1}$ and $P_{k+1}\cup R_k\cup Q_{k}\cup \cdots \cup Q_1$ are laced.
Also, in the other direction, two directed paths $P$ and $Q$ that are laced form a relaxed chain $P \cup Q$. We prove this in the next lemma.

\begin{lemma}\label{lem:lacedchain}
    Let $G$ be a digraph and $v_1, v_2, w_1, w_2\in V(G)$ with $\{v_1, w_1\}\cap \{v_2, w_2\}=\emptyset$.
    Let $P$ be a $(v_1, w_2)$-path and $Q$ be a $(v_2, w_1)$-path such that 
    \begin{itemize}
        \item $P$ and $Q$ are laced, 
        \item $(P\cap Q)-\{v_1, v_2, w_1, w_2\}$ consists of $k$ weakly connected components.
    \end{itemize}  
    Then $P\cup Q$ is a relaxed chain of length $k$ where $(v_1, w_1)$ is a pair of left endpoints and $(v_2, w_2)$ is a pair of right endpoints.
\end{lemma}
\begin{proof}
    If $k=0$, then $P\cup Q$ is a relaxed chain of length $0$ where $(v_1, w_1)$ is a pair of left endpoints and $(v_2, w_2)$ is a pair of right endpoints. We assume that $k>0$.

    Assume $k$ is even. When $k$ is odd, we can prove similarly.
    Let $p_1, p_2, \ldots, p_k, q_1, q_2, \ldots, q_k$ be vertices in $P\cap Q$ such that 
    \begin{itemize}
        \item for each odd $i\in [k]$, $p_iPq_i$ is a weakly connected component of  $(P\cap Q)-\{v_1, v_2, w_1, w_2\}$, 
        \item for each even $i\in [k]$, $q_iPp_i$ is a weakly connected component of  $(P\cap Q)-\{v_1, v_2, w_1, w_2\}$, and
        \item $p_1Pq_1, q_2Pp_2, \ldots, p_{k-1}Pq_{k-1}, q_kPp_k$ appear in this order in $P$.
    \end{itemize}
     We define $P_i$, $Q_i$ and $R_i$ as follows.
    \begin{itemize}
        \item $P_1\coloneqq v_1 P p_1$, $Q_1\coloneqq q_1 Q w_1$, $P_{k+1}\coloneqq p_k P w_2$ and $Q_{k+1}\coloneqq v_2 Q q_k$.
        \item For each $i\in [k-1]$, $P_{i+1}\coloneqq p_{i} P p_{i+1}$ if $i$ is even, and $P_{i+1}\coloneqq p_{i+1} Q p_i$ if $i$ is odd.
        \item For each $i\in [k-1]$, $Q_{i+1}\coloneqq q_{i} Q q_{i+1}$ if $i$ is even, and $Q_{i+1}\coloneqq q_{i+1} Q q_i$ if $i$ is odd.
        \item For each $i\in [k]$, $R_i\coloneqq  p_i P q_i$ if $i$ is odd, and $R_i\coloneqq  q_i P p_i$ if $i$ is even.
    \end{itemize}
    Then $P\cup Q$ is a relaxed chain of length $k$ with tuple \[\left( (P_i:i\in [k+1]), (Q_i:i\in [k+1]), (R_i:i\in [k]) \right)\] where $(v_1, w_1)$ is a pair of left endpoints and $(v_2, w_2)$ is a pair of right endpoints.
\end{proof}

The following lemma shows that contracting the paths $R_i$ of a relaxed chain yields a cycle chain, and thus one of our obstructions.

\begin{lemma}\label{lem:relaxedchain}
    Let $k$ be a positive integer.
    Every relaxed chain of length $k$ contains a cycle chain of order $k$ as a butterfly minor.
\end{lemma}
\begin{proof}
    Let $H$ be a relaxed chain of length $k$ with tuple \[\left( (P_i:i\in [k+1]), (Q_i:i\in [k+1]), (R_i:i\in [k]) \right).\]
    If we first butterfly contract all edges in every $R_i$, then $p_i$ and $q_i$ are identified for all $i\in [k]$.
    Now, by butterfly contracting all edges except one edge of each path in $\{P_i:i\in [k]\setminus \{1\}\}\cup \{Q_i:i\in [k]\setminus \{1\}\}$, we get a cycle chain of order $k$.
\end{proof}

\begin{figure}
    \centering
    \includegraphics[scale=0.6]{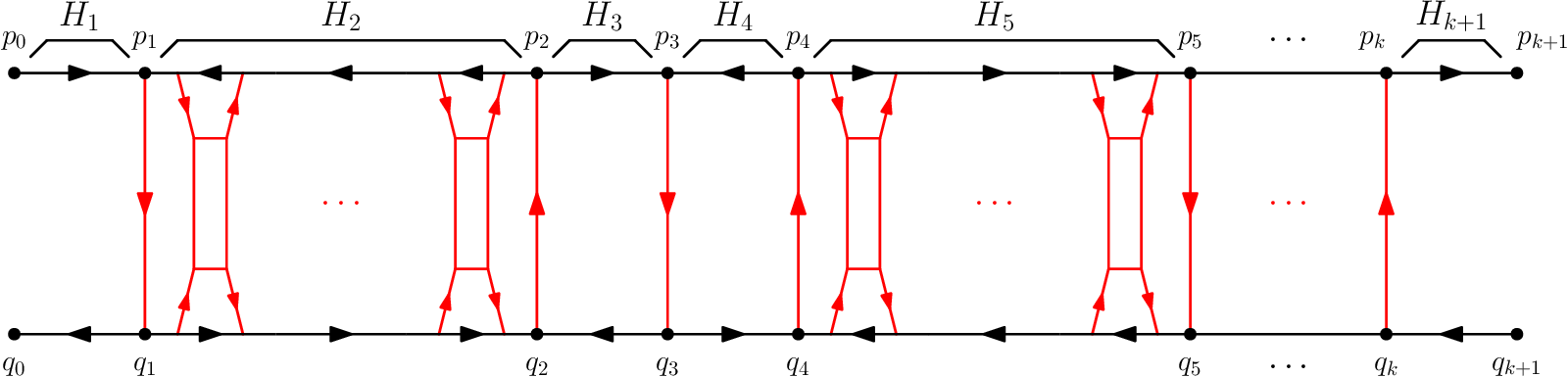}
    \caption{An example of a mixed chain of length $k$. It is possible that the left endpoints are the same, and the right endpoints are the same. }\label{fig:mixedchain}
\end{figure}

\paragraph{Mixed chains.}
Let $p,q,p',q'$ be vertices of a digraph $G$ such that $\{p,q\}\cap \{p',q'\}=\emptyset$.
For a non-negative integer $k$, a subdigraph $H$ of a digraph $G$ is a \emph{mixed chain of length $k$} in $G$ with the pair $(p,q)$ of left endpoints and the pair $(p',q')$ of right endpoints if there is a tuple 
\[ \left((H_i:i\in [k+1]), (R_i:i\in [k]) \right)\] 
and there are vertices $p_0, p_1, \ldots, p_{k+1}$ and $q_0, q_1, \ldots, q_{k+1}$ of $G$
such that 
\begin{itemize}
    \item for distinct $i,j\in \{0\}\cup [k+1]$, $\{p_i, q_i\}\cap \{p_j,q_j\}=\emptyset$,
    \item for each odd $i\in [k+1]$, $H_i$ is a relaxed ladder
    where $(p_{i-1}, q_{i-1})$ is the pair of left endpoints and $(p_i, q_i)$ is the pair of right endpoints,
    \item for each even $i\in [k+1]$, $H_i$ is a relaxed ladder 
    where $(p_i, q_i)$ is the pair of left endpoints and $(p_{i-1}, q_{i-1})$ is the pair of right endpoints,
    \item for each odd $i\in [k]$, $R_i$ is a directed path from $p_i$ to $q_i$, and for each even $i\in [k]$, $R_i$ is a directed path from $q_i$ to $p_i$,
    \item for all $i, j\in [k+1]$ with $i<j$, if $j-i\ge 2$, then $V(H_i)\cap V(H_j)=\emptyset$, and if $j-i=1$, then $V(H_i)\cap V(H_j)=\{p_i,q_i\}$ and 
    \item for all $i\in [k+1]$ and $j\in [k]$, 
    \begin{itemize}
        \item if $i=j$, then  $V(H_i)\cap V(R_j)=\{p_i,q_i\}$,
        \item if $i=j+1$, then $V(H_i)\cap V(R_j)=\{p_{i-1},q_{i-1}\}$, 
        \item otherwise, $V(H_i)\cap V(R_j)=\emptyset$.
    \end{itemize}
    \item $(p,q)=(p_0,q_0)$, and if $k$ is odd, then $(p',q')=(p_{k+1},q_{k+1})$, and if $k$ is even, then $(p',q')=(q_{k+1},p_{k+1})$, and
    \item $H=\left(\bigcup_{i\in [k+1]}H_i\right)\cup \left(\bigcup_{i\in [k]}R_i\right)$.
    \end{itemize}
    A mixed chain is \emph{closed} if $p=q$ and $p'=q'$.
The \emph{weight} of $H$ is defined as 
\[\weight(H)\coloneqq k+\sum_{i\in [k+1]}\ord(H_i).\]

Note that a mixed chain $H$ of length $k$ is a digraph obtained from a relaxed chain of length $k$ with the tuple
\[\left( (P_i:i\in [k+1]), (Q_i:i\in [k+1]), (R_i:i\in [k]) \right)\]
by replacing each $P_i\cup Q_i$ with some relaxed ladder $H_i$ for $i\in [k+1]$. 
We write $\partial(H)$ for the digraph obtained from $H$ by removing the rungs (except their endpoints) of every $H_i$.
Note that $\partial(H)$ is a relaxed chain. 

We prove that every mixed chain of large weight contains a cycle chain of order $k$ or a ladder of order $k$ as a butterfly minor.

\begin{lemma}\label{lemma:ChainsandLaddersinaMix}
    Let $k$ and $w$ be positive integers with $w\ge 4k^2+k-1$.
    Every mixed chain of weight~$w$ in a digraph contains a cycle chain of order $k$ or a ladder of order $k$ as a butterfly minor.
\end{lemma}
\begin{proof}
    Let $H$ be a mixed chain with the tuple 
    \[ \left( (H_i:i\in [x+1]), (R_i:i\in [x]) \right)\] that has weight $w$.

    Suppose that $x\ge k$. 
    Then $\partial(H)$ is a relaxed chain of length at least $k$. Thus, by \cref{lem:relaxedchain}, $\partial(H)$ contains a cycle chain of order at least $k$ as a butterfly minor.

    Therefore, we may assume that $x<k$.
    Since $w\ge 4k^2+k-1$,
    \[\frac{w-x}{x+1}\ge \frac{w-(k-1)}{k}\ge 4k\] and therefore, there is an $H_i$ that is a relaxed ladder of order at least $4k$.
    By \cref{lem:relaxedladder}, $H$ contains a ladder of order at least $k$ as a butterfly minor.
\end{proof}

\begin{figure}[!ht]
        \centering
        \includegraphics[width=0.8\linewidth]{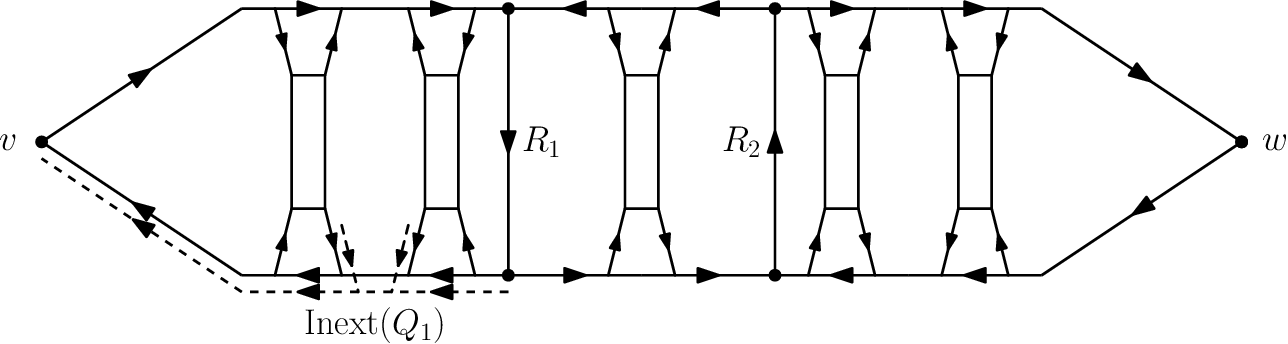}
        \caption{The subdigraph $\inext(Q_1)$ is depicted by dashed lines.  }\label{fig:exinext}
    \end{figure}
    
Let $H$ be a mixed chain with the tuple $\left( (H_i:i\in [x+1]), (R_i:i\in [x]) \right)$ and  
     for each $i\in [k+1]$, let $(P_i, Q_i)$ be the pair of boundary-paths of $H_i$.
     For the first and last relaxed ladders $H_1$ and $H_{x+1}$, it is convenient to define certain extensions of their boundary-paths by adding some parts of the rungs.
     We illustrate one such set $\inext(Q_1)$ in~\cref{fig:exinext}.
    Observe that for every vertex $z\in V(H)\setminus (V(\inext(Q_1))\cup V(R_1))$, there is a path from $z$ to $v$ fully containing~$R_1$. We use this observation in~\cref{sec:mixedchain}.
We formally define these concepts below.

Let $F$ be a relaxed ladder with the tuple
\[\left(P, Q, (P_i:i\in [k]), (Q_i:i\in [k]), (X_i:i\in [k]), (Y_i:i\in [k])\right).\]
As $X_i$ and $Y_i$ are laced by definition, 
$X_i\cup Y_i$ is a relaxed chain. 
For each $j\in [k]$, 
let 
\[\left((P_i^j:i\in [k_j+1]), (Q_i^j:i\in [k_j+1]), (R_i^j:i\in [k_j]) \right)\] 
be the tuple for $X_j\cup Y_j$ 
where $(V(X_j)\cup V(Y_j))\cap V(P_j)$ consists of the left endpoints and $(V(X_j)\cup V(Y_j)) \cap V(Q_j)$ consists of the right endpoints.
For each $j\in [k]$, let $Z_{\mathrm{in}}^j$ and $Z_{\mathrm{out}}^j$ be the paths in $\{P_{k_j+1}^j, Q_{k_j+1}^j\}$ whose head and tail are contained in $Q_j$, respectively. 
Now we define that 
\begin{itemize}
    \item $\inext_F(P)=P\cup \left( \bigcup_{j\in [k]}\left(Q_1^j-\tail(Q_1^j)\right) \right)$,
    \item $\outext_F(P)=P\cup \left( \bigcup_{j\in [k]}\left(P_1^j-\head(P_1^j)\right) \right)$,
    \item $\inext_F(Q)=Q\cup \left( \bigcup_{j\in [k]}\left(Z_{\mathrm{in}}^j-\tail(Z_{\mathrm{in}}^j)\right) \right)$, and
    \item $\outext_F(Q)=Q\cup \left( \bigcup_{j\in [k]}\left(Z_{\mathrm{out}}^j-\head(Z_{\mathrm{out}}^j)\right) \right)$. 
\end{itemize}
We often drop it as a subscript if $F$ is clear from the context.

\paragraph{Mixed extensions.}
We additionally define mixed extensions, which are used to extend mixed chains. Let $p,q,x,y$ be vertices of a digraph $G$ such that $\{p,q\}\cap \{x,y\}=\emptyset$.
A \emph{mixed extension} in a digraph $G$  with the pair $(p,q)$ of left endpoints and the pair $(y,x)$ of right endpoints is a subdigraph~$W$ of $G$ that is one of the following.
 \begin{itemize}
     \item $W=A\cup B$ where $A$ is a $(p,x)$-path, $B$ is a $(y,q)$-path, and $A$ and $B$ are laced and intersect within $(A\cup B)-\{p,q\}$.
     \item $W=A\cup B\cup C\cup D$ where
         \begin{itemize}
             \item $A$ is a $(p,x)$-path in $G$, $B$ is a $(y,q)$-path in $G$, and $A-\{p\}$ and $B-\{q\}$ are vertex-disjoint, and
             \item $C$ is an $(A,B)$-path, $D$ is a $(B,A)$-path in $G-\{p,q\}$, and $C$ and $D$ are laced.
         \end{itemize} 
 \end{itemize}

Note that a mixed extension is not symmetric with respect to the left and right endpoints.
The following lemma establishes how to use mixed extensions to extend mixed chains. The right endpoints of a mixed extension are identified with endpoints of a given mixed chain, and the left endpoints of it become a new set of endpoints in the resulting mixed chain.

\begin{lemma}\label{lem:extension}
    Let $G$ be a digraph and let $p,q,p',q',p'',q''$ be vertices in $G$ such that $\{p,q\}\cap \{p',q',p'',q''\}=\emptyset$ and $\{p',q'\}\cap \{p'',q''\}=\emptyset$.
    Let $H$ be a mixed chain in $G$ where $(p',q')$ is the set of left endpoints and $(p'',q'')$ is the set of right endpoints, and $W$ be a mixed extension in $G$ where $(p,q)$ is the set of left endpoints and $(q',p')$ is the set of right endpoints, such that $V(W)\cap V(H)=\{p',q'\}$. Then $W\cup H$ is a mixed chain where $(p,q)$ is the set of left endpoints, $(p'',q'')$ is the set of right endpoints, and $\weight(W\cup H)>\weight(H)$. 
\end{lemma}
\begin{proof}
    First assume that $W=A\cup B$ where $A$ is a $(p,p')$-path, $B$ is a $(q',q)$-path, and $A$ and $B$ are laced and intersect within $(A\cup B)-\{p,q\}$. Then $W\cup \partial(H)$ is a relaxed chain whose length is greater than that of $\partial(H)$.
    Thus, $W\cup H$ is a mixed chain whose weight is greater than that of $H$. Also, $(p,q)$ is the set of left endpoints and $(p'',q'')$ is the set of right endpoints, as desired.

    Now, assume that $W=A\cup B\cup C\cup D$ where
    \begin{itemize}
         \item $A$ is a $(p,p')$-path in $G$, $B$ is a $(q',q)$-path in $G$, and $A-\{p\}$ and $B-\{q\}$ are vertex-disjoint, and
         \item $C$ is an $(A,B)$-path, $D$ is a $(B,A)$-path in $G-\{p,q\}$, and $C$ and $D$ are laced.
     \end{itemize} 
     In this case, $p'\neq q'$. Let $H_1$ be the relaxed ladder in $H$ containing $p',q'$, as in the definition of a mixed chain, and let \[\left(P, Q, (P_i:i\in [k]), (Q_i:i\in [k]), (X_i:i\in [k]), (Y_i:i\in [k])\right)\] be the tuple for $H_1$.
     Let $P_0=A$, $Q_0=B$, $X_0=C$, and $Y_0=D$.
     Then 
     $H_1\cup W$ is a relaxed ladder whose tuple is
     \[\left(A\cup P, B\cup Q, (P_i:i\in \{0\}\cup [k]), (Q_i:i\in \{0\}\cup [k]), (X_i:i\in \{0\}\cup [k]), (Y_i:i\in \{0\}\cup [k])\right).\] Note that $W\cup H_1$ has length greater than that of $H_1$. Therefore, $W\cup H$ is a mixed chain where $(p,q)$ is the set of left endpoints, $(p'',q'')$ is the set of right endpoints, and $\weight(W\cup H)>\weight(H)$. 
\end{proof}

\paragraph{Relaxed tree chains.}
For an integer $k\ge 0$, we define the classes $\treeChFam_k$ of digraphs as follows.
Let $(F_0, s_0, t_0)$ be a $2$-terminal digraph where $F_{0}$ is a digraph on the vertex set $\{v\}$ with $s_0=t_0=v$, and let $\treeChFam_0= \{(F_0, s_0, t_0)\}$.
For every positive integer $k$, 
let $\treeChFam_k$ be the set of all $2$-terminal digraphs $(F_k, s_k, t_k)$ that can be obtained from the disjoint union of two digraphs $(F^1_{k-1}, s^1_{k-1}, t^1_{k-1})$ and $(F^2_{k-1}, s^2_{k-1}, t^2_{k-1})$ in $\treeChFam_{k-1}$ and some positive integer $a$ by
\begin{itemize}
    \item adding edges $(s^1_{i-1}, t^2_{i-1})$ and $(s^2_{i-1}, t^1_{i-1})$ if $a=1$, and
    \item adding a directed cycle chain of order $a-1$ with two endpoints $v$ and $w$ 
    and adding edges $(s^1_{i-1}, v)$, $(v, t^1_{i-1})$, $(s^2_{i-1}, w)$, $(w, t^2_{i-1})$ if $a\ge 2$,
\end{itemize}
and finally, assigning $s_k\coloneqq s^1_{k-1}$ and $t_k\coloneqq t^2_{k-1}$.
We call a digraph in $\treeChFam_k$ a \emph{relaxed tree chain} of order $k$.
See \cref{fig:constdigraph}.

\begin{figure}[!ht]
    \centering\includegraphics[scale=0.8]{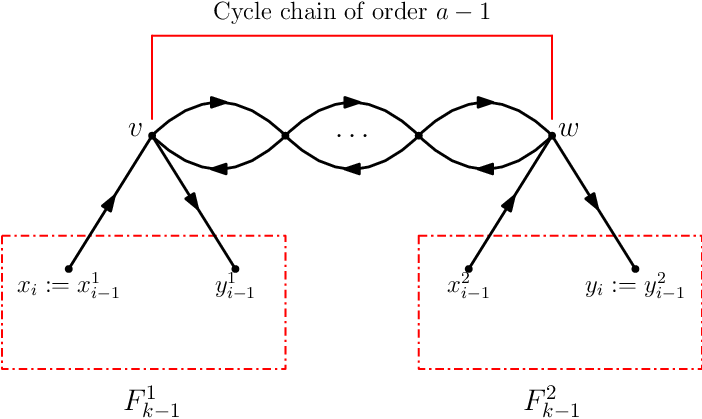}
    \caption{The recursive construction of a digraph in $\treeChFam_k$}
    \label{fig:constdigraph}
\end{figure}

The following lemma shows that every digraph in the class $\treeChFam_{2k-1}$ contains $\treeCh_k$ as a butterfly minor.

\begin{lemma}\label{lem:ReduceToBk}
    Let $k$ be a positive integer.
    Every digraph in $\treeChFam_{2k-1}$  contains $\treeCh_{k}$ as a butterfly minor.
\end{lemma}
\begin{proof}
Let $(G_{2k-1},x,y)\in \treeChFam_{2k-1}$.
We prove by induction on $k$ that $(\treeCh_k, s, t)$ has a butterfly minor model $\mu$ in $(G_{2k-1}, x, y)$ with 
the decomposition into arborescences $\{(\{r_v\}, I_v, O_v)\}_{v\in V(\treeCh_k)}$,
such that $x\in O_{s}\cup \{r_{s}\}$ and $y\in I_{t}\cup \{r_{t}\}$.

If $k=1$, then $\treeCh_1$ is a directed cycle of length $2$ on $\{s,t\}$. Assume that $(G_1, x,y)$ is obtained from two digraphs 
$(H_1, x_1, y_1)$ and $(H_2, x_2, y_2)$ in $\treeChFam_0$ and some positive integer $a$ by
\begin{itemize}
    \item adding edges $(x_1, y_2)$ and $(x_2, y_1)$ if $a=1$, and
    \item adding a directed cycle chain $\cycleCh_{a-1}$ of order $a-1$ with two endpoints $v$ and $w$ 
    and adding edges $(x_1, v)$, $(v, y_1)$, $(x_2, w)$, $(w, y_2)$ if $a\ge 2$.
\end{itemize}
If $a=1$, then let $\mu(s)$ be the single vertex digraph on $x$ and $\mu(t)$ be the single vertex digraph on $y$.
Together with edges between $x$ and $y$, this gives the desired model.
Now, assume $a\ge 2$. In this case, let $\mu(s)$ be the single vertex digraph on $x$ and $\mu(t)$ be the unique directed path from $y$ to $v$ in $G_1$.
Together with edges between $x$ and $v$, we obtain the desired model. 

Next, assume that $k\ge 2$. 
Let $(\treeCh_k,s,t)$ be constructed from the disjoint union of $(\treeCh^1_{k-1},s^1,t^1)$ and $(\treeCh^2_{k-1},s^2,t^2)$ by adding edges $(s^1, t^2)$ and $(s^2,t^1)$, where $(\treeCh^1_{k-1},s^1,t^1)$ and $(\treeCh^2_{k-1},s^2,t^2)$ are two copies of $(\treeCh_{k-1},s,t)$.
Observe that $(G_{2k-1}, x, y)$ is constructed from the disjoint union of four $2$-terminal digraphs $(H_1,x_1,y_1)$, $(H_2,x_2,y_2)$, $(H_3,x_3,y_3)$, $(H_4,x_4,y_4)$ in $\treeChFam_{2k-3}$ as follows: For each $i\in [3]$, let $a_i$ be a positive integer.
\begin{itemize}
    \item For each $i\in [3]$, if $a_i>1$, we add $\cycleCh_{a_i-1}$ with endpoints $v_i$ and $w_i$.
    \item For $i\in [2]$, if $a_i=1$, we add two edges $(x_{2i-1}, y_{2i})$ and $(x_{2i},y_{2i-1})$; otherwise, we add edges $(x_{2i-1},v_i)$, $(v_i, y_{2i-1})$, $(x_{2i}, w_i)$, $(w_i, y_{2i})$.
    \item If $a_3=1$, we add two edges $(x_{1}, y_{4})$ and $(x_{3}, y_{2})$; otherwise, we add edges $(x_{1},v_3)$, $(v_3, y_{2})$, $(x_{3}, w_3)$, $(w_3, y_{4})$.
    \item We assign $x\coloneqq x_1$ and $y\coloneqq y_4$.
\end{itemize}
See~\cref{fig:ReduceToBkCases} for an illustration. 

\begin{figure}
    \centering
    \begin{subfigure}[b]{0.49\textwidth}
        \centering
         \resizebox{\textwidth}{!}{%
         \includegraphics[width=\textwidth]{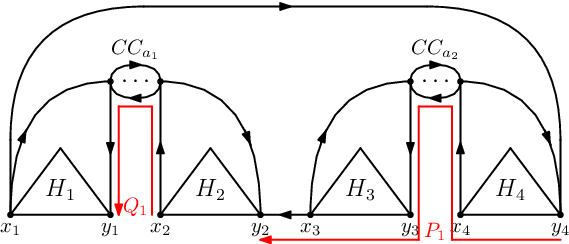}
         }
        \subcaption{Case 1}
    \end{subfigure}
    \hfill
    \begin{subfigure}[b]{0.49\textwidth}
        \centering
         \resizebox{\textwidth}{!}{%
        \includegraphics[width=\textwidth]{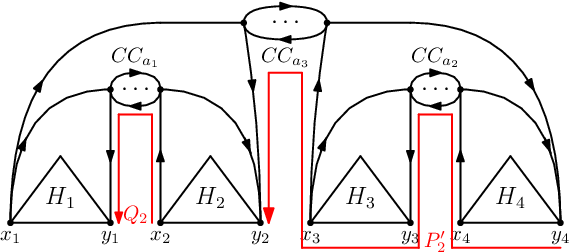}
        }
        \subcaption{Case 2}
    \end{subfigure}
    \caption{Illustrations of Case 1 and Case 2 in~\cref{lem:ReduceToBk}.}
    \label{fig:ReduceToBkCases}
\end{figure}

By induction, we may assume that 
\begin{itemize}
    \item $(B^1_{k-1},s^1,t^1)$ has a butterfly minor model $\mu_1$ in $(H_1,x_1,y_1)$ with the decomposition into arborescences $\{(\{r_v\}, I_v, O_v)\}_{v\in V(B^1_{k-1})}$, such that $x_1\in O_{s^1}\cup \{r_{s^1}\}$ and $y_1\in I_{t^1}\cup \{r_{t^1}\}$, and
    \item $(B^2_{k-1},s^2,t^2)$ has a butterfly minor model $\mu_2$ in $(H_2,x_2,y_2)$ with the decomposition into arborescences $\{(\{r_v\}, I_v, O_v)\}_{v\in V(B^2_{k-1})}$, such that $x_2\in O_{s^2}\cup \{r_{s^2}\}$ and $y_2\in I_{t^2}\cup \{r_{t^2}\}$.
\end{itemize}

We want to construct a butterfly model $\mu$ of $(\treeCh_k, s,t)$ in $(G_{2k-1},x,y)$ with the decomposition into arborescences $\{(\{r_v\}, I_v, O_v)\}_{v\in V(\treeCh_{k})}$, such that $x\in O_{s}\cup \{r_{s}\}$ and $y\in I_{t}\cup \{r_{t}\}$.
Recall that $s=s_1$ and $t=t_4$.

\medskip
\textbf{Case 1.} $a_3=1$.

Let $P_1$ be a $(y_4, y_2)$-path in $G_{2k-1}$ internally disjoint from $H_1\cup H_2$, and let $Q_1$ be a $(x_2, y_1)$-path in $G_{2k-1}$ internally disjoint from $H_1\cup H_2$. Note that $P_1$ and $Q_1$ are vertex-disjoint.
We define $\mu$ as follows.
\begin{itemize}
    \item Let $\mu(s)\coloneqq \mu_1(s^1)$, $\mu(t)\coloneqq \mu_2(t^2)\cup P_1$, and $\mu(s^2)\coloneqq \mu_2(s^2)\cup (Q_1-y_1)$.
    \item If $u\in V(B^1_{k-1})\setminus\{s^1\}$, then $\mu(u)\coloneqq \mu_1(u)$.
    \item If $u\in V(B^2_{k-1})\setminus\{s^2, t^2\}$, then $\mu(u)\coloneqq \mu_2(u)$.
    \item Let $\mu((s^1,t^2))\coloneqq (x_1, y_4)$ and let $\mu((s^2,t^1))$ be the last edge in $Q_1$.
    \item For each $i\in [2]$, if $e\in E(B^i_{k-1})$, then $\mu(e)\coloneqq\mu_i(e)$.
\end{itemize}
Clearly, $x\in O_s\cup\{r_s\}$.
Since $y_2\in I_{t^2}\cup \{r_{t^2}\}$ and $P_1$ is a $(y, y_2)$-path, we have $y\in I_{t}\cup \{r_{t}\}$.
Thus, we obtain a model, as desired.

\medskip
\textbf{Case 2.} $a_3>1$.

Let $P_2$ be a $(y_4, y_2)$-path in $G_{2k-1}$ internally disjoint from $H_1\cup H_2$, and let $Q_2$ be a $(x_2, y_1)$-path in $G_{2k-1}$ internally disjoint from $H_1\cup H_2$. Note that $P_2$ and $Q_2$ are vertex-disjoint.
We define $\mu$ as follows.
\begin{itemize}
    \item Let $\mu(s)\coloneqq \mu_1(s^1)$, $\mu(t)\coloneqq \mu_2(t^2)\cup P_2$, and $\mu(s^2)\coloneqq \mu_2(s^2)\cup (Q_2-y_1)$.
    \item If $u\in V(B^1_{k-1})\setminus\{s^1\}$, then $\mu(u)\coloneqq \mu_1(u)$.
    \item If $u\in V(B^2_{k-1})\setminus\{s^2, t^2\}$, then $\mu(u)\coloneqq \mu_2(u)$.
    \item Let $\mu((s^1,t^2))\coloneqq (x_1, v_3)$ and let $\mu((s^2,t^1))$ be the last edge in $Q_2$.
    \item For each $i\in [2]$, if $e\in E(B^i_{k-1})$, then $\mu(e)\coloneqq\mu_i(e)$.
\end{itemize}
One can observe that $x\in O_s\cup\{r_s\}$ and $y\in I_{t}\cup \{r_{t}\}$.
Thus, we again obtain a model, as desired.
\end{proof}

\section{Lemmas on mixed chains}\label{sec:mixedchain}

In this section, we prove an important technical lemma (see~\cref{lem:mixedchain2}), which we use to refine recursive structures. We consider a digraph $G$ obtained by linking two disjoint digraphs $G_1$ and $G_2$ via a mixed chain. Briefly speaking, we show that given two vertices $x$ and $y$ lying in the same $G_i$ for $i \in [2]$, there exists a \textsl{useful} pair of an $(x, V(G_j))$-path and a $(V(G_j), y)$-path for $j = 3-i$.
Furthermore, when $x$ and $y$ are additionally endpoints of another mixed chain, we show that by combining the two obtained paths, we can construct a new mixed chain with increased weight. 

We first prove this for closed mixed chains, and then generalize to any mixed chain.

\begin{lemma}\label{lem:mixedchain1}
    Let $G$ be a digraph and $H$ be a closed mixed chain in $G$ with a pair of left endpoints $(v, v)$ and a pair of right endpoints $(w, w)$. 
    Let $x,y\in V(H)$. 
Then one of the following holds.
\begin{enumerate}[(1)]
    \item There are a $(v, x)$-path and a $(y, v)$-path in $H$ that intersect within $V(H)\setminus \{v\}$.
    \item There are a $(w, x)$-path and a $(y, w)$-path in $H$ that intersect within $V(H)\setminus \{w\}$.
    \item There are a $(v, x)$-path $A$ and a $(y, v)$-path $B$ in $H$ that are vertex-disjoint except for $v$, and there are an $(A,B)$-path and a $(B,A)$-path in $H-v$.
    \item There are a $(w, x)$-path $A$ and a $(y, w)$-path $B$ in $H$ that are vertex-disjoint except for $w$, and there are an $(A,B)$-path and a $(B,A)$-path in $H-w$.
    \item $x\in V(\outarb(H, v))$ and $y\in V(\inarb(H, w))$.
    \item $x\in V(\outarb(H, w))$ and $y\in V(\inarb(H, v))$.
\end{enumerate}
\end{lemma}
\begin{proof}
     Let
        \[\left( (H_i:i\in [k+1]), (R_i:i\in [k]) \right)\]
     be the tuple for $H$. 
     For each $i\in [k+1]$, let $(P_i, Q_i)$ be the pair of boundary-paths of $H_i$.
     Observe that $\partial(H)$ is a relaxed chain with the tuple
        \[\left( (P_i:i\in [k+1]), (Q_i:i\in [k+1]), (R_i:i\in [k]) \right).\]
     Let $Z_{\mathrm{in}}, Z_{\mathrm{out}}$ be the paths in $\{P_{k+1}, Q_{k+1}\}$ whose head and tail is $w$, respectively.

    Note that there is a unique $(v, w)$-path $P$ and a unique $(w, v)$-path $Q$ in $\partial(H)$.
    Additionally, $P\cup Q=\partial(H)$ and each of $P$ and $Q$ contains $R_i$ for all $i\in [k]$.

    \medskip
    \textbf{Case 1:}
    Suppose first $x=y$. 
    If $v=x$, then $P$ and $Q$ certify that (2) holds. Assume $v\neq x$. Since $H$ is strongly connected, there exist a $(v,x)$-path and a $(y,v)$-path in $H$. As the two paths meet at $x$, (1) holds.
    So, we may assume that $x\neq y$.

    \medskip
    \textbf{Case 2:}
    Suppose that one of $x$ and $y$ is in $\{v, w\}$. First consider the case when $x=v$. If there is a $(y, w)$-path meeting the $(w,x)$-path $Q$ on $V(Q)\setminus \{w\}$,
    then (2) holds. Thus, we may assume that there is no $(y,w)$-path meeting $Q$ on $V(Q)\setminus \{w\}$. As $Q$ contains $R_k$ if $k\ge 1$, $y$ is contained in $\inext(Z_{\mathrm{in}})\setminus V(Q)$. 

    Let $U$ be the unique $(y,w)$-path in $\inext(Z_{\mathrm{in}})$.
    Note that $V(U)\cap V(Q)=\{w\}$.
    If $U-w$ has no vertex of out-degree $2$ in $H$, then $y\in V(\inarb(H, w))$. Since $x\in V(\outarb(H,v))$, (5) holds.
    So, we may assume that $U-w$ has a vertex of out-degree $2$ in $H$, say $z$. Then $z$ is an endpoint of some rung of $H_{k+1}$, and thus, there is a $(U, Q)$-path in $H-w$. Also, $Z_{\mathrm{in}}$ contains a $(Q, U)$-path in $H-w$, and this shows that (3) holds.

    Thus, we may assume that $x\neq v$.
    For the same reason, we may assume that $\{x,y\}\cap \{v, w\}=\emptyset$.

    \medskip
    \textbf{Case 3:}
    First, assume that $k\ge 1$.

    \textbf{Subcase 3.1:}
    Suppose that $y\in V(H)\setminus (V(\inext(Q_1))\cup V(R_1))$. Then there is a $(y,v)$-path in $H$ fully containing $R_1$.
    If $x\in V(H) \setminus (V(\outext(P_1))\setminus V(R_1))$, then there is a $(v,x)$-path intersecting $R_1$, and (1) holds. Thus, we may assume that $x\in V(\outext(P_1))\setminus V(R_1)$.
    Let $A$ be the unique $(v,x)$-path in $\outext(P_1)$.

      \begin{figure}[t]
        \centering
        \includegraphics[width=0.9\linewidth]{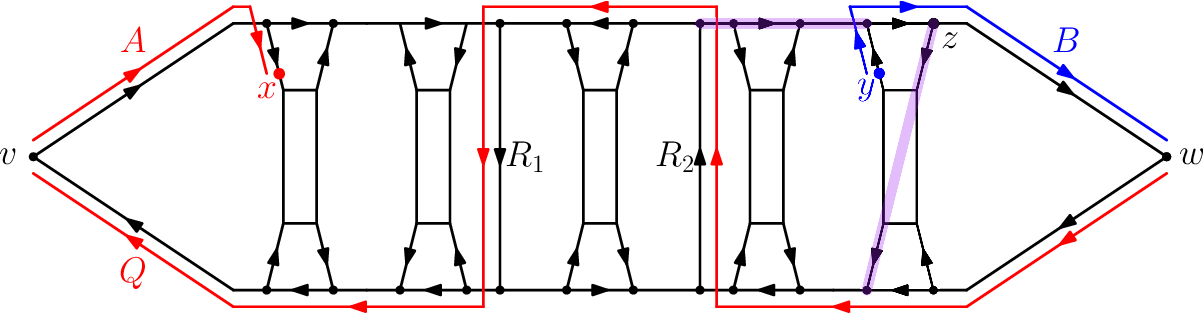}
        \caption{An illustration of directed paths $A$ and $B$ in \textbf{Subcase 3.1} in~\cref{lem:mixedchain1}. Every vertex of $V(A)\setminus \{x\}$ has in-degree $1$ in $H$, but there is a vertex of $V(B)\setminus \{y\}$ of out-degree at least $2$ in $H$, a tail of a $(V(Z_{\mathrm{in}}), V(Z_{\mathrm{out}}))$-path in some rung of $H_{k+1}$. This path is a $(B, Q\cup A)$-path in $H-w$. Also, $Z_{\mathrm{in}}$ contains a $(Q\cup A, B)$-path in $H-w$. Therefore, $Q\cup A$ and $B$ certify that (4) holds.}\label{fig:case31}
    \end{figure}

    Observe that $Q\cup A$ is a $(w,x)$-path in $H$ fully containing $R_k$. If $y\notin V(\inext(Z_{\mathrm{in}}))\setminus V(R_k)$, then there is a $(y,w)$-path intersecting $R_k$, and (2) holds.  Thus, we may assume that $y\in V(\inext(Z_{\mathrm{in}}))\setminus V(R_k)$.
    Let $B$ be the unique $(y,w)$-path in $\inext(Z_{\mathrm{in}})$. 
    See~\cref{fig:case31} for an illustration.

    If 
    every vertex of $V(A)\setminus \{v\}$ has in-degree $1$ in $H$ and 
    every vertex of $V(B)\setminus \{w\}$ has out-degree $1$ in $H$, 
    then $x\in V(\outarb(H,v))$ and $y\in V(\inarb(H,w))$, so (5) holds.
    Thus, we may assume that either there is a vertex of $V(A)\setminus \{v\}$ having in-degree more than $1$ in $H$, or 
    there is a vertex of $V(B)\setminus \{w\}$ having out-degree more than $1$ in $H$.
    By symmetry, we assume that 
    there is a vertex of $V(B)\setminus \{w\}$, say $z$, having out-degree more than $1$ in $H$.
    Then $z$ is a tail of a $(Z_{\mathrm{in}},Z_{\mathrm{out}})$-path in some rung of $H_{k+1}$.

    Note that $Q\cup A$ is a $(w,x)$-path fully containing $R_k$.
    As $z$ is a tail of a $(Z_{\mathrm{in}}, Z_{\mathrm{out}})$-path, there is a $(B, Q\cup A)$-path in $H-w$. Also, $Z_{\mathrm{in}}$ contains an $(Q\cup A, B)$-path in $H-v$. Therefore, (4) holds.

    \medskip
    \textbf{Subcase 3.2:}
    So, we may assume that $y\in V(\inext(Q_1))\cup V(R_1)$. Applying a similar argument to $x$, we may assume that 
    \[x\in (V(\outext(P_1))\cup V(R_1))\cup (V(\outext(Z_{\mathrm{out}}))\cup V(R_k)).\]
    If $x\in V(\outext(P_1))\cup V(R_1)$, then 
    (2) holds. Thus, we may assume that 
    $x\in V(\outext(Z_{\mathrm{out}}))\cup V(R_k)$.
    Let $A'$ be the unique $(w,x)$-path in $\outext(Z_{\mathrm{out}})\cup R_k$, and 
    let $B'$ be the unique $(y,v)$-path in $\inext(Q_1)\cup R_1$.

    First consider the case when $x\in V(R_k)$ or $y\in V(R_1)$. We assume that $x\in V(R_k)$. If $y\in V(\inext(Q_1))\setminus V(R_1)$, then $B'\cup P$ is a $(y,w)$-path that intersects $A'$ on $x$, and (2) holds.
    Thus, we may assume that $y\in V(R_1)$. 
    Let $B''$ be the unique $(y,w)$-path in $P$.

    If $k\ge 2$, then $B''$ intersects $A'$ on $x$. 
    Furthermore, if $k=1$ and $y,x$ appear in this order in $R_1$, then $B''$ intersects $A'$ on $x$. 
    So, we may assume $k=1$, $x\neq y$, and $x, y$ appear in this order in $R_1$. 
    In this case, there is an $(A', B'')$-path in $R_1$, and there is a $(B'', A')$-path, which is $P_1\cup Q_1$.
    Therefore, the statement (4) holds.

    As the same argument can be applied when $y\in V(R_1)$, 
    we may assume that 
    $x\notin V(R_k)$ and $y\notin V(R_1)$.
    This implies that \[y\in V(\inext(Q_1))\setminus V(R_1) \text{ and } x\in V(\outext(Z_{\mathrm{out}}))\setminus V(R_k).\]
     If 
    every vertex of $V(A')\setminus \{w\}$ has in-degree $1$ in $H$ and 
    every vertex of $V(B')\setminus \{v\}$ has out-degree $1$ in $H$, 
    then $x\in V(\outarb(H,w))$ and $y\in V(\inarb(H,v))$, so (6) holds.
    Thus, we may assume that either there is a vertex of $V(A')\setminus \{w\}$ having in-degree more than $1$ in $H$, or 
    there is a vertex of $V(B')\setminus \{v\}$ having out-degree more than $1$ in $H$. By symmetry, we assume that 
    there is a vertex of $V(A')\setminus \{w\}$, say $z$, having in-degree more than $1$ in $H$.
    Then $z$ is a head of a $(V(Z_{\mathrm{in}}), V(Z_{\mathrm{out}}))$-path in some rung of $H_{k+1}$.

     Note that $B'\cup P$ is a $(y,w)$-path fully containing $R_k$.
    As $z$ is a head of a $(Z_{\mathrm{in}}, Z_{\mathrm{out}})$-path in some rung of $H_{k+1}$,
    there is 
    a $(B'\cup P, A')$-path in $H-w$. Also, following $Q$, there is a $(A', B'\cup P)$-path in $H-w$. Therefore, (4) holds.

    This finishes the case when $k\ge 1$.

\medskip
    \textbf{Case 4:}
    In the rest, suppose that $k=0$. The argument is similar to the case $k\ge 1$, where $v$ or $w$ has a role as some $R_i$.
    
    \medskip
    \textbf{Subcase 4.1:}
    Suppose that $y\in V(H)\setminus V(\inext(Q_1))$. Then there is a $(y,v)$-path in $H$ containing $w$.
    If $x\in V(H) \setminus V(\outext(P_1))$, then there is a $(v,x)$-path containing $w$, and (1) holds. Thus, we may assume that $x\in V(\outext(P_1))$. Note that $x\neq w$. Let $A'''$ be the unique $(v,x)$-path in $\outext(P_1)$.

    Then $Q_1\cup A'''$ is a $(w,x)$-path in $H$ containing $v$. 
    If there is a $(y,w)$-path containing $v$, then we have (2). Thus, we may assume that there is no $(y,w)$-path containing $v$. This implies that $y\in V(\inext(P_1))\setminus \{v\}$. Let $B'''$ be the unique $(y,w)$-path in $\inext(P_1)$. 
    
    If there is some vertex in $B'''$ of out-degree at least $2$, then we have (4).
    So, we may assume that there is no such vertex, which implies that $y\in V(\inarb(H,w))$. 

    If $x\in V(\outarb(H, v))$, then we have (5). Thus, there is a vertex in $V(A''')\setminus \{v\}$ having in-degree more than $1$, which is a head of some $(V(Q_1), V(P_1))$-path.
    Then the paths $A'''$ and $B'''\cup Q$ certify that (3) holds.

    \medskip
    \textbf{Subcase 4.2:}
    Therefore, we may assume that $y\in V(\inext(Q_1))$. Applying a similar argument to $x$, we may assume that 
    $x\in V(\outext(P_1))\cup V(\outext(Q_1)).$
    If $ x\in V(\outext(P_1))$, then (2) holds. So, we may assume that 
    $x\in V(\outext(Q_1))$.
    Let $A^*$ be the unique $(w,x)$-path in $\outext(Q_1)$ and $B^*$ be the unique $(y,v)$-path in $\inext(Q_1)$.

    If $A^*$ and $B^*$ intersect, then $P\cup A^*$ and $B^*$ certify that (1) holds. Thus, we may assume that they are vertex-disjoint.

    If 
    every vertex of $V(A^*)\setminus \{w\}$ has in-degree $1$ in $H$ and 
    every vertex of $V(B^*)\setminus \{v\}$ has out-degree $1$ in $H$, 
    then $x\in V(\outarb(H,w))$ and $y\in V(\inarb(H,v))$, so (6) holds.
    Thus, we may assume that either there is a vertex of $V(A^*)\setminus \{w\}$ having in-degree more than $1$ in $H$, or 
    there is a vertex of $V(B^*)\setminus \{v\}$ having out-degree more than $1$ in $H$. By symmetry, we assume that 
    there is a vertex of $V(A^*)\setminus \{w\}$, say $z$, having in-degree more than $1$ in $H$.  The other case is symmetric. Then $z$ is a head of a $(P_1, Q_1)$-path in some rung of $H_1$.

    Note that $B^*\cup P$ is a $(y,w)$-path disjoint from $A^*$.
    As $z$ is a head of a $(P_1, Q_1)$-path in some rung of $H_1$,
    there is 
    a $(B^*\cup P, A^*)$-path in $H-w$. Also, following $Q$, there is a $(A^*, B^*\cup P)$-path in $H-w$. Therefore, (4) holds.

    This concludes the proof.
\end{proof}

\begin{lemma}\label{lem:mixedchain2}
    Let $(G_1, v_1, w_1), (G_2, v_2, w_2)$ be 2-terminal strongly connected digraphs, and $H$ be a mixed chain with a set of left endpoints $(p,q)$ and a set of right endpoints $(p',q')$ where 
    $v_1=w_1$ if and only if $p=q$, and
    $v_2=w_2$ if and only if $p'=q'$.
    Let $G$ be the digraph obtained from the disjoint union of $G_1, G_2, H$ by identifying $v_1$ with $p$, $w_1$ with $q$, $v_2$ with $p'$, and $w_2$ with $q'$.
    Let $x,y\in V(G)$.
    Then one of the following holds. 
    \begin{enumerate}[a)]
    \item There is a mixed extension $W$ in $G_2\cup H$ where $(p,q)$ is the set of left endpoints and $(x,y)$ is the set of right endpoints. 
    \item There is a mixed extension $W$ in $G_1\cup H$ where $(p',q')$ is the set of left endpoints and $(x,y)$ is the set of right endpoints. 
    \item $x\in V(\outarb(H, p))\cup V(G_1)$ and $y\in V(\inarb(H, q'))\cup V(G_2)$.
    \item $x\in V(\outarb(H, p'))\cup V(G_2)$ and $y\in V(\inarb(H, q))\cup V(G_1)$.
\end{enumerate}
\end{lemma}
\begin{proof}
    Let $H'$ be the digraph obtained from $H$ as follows:
    \begin{itemize}
        \item If $v_1=w_1$, then assign $v=v_1$, and otherwise, add a new vertex $v$ and add edges $(v,v_1)$ and $(w_1,v)$.
        \item If $v_2=w_2$, then assign $w=v_2$, and otherwise, add a new vertex $w$ and add edges $(w_2,w), (w, v_2)$.
    \end{itemize}
    Then $H'$ is a closed mixed chain with a pair of left endpoints $(v,v)$ and a pair of right endpoints $(w,w)$.
     We define that for each $z\in \{x,y\}$
    \begin{itemize}
        \item $z^*=z$ if $z^*\in V(H)\setminus (V(G_1)\cup V(G_2))$, 
        \item $z^*=v$ if $z^*\in V(G_1)$, and
        \item $z^*=w$ if $z^*\in V(G_2)$.
    \end{itemize}  
    By \cref{lem:mixedchain1}, one of the following holds.
    \begin{enumerate}[(1)]
    \item There are a $(v, x^*)$-path and a $(y^*, v)$-path in $H'$ that intersect within $V(H')\setminus \{v\}$.
    \item There are a $(w, x^*)$-path and a $(y^*, w)$-path in $H'$ that intersect within $V(H')\setminus \{w\}$.
    \item There are a $(v, x^*)$-path $A$ and a $(y^*, v)$-path $B$ in $H'$ that are vertex-disjoint except for $v$, and there are an $(A,B)$-path and a $(B,A)$-path in $H'-v$.  
    \item There are a $(w, x^*)$-path $A$ and a $(y^*, w)$-path $B$ in $H'$ that are vertex-disjoint except for $w$, and there are an $(A,B)$-path and a $(B,A)$-path in $H'-w$.  
    \item $x\in V(\outarb(H', v))$ and $y\in V(\inarb(H', w))$.
    \item $x\in V(\outarb(H', w))$ and $y\in V(\inarb(H', v))$.
\end{enumerate}

Assume that (1) holds. Let $A$ and $B$ be the $(v,x^*)$-path and $(y^*,v)$-path, respectively. 
Let $A'$ be obtained from $A$ by replacing the first edge $(v,z)$ with $(v_1, z)$, and let $B'$ be obtained from $B$ by replacing the last edge $(z,w)$ with $(z, w_1)$.
We define $A'', B''$ as follows. 
\begin{itemize}
    \item Let $A''=A'$ if it does not contain $w$. 
    \item If $A'$ contains $w$ as an internal vertex, then we obtain $A''$ from $A'$ by replacing $(z_1,w), (w,z_1)$ with $(z_1,q'), (p',z_2)$ respectively, and then replacing $w$ with some $(q',p')$-path in $G_2$.
    \item If $w$ is the tail of $A'$, then 
    we obtain $A''$ from $A'$ by replacing $(z_1,w)$ with $(z_1, q')$ and replacing $w$ with some $(q',x)$-path in $G_2$.
    \item We similarly define $B'$.
\end{itemize} 

 Assume that $A''$ and $B''$ intersect within $V(G)\setminus V(G_1)$. 
 By~\cref{lem:laced2}, there exists a path $B'''$ whose endpoints are the same as $B''$ such that $A''$ and $B'''$ are laced and they intersect within $V(G)\setminus V(G_1)$. Thus, $W=A''\cup B'''$ is a mixed extension as in 1).

 Now, assume that $A''$ and $B''$ do not intersect.  
 Since $A$ and $B$ intersect in $V(H')\setminus \{v\}$ while $A''$ and $B''$ do not intersect, 
the tail of $A$ and the head of $B$ are $w$, and $A''\cap G_2$ and $B''\cap G_2$ are vertex-disjoint. In $G_2$, there are a $(A''\cap G_2, B''\cap G_2)$-path $C$, and a $(B''\cap G_2, A''\cap G_2)$-path $D$. 
By~\cref{lem:laced1}, there exists a path $D'$ in $G_2$ such that $D'$ and $D$ have the same endpoints and $C$ and $D'$ are laced.
Then $W=A''\cup B''\cup C\cup D'$ is a mixed extension yielding a).

By a symmetric argument, (2) implies statement b).
Thus, we may assume that (1) and (2) do not hold. 

Now, assume that (3) holds. We define $A'', B''$ in the same way. In this case, we already have an $(A,B)$-path and a $(B,A)$-path in $H'-v$, and these correspond to an $(A'',B'')$-path and a $(B'', A'')$-path in $(G_2\cup H)-\{p,q\}$, respectively. So, a) holds. 
Similarly, if (4) holds, then b) is satisfied. 

If (5) or (6) is satisfied, then c) or d) is satisfied. 

We conclude the proof.
\end{proof}

    In some special case, we want to butterfly contract some edge in a mixed chain. The following lemma states that such a contraction preserves the property of being a mixed chain.

\begin{lemma}\label{lem:reducingmixedchain}
    Let $H$ be a mixed chain with a set of left endpoints $(p,q)$ and a set of right endpoints $(p',q')$. 
    \begin{enumerate}[(1)]
        \item Let $x,y$ be vertices in $H$ such that
        $(x,y)\in E(\outarb(H, p))$ or
        $(x,y)\in E(\inarb(H,q))$.
        If $p\neq q$, then $(x,y)$ is butterfly contractable in $H$, and $H/(x,y)$ is a mixed chain having the same weight as~$H$.
        \item Let $x,y$ be vertices in $H$ such that 
        $(x,y)\in E(\outarb(H, p'))$ or $(x,y)\in E(\inarb(H,q'))$.
         If $p'\neq q'$, then $(x,y)$ is butterfly contractable in $H$, and $H/(x,y)$ is a mixed chain having the same weight as~$H$.
    \end{enumerate} 
\end{lemma}
\begin{proof}
    This is straightforward to verify.
\end{proof}

\section{Chain decompositions}\label{sec:chaindecomp}

In this section, we define chain decompositions and related concepts.

For two $2$-terminal digraphs $\mathcal{G}_1=(G_1,v_1,w_1)$ and $\mathcal{G}_2=(G_2,v_2,w_2)$, a \emph{mixed link} of the pair $(\mathcal{G}_1, \mathcal{G}_2)$ is a digraph $G$ obtained as follows. Let $H$ be a mixed chain where $(p,q)$ is the pair of left endpoints and $(p',q')$ is the pair of right endpoints such that 
$v_1=w_1$ if and only if $p=q$, and
$v_2=w_2$ if and only if $p'=q'$.
Then, $G$ is obtained from the disjoint union of $G_1$, $G_2$, and $H$ 
by identifying $v_1$ with $p$, $v_2$ with $q$, $w_1$ with $p'$, and $w_2$ with $q'$.
We say that $(G_1, G_2, H)$ is a \emph{simple decomposition} of $G$. 

Let $G$ be a mixed link of $(\mathcal{G}_1=(G_1,v_1,w_1), \mathcal{G}_2=(G_2,v_2,w_2))$, where $(G_1, G_2, H)$ is a simple decomposition of $G$. Let $x\in V(G)$.
The \emph{out-type} of $x$ in $G$ with respect to $(G_1, G_2, H)$ is 
\begin{itemize}
    \item \emph{left} if $x\in V(G_1)\cup V(\outarb(H, v_1))$,
    \item \emph{right} if $x\in V(G_2)\cup V(\outarb(H, v_2))$, and
    \item \emph{central} otherwise.
\end{itemize}
Similarly, the \emph{in-type} of $x$ in $G$ with respect to $(G_1, G_2, H)$ is 
\begin{itemize}
    \item \emph{left} if $x\in V(G_1)\cup V(\inarb(H, w_1))$, 
    \item \emph{right} if $x\in V(G_2)\cup V(\inarb(H, w_2))$, and
    \item \emph{central} otherwise.
\end{itemize}
We say that a pair $(x,y)$ of vertices in $G$ is \emph{crossing} with respect to $(G_1, G_2, H)$ if either
\begin{itemize}
    \item the out-type of $x$ is left and the in-type of $y$ is right, and the unique $(G_1, x)$-path in $G_1\cup \outarb(H, v_1)$ intersects the unique $(y, G_2)$-path in $G_2\cup \inarb(H, w_2)$, or
    \item the out-type of $x$ is right and the in-type of $y$ is left, and the unique $(G_2, x)$-path in $G_2\cup \outarb(H, v_2)$ intersects the unique $(y, G_1)$-path in $G_1\cup \inarb(H, w_1)$.
\end{itemize}
Note that if $(x,y)$ is a crossing pair, then $x,y\in V(H)\setminus (V(G_1)\cup V(G_2))$. Furthermore, this may appear only when $\partial(H)$ has length $0$. 
See \cref{fig:CrossingEx}.

\begin{figure}[!ht]
    \centering
    \includegraphics[scale=1.0]{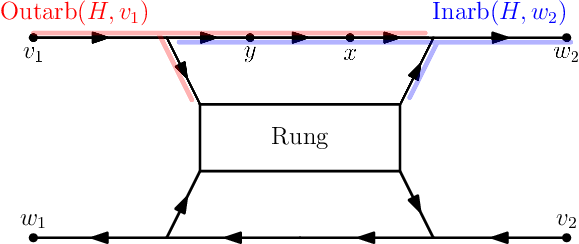}
    \caption{An example of a crossing pair $(x,y)$.}
    \label{fig:CrossingEx}
\end{figure}

For a digraph $G$, a tuple $\mathcal{T}=(T, r, \lambda, \eta)$ is a \emph{chain decomposition} of $G$ if the following is satisfied:
\begin{itemize}
    \item $(T, r)$ is a rooted tree and $\lambda(r)=G$, 
    \item every internal node of $T$ has exactly two children, and for every internal node $d$ of $T$ with children $d_1$ and $d_2$, $\lambda(d)$ is a mixed link of $(\lambda(d_1), \lambda(d_2))$ and $\eta(d)=(\lambda(d_1),\lambda(d_2),H_d)$ is a simple decomposition of $\lambda(d)$,
    \item for every leaf node $d$ of $T$,
    $\lambda(d)$ is a strongly connected digraph.
\end{itemize}
For every node $d$ of $T$, the \emph{weight} of $d$ in $\mathcal{T}$, denoted by $\weight_\mathcal{T}(d)$, is the weight of $H_d$ if
$d$ is an internal node with $\eta(d)=(\lambda(d_1), \lambda(d_2), H_d)$, and it is $0$ if $d$ is a leaf. 
The \emph{width} of $\mathcal{T}$ is the maximum weight over all nodes $d$ of $T$.
The \emph{full height} of $\mathcal{T}$ is the maximum integer $x$ such that every node $t$ with $\dist_T(r,t)\le x-2$ has two children in $T$.
In other words, the full height of $\mathcal{T}$ is the maximum height of a complete binary subtree of $T$ rooted at $r$.

When $\eta(d)=(\lambda(d_1), \lambda(d_2), H_d)$ for an internal node $d$, we say that $d_1$ is the \emph{left child} of $d$ and $d_2$ is the \emph{right child} of $d$.
Let $v_1, v_2, \ldots, v_n$ be the BFS-ordering of $T$ starting from the root $r$ of $T$ such that 
\begin{itemize}
    \item for every internal node, its left child appears before its right child, and
    \item for every two distinct nodes $v$ and $w$ of $T$ with $\dist_T(v,r)=\dist_T(w,r)$ and the least common ancestor $z$ in $T$, if $z_1$ is the child of $z$ on the path from $v$ to $z$ in $T$ and $z_2$ is the child of $z$ on the path from $w$ to $z$ in $T$ and $z_1$ appears before $z_2$, then $v$ appears before $w$.
\end{itemize}
We denote by $\mathsf{v}(\mathcal{T})$ the vector $\mathsf{v}\in \mathbb{Z}^n$ such that for every $i\in [n]$, $\mathsf{v}(i)=\weight_\mathcal{T}(v_i)$.

We define chain decompositions with additional properties, which are used to describe the increasingly refined structures of digraphs we construct in the following sections.
These additional properties are described in terms of types that depend on how, with respect to the mixed link $\lambda(d')$ of a proper descendant $d'$ of $d$, the endpoints of the mixed chain $H_d$ look.

Let $\mathcal{T}=(T, r, \lambda,\eta)$ be a \emph{chain decomposition} of a digraph $G$.
For every internal node $d$ with the simple decomposition $\eta(d)=(\lambda(d_1), \lambda(d_2), H_{d})$ and a proper descendant $d'$, we define that 
\begin{itemize}
    \item $\outvertex(d, d')=x$ if $\lambda(d')$ contains an endpoint $x$ of $H_d$ of out-degree $1$, and  
    \item $\outvertex(d, d')=\bot$, otherwise.
\end{itemize}
Similarly, we define that
\begin{itemize}
    \item $\invertex(d, d')=y$ if $\lambda(d')$ contains an endpoint $y$ of $H_d$ of in-degree $1$, and  
    \item $\invertex(d, d')=\bot$, otherwise.
\end{itemize}
When $\outvertex(d,d')\neq \bot$, we define \emph{the type of $\outvertex(d,d')$ with respect to $\eta(d')$} as the out-type of it. Similarly, when $\invertex(d,d')\neq \bot$, we define \emph{the type of $\invertex(d,d')$ with respect to $\eta(d')$} as the in-type of it.
For convenience, we say that their types are with respect to $d'$, instead of $\eta(d')$.
Similarly, when $\outvertex(d,d')\neq\bot$, $\invertex(d,d')\neq\bot$, and $\left(\outvertex(d,d'),\invertex(d,d')\right)$ is crossing with respect to $\eta(d')$, we say that it is crossing with respect to $d'$.

We also define that for every child $d'$ of an internal node $d$ with $\eta(d)=(\lambda(d_1), \lambda(d_2), H_{d})$, $\outarbor(d')\coloneqq \outarb(H_d, p)$ where $p$ is the endpoint of $\partial(H_d)$ of out-degree $1$ contained in $\lambda(d')$, and $\inarbor(d')\coloneqq \inarb(H_d, q)$ where $q$ is the endpoint of $\partial(H_d)$ of in-degree $1$ contained in $\lambda(d')$.
Note that $d$ is always the parent of $d'$, so we do not need to specify $d$ in these notations.

For internal nodes $d, d'$ where $d'$ is a proper descendant of $d$, we say that $d$ \emph{acts upon} $d'$ if $\outvertex(d,d')\neq\bot$, $\invertex(d,d')\neq\bot$,  and either
\begin{itemize}
    \item $\outvertex(d,d')$ or $\invertex(d,d')$ is central with respect to $d'$, or 
    \item $\left(\outvertex(d,d'),\invertex(d,d')\right)$ is crossing with respect to $d'$.
\end{itemize}
We say that $\mathcal{T}$ is \emph{rinsed} if for every internal node $d$ of $T$ and every proper descendant $d'$ of $d$, $d$ does not act upon $d'$.
We say that $\mathcal{T}$ is \emph{clean} if it is rinsed and 
for every internal node $d$ of $T$ and every child $d'$ of $d$ in $T$, $\outvertex(d,d')$ and $\invertex(d,d')$ have distinct types with respect to $d'$.
For a triple of internal vertices $(d_1,d_2,d_3)$ of $T$ where $d_{i+1}$ is a proper descendant of $d_i$ for $i\in[2]$, we say that $(d_1,d_2,d_3)$ satisfies the \emph{spotless property} in $\mathcal{T}$ if the following is satisfied:
\begin{itemize}
    \item If $\invertex(d_1,d_3)\neq \bot$, then 
    \begin{itemize}
        \item $\invertex(d_2,d_3)\neq \bot$,
        \item $\invertex(d_1,d_3)$ and $\invertex(d_2,d_3)$ have the same type with respect to $d_3$, and
        \item if $d_2$ is the parent of $d_3$, then $\left(\outvertex(d_2,d_3), \invertex(d_1,d_3)\right)$ is not crossing with respect to $d_3$.
    \end{itemize}
    \item If $\outvertex(d_1,d_3)\neq \bot$, then 
    \begin{itemize}
        \item $\outvertex(d_2,d_3)\neq \bot$,
        \item $\outvertex(d_1,d_3)$ and $\outvertex(d_2,d_3)$ have the same type with respect to $d_3$, and
        \item if $d_2$ is the parent of $d_3$, then
        $\left(\outvertex(d_1,d_3),\invertex(d_2,d_3)\right)$ is not crossing with respect to $d_3$.
    \end{itemize}
\end{itemize}
A triple $(d_1,d_2,d_3)$ that does not satisfy the spotless property is called a \emph{bad triple} in $\mathcal{T}$.
We say that $\mathcal{T}$ is \emph{spotless} if it is clean and every triple of internal vertices $(d_1,d_2,d_3)$ of $T$ where $d_{i+1}$ is a proper descendant of $d_i$ for $i\in[2]$ satisfies the spotless property.

Observe that a natural chain decomposition of a digraph in $\treeCh_k$ is spotless.
In~\cref{lem:SpotlessToBk}, we show that a relaxed tree chain of order $k$ can be obtained from a digraph admitting a spotless chain decomposition with large height as a butterfly minor.

We prove variants of~\cref{lem:mixedchain2} with the additional assumption that $x$ and $y$ have the same types, or that $(x,y)$ is a crossing pair.

\begin{lemma}\label{lem:mixedchain3}
    Let $(G_1, v_1, w_1), (G_2, v_2, w_2)$ be 2-terminal strongly connected digraphs, and $H$ be a mixed chain with a set of left endpoints $(p,q)$ and a set of right endpoints $(p',q')$ where 
    $v_1=w_1$ if and only if $p=q$, and
    $v_2=w_2$ if and only if $p'=q'$.
    Let $G$ be the digraph obtained from the disjoint union of $G_1, G_2, H$ by identifying $v_1$ with $p$, $w_1$ with $q$, $v_2$ with $p'$, and $w_2$ with $q'$.
    Let $x,y\in V(G)$.

    \begin{enumerate}[(1)]
        \item If the out-type of $x$ and the in-type of $y$ with respect to $(G_1, G_2, H)$ are right, then 
        there is a mixed extension $W$ in $G_2\cup H$ where $(p,q)$ is the set of left endpoints and $(x,y)$ is the set of right endpoints. 
        \item If the out-type of $x$ and the in-type of $y$ with respect to $(G_1, G_2, H)$ are left, then 
        there is a mixed extension $W$ in $G_1\cup H$ where $(p',q')$ is the set of left endpoints and $(x,y)$ is the set of right endpoints. 
    \end{enumerate}
\end{lemma}
\begin{proof}
    As the statements (1) and (2) are symmetric, it suffices to show (1). Suppose that the out-type of $x$ and the in-type of $y$ with respect to $(G_1, G_2, H)$ are right.

          Let $P$ be the unique $(p,q')$-path in  $\partial(H)$, and $Q$ be the unique $(p',q)$-path in $\partial(H)$.
        Let $R_x$ be a path from $q'$ to $x$ in $G_2\cup \outarb(H, p')$, and $R_y$ be a path from $y$ to $p'$ in $G_2\cup \inarb(H, q')$. So, $P\cup R_x$ is a $(p,x)$-path, and $Q\cup R_y$ is a $(y,q)$-path. If they intersect at $(V(G_2)\cup V(H))\setminus \{p,q\}$, then we obtain the mixed extension. Thus, we may assume that $P\cup R_x$ and $Q\cup R_y$ are vertex-disjoint except $\{p,q\}$.

        Note that both $P\cup R_x$ and $Q\cup R_y$ intersect $G_2$, which is strongly connected. So, there is a path from $P\cup R_x$ to $Q\cup R_y$ in $G_2$, and there is a path from $Q\cup R_y$ to $P\cup R_x$ in $G_2$. In this case, we also obtain a desired mixed extension.
\end{proof}

\begin{lemma}\label{lem:crossingpair}
    Let $(G_1, v_1, w_1), (G_2, v_2, w_2)$ be 2-terminal strongly connected digraphs, and $H$ be a mixed chain with a set of left endpoints $(p,q)$ and a set of right endpoints $(p',q')$ where 
    $v_1=w_1$ if and only if $p=q$, and
    $v_2=w_2$ if and only if $p'=q'$.
    Let $G$ be the digraph obtained from the disjoint union of $G_1, G_2, H$ by identifying $v_1$ with $p$, $w_1$ with $q$, $v_2$ with $p'$, and $w_2$ with $q'$.
    Let $x,y\in V(G)$. If $(x,y)$ is crossing, then there is a mixed extension $W$ in $G_2\cup H$ where $(p,q)$ is the set of left endpoints and $(x,y)$ is the set of right endpoints. 
\end{lemma}
\begin{proof}
    Suppose that $(x,y)$ is crossing.        
    Let $P$ be the unique $(p,q')$-path in  $\partial(H)$, and $Q$ be the unique $(p',q)$-path in $\partial(H)$. 
    
    First assume that the out-type of $x$ is left and the in-type of $y$ is right, and the unique $(G_1, x)$-path in $G_1\cup \outarb(H, p)$ intersects the unique $(y, G_2)$-path in $G_2\cup \inarb(H, q')$.
    Let $R_y$ be a path from $y$ to $p'$ in $G_2\cup \inarb(H, q')$. 
    Then the unique $(G_1, x)$-path in $\outarb(H, p)$ and $R_y\cup Q$ are desired paths. 

    Now, assume that the out-type of $x$ is right and the in-type of $y$ is left, and the unique $(G_2, x)$-path in $G_2\cup \outarb(H, q)$ intersects the unique $(y, G_1)$-path in $G_1\cup \inarb(H, p')$. Let $Q_x$ be a path from $q'$ to $x$ in $G_2\cup \outarb(H, q)$ and let $Q_y$ be a path from $y$ to $q$ in $Q$. Then $P\cup Q_x$ and $Q_y$ form a mixed extension. 
\end{proof}

\section{Proof of~\texorpdfstring{\cref{thm:cyclerankmainthm}}{Theorem 1.1}.}\label{sec:mainproof}

In this section, we prove~\cref{thm:cyclerankmainthm}.
We start with defining classes $\mathcal{M}_k$ of digraphs, where every digraph in $\mathcal{M}_k$ admits a chain decomposition of full height~$k$. In~\cref{thm:inductionmainstar}, we show that every digraph of sufficiently large cycle rank contains either $\cycleCh_{t}$ as a butterfly minor, or a subdigraph isomorphic to a digraph in $\mathcal{M}_k$.
In the remainder of this section, we show that every digraph in $\mathcal{M}_k$ with sufficiently large $k$ contains a digraph in $\mathcal{M}_{k'}$ with large $k'$ that admits a spotless chain decomposition, unless it does not contain $\ladder_t$ or $\cycleCh_t$ as a butterfly minor.
We then complete the proof by showing that every digraph admitting a spotless chain decomposition of sufficiently large height contains $\treeCh_t$ as a butterfly minor.

Let $\mathcal{M}_1$ be the set of all strongly connected digraphs.
For $k\geq 2$, let $\mathcal{M}_k$ be the set of all digraphs $H$ that can be obtained as follows.
\begin{itemize}
    \item 
    Let $H_1$ and $H_2$ be two digraphs from $\mathcal{M}_{k-1}$, and for each $j\in[2]$, we choose two vertices (not necessarily distinct) $v_{j,1}$ and $v_{j,2}$ in $H_j$. 
    \item $H$ is a mixed link of $(H_1, v_{1,1}, v_{1,2})$ and $(H_2, v_{2,1}, v_{2,2})$.
\end{itemize}
Note that $M_i \subsetneq M_{i-1}$ for all integers $i\geq 2$.

We start by proving an Erd\H{o}s-P\'osa type argument. 
\begin{lemma}\label{lem:erdosposa}
    Let $k$ and $w$ be positive integers, and let $\mathcal{F}$ be a family of strongly connected digraphs.
    Then every digraph of directed treewidth at most $w$ contains $k$ vertex-disjoint subdigraphs each isomorphic to a digraph in $\mathcal{F}$, or a set $S$ of at most $(k-1)(w+1)$ vertices that meet all subdigraphs isomorphic to some digraph in $\mathcal{F}$.
\end{lemma}
\begin{proof}
Let $G$ be a digraph of directed treewidth at most $w$, and 
let $(T,\beta,\gamma)$ be a directed tree decomposition of width at most $w$ for $G$.
For every $t\in V(T)$, we denote by $T_t$ the subarboresence of $T$ with root $t$, and let $\Gamma(t)\coloneqq\beta(t)\cup\bigcup_{t\sim e}\gamma(e)$ where $t\sim e$ means that $t$ is an endpoint of $e$.

We prove the statement by induction on $k$.
If $k=1$, then there is nothing to prove. Suppose that $k\ge 2$.
We may assume that $G$ contains a member of $\mathcal{F}$ as a subdigraph.
Let $t\in V(T)$ such that the subdigraph $G_t$ of $G$ induced by $\bigcup_{d\in V(T_t)}\beta(d)$ contains a member of $\mathcal{F}$ as a subdigraph while $G_t-\beta(t)$ does not.

Since $\Gamma(t)$ strongly guards $G_t$, no strongly connected subdigraph $H$ of $G-\Gamma(t)$ contains vertices of both $G_t$ and $G-V(G_t)$.
In particular, this means that no member of $\mathcal{F}$ that is a subdigraph of $G-\Gamma(t)$ can contain a vertex of $G_t$.
By induction, $G-(V(G_t)\cup \Gamma(t))$ contains either $k-1$ vertex-disjoint subdigraphs of $G-(V(G_t)\cup \Gamma(t))$ each isomorphic to a digraph in $\mathcal{F}$, or a set $S$ of at most $(k-2)(w+1)$ vertices which meet all subdigraphs of $G-(V(G_t)\cup \Gamma(t))$ isomorphic to some digraph in $\mathcal{F}$. If the former holds, then together with a member of $\mathcal{F}$ contained in $G_t$, we have $k$ vertex-disjoint subdigraphs of $G$, each isomorphic to a digraph in $\mathcal{F}$. If the latter holds, then $S\cup \Gamma(t)$ is a vertex set of size at most $(k-2)(w+1)+(w+1)=(k-1)(w+1)$ that meets all subdigraphs of $G$ isomorphic to some digraph in $\mathcal{F}$.
\end{proof}

Our further strategy is as follows.
We first show that every digraph of sufficiently large cycle rank contains $\cycleCh_t$ as a butterfly minor, or a subdigraph isomorphic to a digraph in $\mathcal{M}_k$.
We then refine this result in a sequence of iterative steps until we reach the situation of~\cref{thm:cyclerankmainthm}.

\begin{theorem}\label{thm:inductionmainstar}
    There is a function $f_{\ref{thm:inductionmainstar}}:\mathds{N}\times\mathds{N}\to \mathds{N}$ satisfying the following. 
    Let $t$ and $k$ be positive integers.
    Every digraph of cycle rank at least $f_{\ref{thm:inductionmainstar}}(t,k)$ contains either 
    $\cycleCh_{t}$ as a butterfly minor, or
    a subdigraph isomorphic to a digraph in  $\mathcal{M}_k$.
\end{theorem}
\begin{proof}
Let $f_{\mathsf{dtw}}$ be the function from \cref{thm:directedgrid}.
We fix $f_{\ref{thm:inductionmainstar}}(t,1)\coloneqq 2$ for all $t\ge 1$, $f_{\ref{thm:inductionmainstar}}(1,k)\coloneqq 2$ for all $k\ge 1$, and 
\[f_{\ref{thm:inductionmainstar}}(t,k)\coloneqq f_{\mathsf{dtw}}(t-1) + f_{\ref{thm:inductionmainstar}}(t,k-1)=(k-1)f_{\dtw}(t-1)+2\]
for all $t\ge 2$ and $k\ge 2$. 

We prove the statement by induction on $k$.
The statement holds when $k=1$ because every digraph of cycle rank at least $2$ contains a directed cycle which lies in $\mathcal{M}_1$.
Now, assume that $k\ge 2$ and let $G$ be a digraph of cycle rank at least $f_{\ref{thm:inductionmainstar}}(t,k)$.
We may assume that $G$ is strongly connected.
Suppose that $G$ does not have $\cycleCh_t$ as a butterfly minor. We show that $G$ has a subdigraph isomorphic to a digraph in $\mathcal{M}_k$.

Observe that $G$ has directed treewidth less than $f_{\mathsf{dtw}}(t-1)$; otherwise, $G$ contains a butterfly minor isomorphic to a cylindrical grid of order $t-1$ by~\cref{thm:directedgrid}, and therefore contains $\cycleCh_t$ as a butterfly minor by~\cref{lem:chainfromgrid}.

We claim the following.
\begin{claim}
   The digraph $G$ contains two vertex-disjoint subdigraphs $G_1$ and $G_2$ that are isomorphic to digraphs in $\mathcal{M}_{k-1}$.
\end{claim}
\begin{clproof}
    Suppose that this does not hold.
    By applying~\cref{lem:erdosposa} with $\mathcal{M}_{k-1}$, we obtain that there is a set $S$ of size at most $f_{\mathsf{dtw}}(t-1)$ in $G$ that meets all subdigraphs isomorphic to a digraph in $\mathcal{M}_{k-1}$.
    As $G-S$ has no $\cycleCh_t$ as a butterfly minor, by the induction hypothesis, $G-S$ has cycle rank less than $f_{\ref{thm:inductionmainstar}}(t, k-1)$ and thus
    \[\crank(G)\le \crank(G-S)+|S|<f_{\ref{thm:inductionmainstar}}(t, k-1)+ f_{\mathsf{dtw}}(t-1)=f_{\ref{thm:inductionmainstar}}(t, k). \]
    This contradicts our assumption.
\end{clproof}
Let $P$ be a shortest $(G_1, G_2)$-path and let $Q$ be a shortest $(G_2, G_1)$-path in $G$.
Such paths exist because $G$ is strongly connected.
By~\cref{lem:laced1} we may assume that $P$ and $Q$ are laced, and by~\cref{lem:lacedchain}, $P\cup Q$ is a relaxed chain.
Let $v_1$ and $w_2$ be the tail and head of $P$, respectively, and 
let $v_2$ and $w_1$ be the tail and head of $Q$, respectively. 
Then $G_1\cup G_2\cup P\cup Q$ is a mixed link of $(G_1, v_1, w_1)$ and $(G_2, v_2, w_2)$, and thus it is contained in $\mathcal{M}_k$. 
\end{proof}

\begin{figure}
\centering
\includegraphics[width=0.35\linewidth]{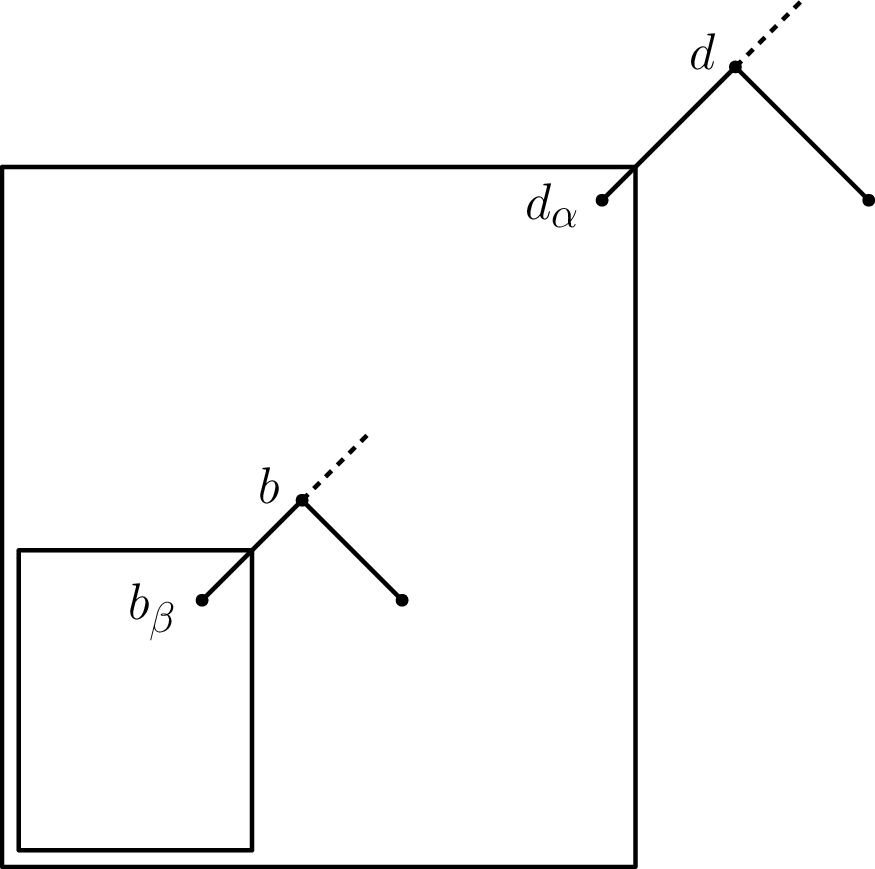}
\caption{The positions of $d,b, d_{\alpha}, b_{\beta}$ in the given chain decomposition in~\cref{lem:newdecomposition}.}
\label{fig:extendchaindecomposition}
\end{figure}

We iteratively improve a given chain decomposition when we can increase the weight of some node. 
For this purpose, the following lemma is used repeatedly.
See~\cref{fig:extendchaindecomposition} for an illustration of the setting. 

\begin{lemma}\label{lem:newdecomposition}
Let $k_1, k_2, z$ be positive integers with $k_2\ge k_1$ and $z=2^{k_1}-1$, and let $F$ be a digraph with a chain decomposition $\mathcal{T}=(T, r, \lambda, \eta)$ of full height at least $k_1+k_2$. Let $d$ and $b$ be internal nodes of $T$ where $b$ is a proper descendant of $d$ and $\dist_T(r,b)\le k_1-2$. Let $\eta(d)=(\lambda(d_1), \lambda(d_2), H_{d})$ and $\eta(b)=(\lambda(b_1), \lambda(b_2), H_{b})$, and let $\alpha, \beta\in [2]$ such that $b$ is a descendant of $d_\alpha$. 

Let $H$ be a subdigraph in $\lambda(d_{\alpha})$ and $u,v$ be vertices in $\lambda(b_{\beta})$, not necessarily distinct, such that 
\begin{itemize}
    \item $V(H)\cap V(\lambda(b_{\beta}))=\{u,v\}$, 
    \item $H_d\cup H$ is a mixed chain where $\{u, v\}$ is a set of endpoints and the other set of endpoints is the set of endpoints of $H_d$ in $\lambda(d_{3-\alpha})$, and it has weight larger than $H_d$.
\end{itemize}

Then $F$ contains a subdigraph $F'$ with a chain decomposition $\mathcal{T'}$ such that $\mathcal{T}'$ has full height at least $k_2$ and $\mathsf{v}(\mathcal{T}')[1, z]>_{\lex} \mathsf{v}(\mathcal{T})[1, z]$. 
\end{lemma}
\begin{proof}
    We take a tree $T'$ obtained from the disjoint union of $T-V(T_{d_\alpha})$ and $T_{b_\beta}$ by adding an edge between $d$ and $b_\beta$. We consider $T'$ as a tree rooted at $r$.
    We want to construct a chain decomposition on the tree $T'$ so that one of the nodes on the path from $d$ to the root has a higher weight than before.  

    Let $c_1c_2 \cdots c_m$ be the path from $d=c_1$ to $r=c_m$ in $T$.  For each $i\in [m]$, let $\eta(c_i)=(\lambda(c_{i,1}), \lambda(c_{i,2}), H_{c_i})$. We recursively extend $H_{c_i}$.
    Let $F_1=H_{c_1}\cup H=H_d\cup H$. Assume that $i\ge 2$ and $F_1, \ldots, F_{i-1}$ have been constructed. Now, we extend $H_{c_i}$ to $F_i$ as follows.  Let $(p,q)$ be the pair of endpoints of $H_{c_i}$ contained in one of $\lambda(c_{i,1})$ and $\lambda(c_{i,2})$ containing $\lambda(b)$. 
    \begin{itemize}
        \item If $p$ is not contained in $\lambda(d_\alpha)$, then $P$ is defined as the digraph on $\{p\}$. Otherwise, we take a shortest path $P$ from $\lambda(b_\beta)\cup F_1\cup \cdots \cup F_{i-1}$ to $p$ in $\lambda(d_\alpha)$. Such a path exists as $\lambda(d_\alpha)$ is strongly connected.
        \item If $q$ is not contained in $\lambda(d_\alpha)$, then $Q$ is defined as the digraph on $\{q\}$. Otherwise, we take a shortest path $Q$ from $q$ to $\lambda(b_\beta)\cup F_1\cup \cdots \cup F_{i-1}$ in $\lambda(d_\alpha)$. Such a path exists as $\lambda(d_\alpha)$ is strongly connected.
        \item By applying \cref{lem:laced1}, we obtain a directed path $Q'$ with the same endpoints such that $Q'$ is a subdigraph of $P \cup Q$, and $P$ and $Q'$ are laced. Let $F_i=H_{c_i}\cup P\cup Q'$.
    \end{itemize}
    
We define $\lambda'$ providing mixed links as follows.
For every node $t$ in $V(T')$ that is not an ancestor of~$d$, let $\lambda'(t) \coloneqq \lambda(t)$.
Let $t=c_i$ for some $i\in [m]$.
Then $\lambda'(t) \coloneqq \left(\lambda(t)-V(\lambda(d_\alpha))\right)\cup \lambda(b_{\beta})\cup \left(\bigcup_{j\in [i]} F_j\right).$

We define $\eta'$ providing simple decompositions as follows. 
For every node $t$ in $V(T')$ that is not an ancestor of $d$, let $\eta'(t) \coloneqq \eta(t)$.
If $\alpha=1$, then let $\eta'(c_1) \coloneqq (\lambda'(b_\beta), \lambda'(d_{3-\alpha}), F_1)$, and if $\alpha=2$, then let $\eta'(c_1) \coloneqq (\lambda'(d_{3-\alpha}), \lambda'(b_\beta), F_1)$.
For $i\in [m]\setminus \{1\}$,
    let $\eta'(c_i) \coloneqq (\lambda'(c_{i,1}), \lambda'(c_{i,2}), F_i)$.

It is straightforward to verify that $(T', r, \lambda', \eta')$ is a chain decomposition of $\lambda'(r)$ with full height at least $k_2$.
For each node $t$ from $r$ to $d$ in the BFS ordering, if $t$ is not an ancestor of $d$, then its weight remains the same as in $\mathcal{T}$.
By the assumption, the weight of $d$ is increased by at least one, and the weight of every ancestor of $d$ is at least as large as in $\mathcal{T}$.
This shows that $\mathsf{v}(\mathcal{T}')[1, z]>_{\lex} \mathsf{v}(\mathcal{T})[1, z]$, as desired.
\end{proof}

\begin{lemma}\label{lem:directclean}
Let $k_1,k_2,z$ be positive integers with $k_2\ge k_1$ and $z=2^{k_1}-1$, and let $F$ be a digraph with a chain decomposition $\mathcal{T}$ of full height at least $k_1+k_2$.
Then $F$ contains a subdigraph~$F'$ with a chain decomposition $\mathcal{T}'$ such that either
\begin{enumerate}[(1)]
    \item $\mathcal{T}'$ is clean and its full height is at least $k_1$, or
    \item $\mathcal{T}'$ has full height at least $k_2$ and $\mathsf{v}(\mathcal{T}')[1, z]>_{\lex} \mathsf{v}(\mathcal{T})[1, z]$. 
\end{enumerate}
\end{lemma}
\begin{proof}
    Let $\mathcal{T}=(T, r, \lambda, \eta)$, and 
    let $T'$ be the subtree of $T$ rooted at $r$ with height $k_1$.
    Assume that 
    \begin{itemize}
        \item (\textsl{rinsed condition}) for every internal node $d$ of $T'$ and every proper descendant $b$ of $d$ in $T'$, $d$ does not act upon $b$ in $\mathcal{T}$, and
        \item (\textsl{clean condition}) for every internal node $d$ of $T'$ and every child $b$ of $d$ in $T'$, $\outvertex(d,b)$ and $\invertex(d,b)$  have distinct types with respect to $b$. 
    \end{itemize}
    Note that for each leaf $\ell$ of $T'$, $\lambda(\ell)$ is a strongly connected digraph.
    Thus, $\left(T', r, \lambda|_{V(T')}, \eta|_{U}\right)$ is a clean chain decomposition of $F$, where $U$ is the set of internal nodes of $T'$.
    As it has full height $k_1$, we obtain a chain decomposition in (1).
    Therefore, we may assume that (\textsl{rinsed condition}) or (\textsl{clean condition}) is violated for some pair $(d,b)$ in $T'$ where $b$ is a proper descendant of $d$.

    As $T$ has full height at least $k_1+k_2>k_1$, $b$ is not a leaf in $T$.
    Let $\eta(d)=(\lambda(d_1), \lambda(d_2), H_d)$ and $\eta(b)=(\lambda(b_1), \lambda(b_2), H_b)$. Without loss of generality, we assume that $V(\lambda(b))\subseteq V(\lambda(d_1))$. The case when $V(\lambda(b))\subseteq V(\lambda(d_2))$ is symmetric.

    By the assumption, we have that $\outvertex(d,b)\neq\bot$ and $\invertex(d,b)\neq\bot$. Let 
    \begin{itemize}
        \item $x=\outvertex(d,b)$ and $y=\invertex(d,b)$, and 
        \item $p=\outvertex(b,b_1)$, $q=\invertex(b,b_1)$, $p'=\outvertex(b,b_2)$, and $q'=\invertex(b,b_2)$.
    \end{itemize}
    By \cref{lem:mixedchain2}, one of the following holds. 
   \begin{enumerate}[a)]
    \item There is a mixed extension $W$ in $\lambda(b_2)\cup H_b$ where $(p,q)$ is the set of left endpoints and $(x,y)$ is the set of right endpoints. 
    \item There is a mixed extension $W$ in $\lambda(b_1)\cup H_b$ where $(p',q')$ is the set of left endpoints and $(x,y)$ is the set of right endpoints. 
    \item $x\in V(\outarb(H_b, p))\cup V(\lambda(b_1))$ and $y\in V(\inarb(H_b, q'))\cup V(\lambda(b_2))$.
    \item $x\in V(\outarb(H_b, p'))\cup V(\lambda(b_2))$ and $y\in V(\inarb(H_b, q))\cup V(\lambda(b_1))$.
\end{enumerate}

We claim that one of a) and b) is satisfied.
Assume that $d$ acts upon $b$ in  $\mathcal{T}$, that is, $(d,b)$ witnesses that $\mathcal{T}$ does not satisfy the \textsl{rinsed condition}. 
If $(x,y)$ is crossing, then by~\cref{lem:crossingpair}, the statement a) holds. 
If $x$ or $y$ is central with respect to $b$, and c) and d) do not happen. 
In case when $b$ is a child of $d$, $d$ does not act upon $b$, but $x$ and $y$ do not have distinct types with respect to $b$, 
one can observe that c) and d) do not happen.
Thus, one of a) and b) is satisfied.

Then
$H_d\cup W$ is a mixed chain that has a weight larger than that of $H_d$.
 By applying \cref{lem:newdecomposition} with $H=W$, $F$ contains a subdigraph $F'$ with a chain decomposition $\mathcal{T}'$ where $\mathcal{T}'$ has full height at least $k_2$ and $\mathsf{v}(\mathcal{T}')[1, z]>_{\lex} \mathsf{v}(\mathcal{T})[1, z]$.    
\end{proof}

    In the next lemma, we show that given a clean chain decomposition, one can extract a spotless chain decomposition, or we can increase the weight of some node, or we can butterfly contract some edge without decreasing any weight.
\begin{lemma}\label{lem:spotless}
Let $k_1,k_2,z$ be positive integers with $k_2\ge k_1$ and $z=2^{k_1}-1$, and let $F$ be a digraph with a clean chain decomposition $\mathcal{T}$ of full height at least $k_1+k_2$. Then $F$ contains a butterfly minor $F'$ with a chain decomposition $\mathcal{T}'$ satisfying one of the following.
\begin{enumerate}[(1)]
    \item $\mathcal{T}'$ is spotless and its full height is at least $k_1$.
    \item $\mathcal{T}'$ has full height at least $k_2$, and  
    $\mathsf{v}(\mathcal{T}')[1,z]>_{\lex} \mathsf{v}(\mathcal{T})[1,z]$.
    \item There is a non-empty set of butterfly contractable edges $J\subseteq E(F)$ such that $F'=F/J$ and $\mathcal{T}'$ is clean, $\mathcal{T}'$ has full height at least $k_1+k_2$, and $\mathsf{v}(\mathcal{T}')[1,z]\ge_{\lex} \mathsf{v}(\mathcal{T})[1,z]$. 
\end{enumerate}
\end{lemma}
\begin{proof}
    Let $\mathcal{T}=(T, r, \lambda, \eta)$ be a clean chain decomposition of $F$ of full height at least $k_1+k_2$. Let $T'$ be the subtree of $T$ rooted at $r$ with height $k_1$. 
    Suppose that there is no bad triple $(d_1, d_2, d_3)$ in $T$ where $d_1, d_2, d_3$ are internal nodes of $T'$. 
    In this case,
    $\left(T', r, \lambda|_{V(T')}, \eta|_{U}\right)$
    is a spotless chain decomposition of $F$, where $U$ is the set of internal nodes of $T'$.
    As it has full height $k_1$, we obtain the spotless chain decomposition of (1).
    Thus, we may assume that there is a bad triple $(d_1, d_2, d_3)$ in $T$ where $d_1, d_2, d_3$ are internal nodes of $T'$.
    We choose a bad triple $(d_1, d_2, d_3)$ where $d_1, d_2, d_3$ are internal nodes of $T'$ such that 
    \begin{enumerate}[(i)]
        \item $\dist_T(r,d_1)$ is minimum, and 
        \item subject to (i), $\dist_T(d_1, d_3)$ is minimum.
    \end{enumerate}   

    Since $(d_1, d_2, d_3)$ is bad, 
    one of the following holds. 
    \begin{enumerate}
        \item $\invertex(d_1,d_3)\neq \bot$, and $\invertex(d_2,d_3)= \bot$.
        \item $\invertex(d_1,d_3)\neq \bot$, $\invertex(d_2,d_3)\neq  \bot$, and $\invertex(d_1,d_3)$ and $\invertex(d_2,d_3)$ have distinct types with respect to $d_3$.
        \item $\invertex(d_1,d_3)\neq \bot$, $d_2$ is the parent of $d_3$, and $\left(\outvertex(d_2,d_3),\invertex(d_1,d_3)\right)$ is crossing with respect to $d_3$.
        \item $\outvertex(d_1,d_3)\neq \bot$ and $\outvertex(d_2,d_3)= \bot$.
        \item $\outvertex(d_1,d_3)\neq \bot$, $\outvertex(d_2,d_3)\neq  \bot$, and $\outvertex(d_1,d_3)$ and $\outvertex(d_2,d_3)$ have distinct types with respect to $d_3$.
        \item $\outvertex(d_1,d_3)\neq \bot$, $d_2$ is the parent of $d_3$, and $\left(\outvertex(d_1,d_3),\invertex(d_2,d_3)\right)$ is crossing with respect to $d_3$.
    \end{enumerate}

    We deal with the case when $\invertex(d_1,d_3)\neq \bot$.
    The case when $\outvertex(d_1,d_3)\neq \bot$ follows along analogous arguments.
    We also assume that $\invertex(d_1,d_3)$ is right with respect to $d_3$. The case when $\invertex(d_1,d_3)$ is left is symmetric.

    Let $c_1c_2 \cdots c_m$ be the path in $T$ from $d_1=c_1$ to $d_3=c_m$, and let $\ell\in [m]$ such that $c_{\ell}=d_2$. 
    For each $i\in [m]$, let $\eta(c_i)=(\lambda(c_{i,1}), \lambda(c_{i,2}), H_{c_i})$. 
    Let $\alpha, \beta:[m-1]\to \{1,2\}$ be the functions such that for each $i\in [m-1]$, $c_{i,\alpha(i)}=c_{i+1}$ and $\beta(i)=3-\alpha(i)$.
    For convenience, we define $\alpha(m)=1$ and $\beta(m)=2$.

    We observe the following.
    \begin{claim}\label{claim:d2d3type}
    Assume $\invertex(d_2, d_3)= \bot$.
    Then $\invertex(d_2,c_{m-1})\neq \bot$.
    Moreover, if $\alpha(m-1)=1$, then $\invertex(d_2,c_{m-1})$ is left with respect to $c_{m-1}$, and otherwise, it is right with respect to $c_{m-1}$.
    \end{claim}
    \begin{clproof}
    Since $\invertex(d_2, c_{\ell, \alpha(\ell)})$ is not contained in $\lambda(d_3)$,
    there is $j\in [m-1]\setminus [\ell]$ such that
    $\invertex(d_2, c_{\ell, \alpha(\ell)})$ is contained in $\lambda(c_j)$ but not contained in $\lambda(c_{j+1})$.
   
    If $j<m-1$, then $\invertex(d_2, c_{\ell, \alpha(\ell)})$ is not contained in $\lambda(c_{m-1})$. This implies that $(d_1, d_2, c_{m-1})$ is a bad triple as $\invertex(d_1, c_{m-1})\neq \bot$. This is a contradiction because $\dist_T(d_1, c_{m-1})<\dist_T(d_1, d_3)$.
    Thus, $j=m-1$ and $\invertex(d_2,c_{m-1})\neq \bot$.

    Assume $\alpha(m-1)=1$ and $\invertex(d_2, c_{m-1})$ is right with respect to $c_{m-1}$. Then the types of $\invertex(d_1,c_{m-1})$ and $\invertex(d_2, c_{m-1})$ are distinct and thus, $(d_1, d_2, c_{m-1})$ is a bad triple. Similarly, if $\alpha(m-1)=2$ and $\invertex(d_2, c_{m-1})$ is left, then 
    $(d_1, d_2, c_{m-1})$ is a bad triple.
    Each case leads to a contradiction because $\dist_T(d_1, c_{m-1})<\dist_T(d_1, d_3)$. As $\mathcal{T}$ is rinsed, this proves the claim. 
    \end{clproof}

    \begin{claim}\label{claim:outvertexplace}
        For each $i\in [m-1]\setminus [1]$, $\outvertex(c_{i-1},c_i)$ is in $\lambda(c_{i,\beta(i)})\cup \outarbor(c_{i,\beta(i)})$. 
    \end{claim}
    \begin{clproof}
        Let $i\in [m-1]\setminus [1]$. Note that $\invertex(c_1,c_m)=\invertex(c_1,c_i)$. If $\invertex(c_1, c_i)$ and $\invertex(c_{i-1},c_i)$ have distinct types with respect to $c_i$, then $(c_1, c_{i-1}, c_i)$ is a bad triple. As $i<m$ and $\dist_T(c_1,c_i)<\dist_T(d_1, d_3)$, we have a contradiction.
    So, $\invertex(c_1, c_i)$ and $\invertex(c_{i-1},c_i)$ have the same type with respect to $c_i$. 
    This implies that $\invertex(c_1,c_i)$ and $\outvertex(c_{i-1},c_i)$ have distinct types, as $\mathcal{T}$ is clean. This implies that $\outvertex(c_{i-1},c_i)$ is contained in $\lambda(c_{i,\beta(i)})\cup \outarbor(c_{i,\beta(i)})$.
    \end{clproof}

    \begin{claim}\label{claim:goingup}
        For each $i\in [m]\setminus [2]$, there is a path $U_i$ from $\outvertex(c_{i-1},c_i)$ to $\outvertex(c_{i-2},c_{i-1})$ in
        $\lambda(c_{i-1,\beta(i-1)})\cup H_{c_{i-1}}$, where $U_i$ intersects $\lambda(c_m)$ only when $i=m$ and the intersection is the tail of $U_i$. 
    \end{claim}
    \begin{clproof}
    Assume that $i\in [m]\setminus [2]$. There is a directed path from $\outvertex(c_{i-1},c_i)$ to $\lambda(c_{i-1, \beta(i-1)})$ in $H_{c_{i-1}}$. By~\cref{claim:outvertexplace}, $\outvertex(c_{i-2},c_{i-1})$ is contained in $\lambda(c_{i-1,\beta(i-1)})\cup \outarbor(c_{i-1,\beta(i-1)})$. As $\lambda(c_{i-1, \beta(i-1  )})$ is strongly connected, we obtain the desired path. 
    \end{clproof}

    For integers $i_1, i_2$ with $3\le i_1<i_2\le m$, we define
    \[ U[i_1, i_2] \coloneqq \bigcup_{i\in [i_2]\setminus [i_1-1]} U_i. \]
    By~\cref{claim:goingup}, $U[i_1,i_2]$ is a directed path from $\outvertex(c_{i_2-1}, c_{i_2})$ to $\outvertex(c_{i_1-2}, c_{i_1-1})$.

    Let $x=\outvertex(c_{m-1}, c_m)$ and   
    $y=\invertex(d_1,d_3)$. Now, we divide cases depending on whether $x$ is left or right with respect to $c_m$.
    Let $(p,q)$ be the pair of left endpoints of $H_{d_3}$ and $(p',q')$ be the pair of right endpoints of $H_{d_3}$.
    We remind the reader that we assume $y$ to be to the right with respect to $d_3$.

    \medskip
    \textbf{Case 1:} $x$ is right with respect to $c_m$.

    \begin{figure}[!ht]
        \centering
        \includegraphics[width=0.3\linewidth]{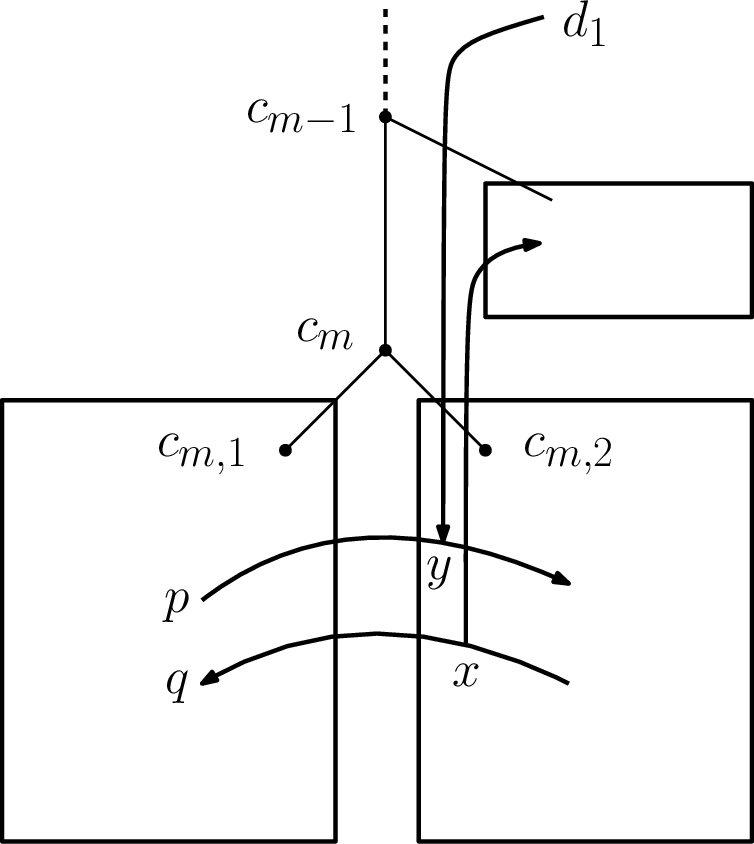}
        \caption{The positions of $x$ and $y$ in \textbf{Case 1} in~\cref{lem:spotless}. Both $x$ and $y$ are right with respect to $c_m$.}
        \label{fig:pmcase1}
    \end{figure} 

    See~\cref{fig:pmcase1} for an illustration.
    By (1) of~\cref{lem:mixedchain3}, there is a mixed extension $W$ in $H_{d_3}\cup\lambda(c_{m,2})$ where $(p,q)$ is the set of left endpoints and $(x,y)$ is the set of right endpoints. Note that $W\cup U[3,m]$ is also a mixed extension.
    By applying \cref{lem:newdecomposition} with $H=W\cup U[3,m]$, $F$ contains a subdigraph $F'$ with a chain decomposition $\mathcal{T}'$ where $\mathcal{T}'$ has full height at least $k_2$ and $\mathsf{v}(\mathcal{T}')[1, z]>_{\lex} \mathsf{v}(\mathcal{T})[1, z]$.

  So, if $x$ is right with respect to $c_m$, then we obtain $F'$ and $\mathcal{T}'$ satisfying (2).
  
  \medskip
    \textbf{Case 2:} $x$ is left with respect to $c_m$.

    We further distinguish two subcases, namely the one where $\invertex(d_2,d_3)\neq \bot$ and the one where $\invertex(d_2,d_3) = \bot$.

    \smallskip
    \textbf{Subcase 2.1:}
    We first assume that $\invertex(d_2,d_3)\neq \bot$. 
    Let $u=\invertex(d_2,d_3)$.

    \begin{figure}[!ht]
        \centering
        \includegraphics[width=0.3\linewidth]{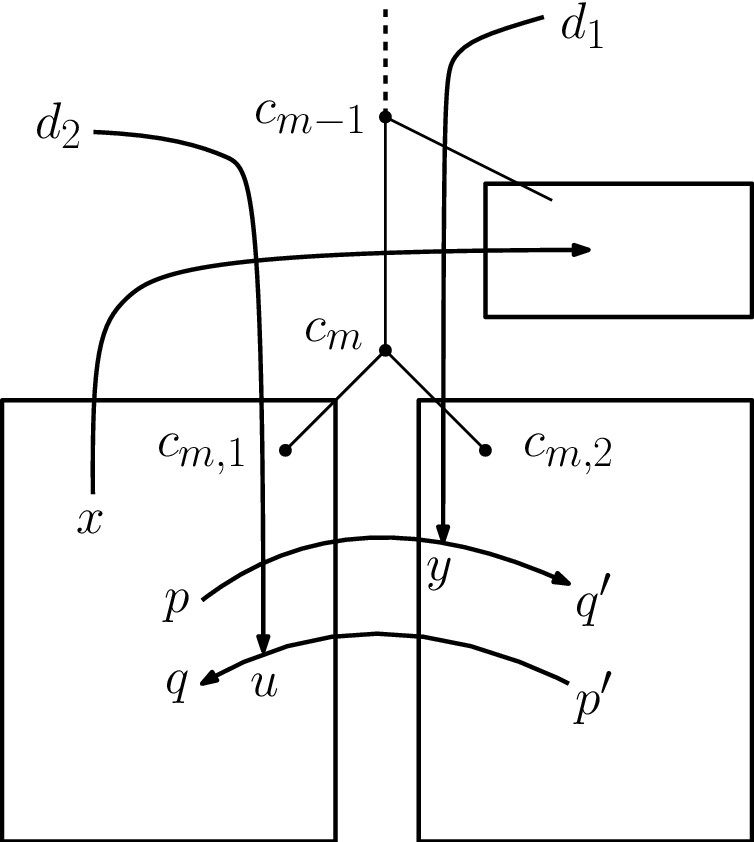}
        \caption{The positions of $x, y, u$ in \textbf{Subcase 2.1} in~\cref{lem:spotless}. Both $x$ and $u$ are left with respect to $c_m$.}
        \label{fig:pmcase21}
    \end{figure}

    As $(d_1, d_2, d_3)$ is a bad triple, $y$ and $u$ have distinct types with respect to $d_3$.
    So, because $y$ is right, $u$ is left with respect to $d_3$.
    See~\cref{fig:pmcase21} for an illustration.

    By applying (2) of~\cref{lem:mixedchain3} to $(x,u)$, 
    there is a mixed extension $W$ in $H_{d_3}\cup\lambda(c_{m,1})$ where $(p',q')$ is the set of left endpoints and $(x,u)$ is the set of right endpoints.
    By applying~\cref{lem:newdecomposition} with $H=W\cup U[\ell+2,m] $, $F$ contains a subdigraph $F'$ with a chain decomposition $\mathcal{T}'$ where $\mathcal{T}'$ has full height at least $k_2$ and $\mathsf{v}(\mathcal{T}')[1, z]>_{\lex} \mathsf{v}(\mathcal{T})[1, z]$.

     Thus, in \textbf{Subcase 2.1}, we obtain $F'$ and $\mathcal{T}'$ satisfying (2).

    \smallskip
    \textbf{Subcase 2.2:}
    Assume that $\invertex(d_1,d_3)\neq \bot$, $d_2$ is the parent of $d_3$, and 
    \[\left(\outvertex(d_2,d_3),\invertex(d_1,d_3)\right)\] is crossing with respect to $d_3$.
    Let $s=\outvertex(d_2,d_3)$.
    See~\cref{fig:pmcase22} for an illustration.
    
    In this case, $\partial(H_{d_3})$ is a relaxed chain of length $0$ and thus $H_{d_3}$ is a relaxed ladder.
    Let $(p,q)$ be the pair of left endpoints and $(p',q')$ be the pair of right endpoints of $H_{d_3}$.
    Let $A$ be the unique $\left(p,s\right)$-path in $\outarbor(c_{m,1})$ and
    let $B$ be the union of the unique $(p',q)$-path in $H_{d_3}$, a $(q',p')$-path in $\lambda(c_{m,2})$, and the unique $(y,q')$-path in $\inarbor(c_{m,2})$.
    Then $B$ is a $(y,q)$-path in $H_{d_3}\cup \lambda(c_{m,2})$.
    Since the two paths $A$ and $B$ intersect in $H_{d_3}-\{p,q\}$, $A\cup B$ is a mixed extension in $H_{d_3}\cup\lambda(c_{m,2})$ with the pair $(p,q)$ of left endpoints and the pair $\left(s,y\right)$ of right endpoints. Note that $A\cup B\cup U[3,m]$ is also a mixed extension.
    By applying~\cref{lem:newdecomposition} with $H=A\cup B\cup U[3,m]$, $F$ contains a subdigraph $F'$ with a chain decomposition $\mathcal{T}'$ where $\mathcal{T}'$ has full height at least $k_2$ and $\mathsf{v}(\mathcal{T}')[1, z]>_{\lex} \mathsf{v}(\mathcal{T})[1, z]$.

    \begin{figure}[!ht]
        \centering
        \includegraphics[width=0.35\linewidth]{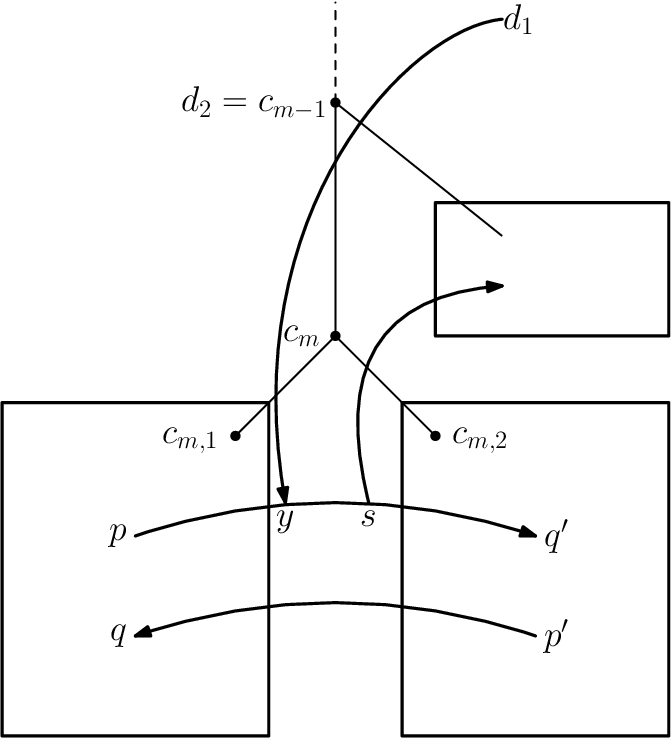}
        \caption{The positions of $y, s$ in \textbf{Subcase 2.2} in~\cref{lem:spotless}. The pair $(s,y)$ is crossing with respect to $d_3$.}
        \label{fig:pmcase22}
    \end{figure} 
    
    \begin{figure}[!ht]
        \centering
        \includegraphics[width=0.55\linewidth]{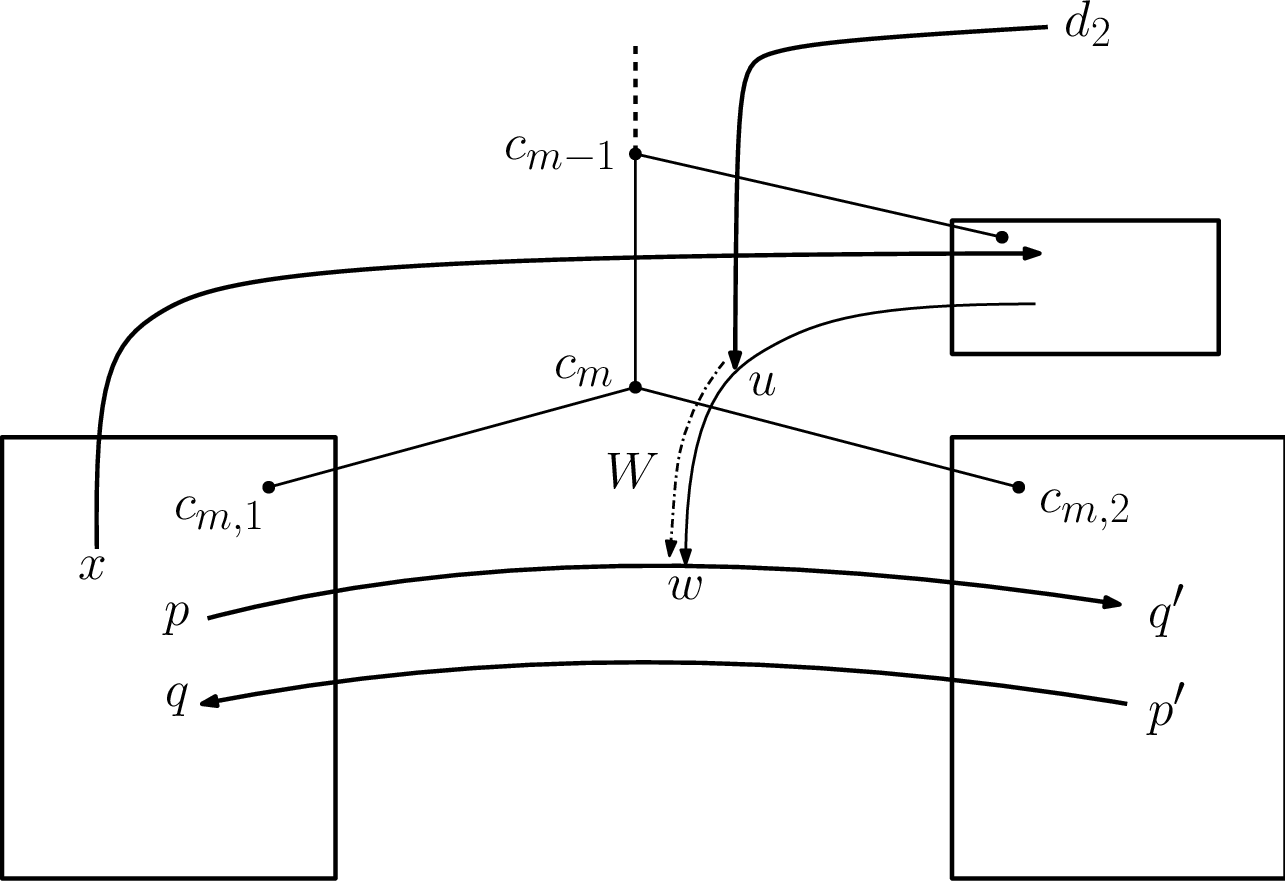}
        \caption{The positions of $x, w, u$ in \textbf{Subcase 2.3} in~\cref{lem:spotless}. The vertex $w$ is right with respect to $c_m$ and $u$ is contained in $\inarbor(c_m)$.}
        \label{fig:pmcase23}
    \end{figure} 

     \smallskip
     \textbf{Subcase 2.3:}
     Now, we assume that $\invertex(d_2,d_3)=\bot$. As $\invertex(c_{m-1},c_m)\neq\bot$, $d_2\neq c_{m-1}$ and $m\ge 4$.  
     
     By \cref{claim:d2d3type},
     $\invertex(d_2,c_{m-1})\neq \bot$ and if $\alpha(m-1)=1$, then $\invertex(d_2,c_{m-1})$ is left with respect to $c_{m-1}$, and otherwise, it is right with respect to $c_{m-1}$. 
     Assume that $\alpha(m-1)=1$. The case when $\alpha(m-1)=2$ would be symmetric.
     Let $u=\invertex(d_2, c_{m-1})$ and
     let $w=\invertex(c_{m-1}, c_m)$. See~\cref{fig:pmcase23} for an illustration. Observe that since $\invertex(d_2,d_3)=\bot$, $u$ is contained in $\inarbor(c_m)-w$. As $\mathcal{T}$ is clean and $x$ is left with respect to $c_m$, $w$ is right with respect to $c_m$.

     Let $W$ be the unique path in $\inarbor(c_m)$ from $u$ to $w$. We prove the following.

    \begin{claim}\label{claim:caseunotinW}
        If $W-w$ contains a vertex of 
out-degree more than $1$ in $F$, then we obtain $F'$ and $\mathcal{T}'$ satisfying (2).
    \end{claim}
    \begin{clproof}
    Suppose $W-w$ contains a vertex $v$ of out-degree more than $1$ in $F$. This means that $v=\outvertex(b, c_{m-1})$ for some proper ancestor $b$ of $c_{m-1}$. 
     We distinguish cases depending on $b$.

    First assume that $b=c_i$ for some $i\in [m-3]$.
         Then $\outvertex(c_i, c_{m-1})\neq \bot$. If $\invertex(c_i, c_{m-1})\neq \bot$, then $\outvertex(c_i, c_{i+1})$ and $\invertex(c_i, c_{i+1})$ have the same type with respect to $c_{i+1}$, contradicting the assumption that $\mathcal{T}$ is clean. Thus $\invertex(c_i, c_{m-1})=\bot$. Clearly, $i\neq 1$, as $\invertex(c_1, c_m)\neq \bot$. Then $\invertex(c_i, c_{m-1})$ and $\invertex(c_1, c_{m-1})$ certify that $(c_1, c_i, c_{m-1})$ is a bad triple. As $\dist_T(c_1, c_{m-1})<\dist_T(c_1, c_m)$, this is a contradiction. 

 Secondly, assume that $b$ is a proper ancestor of $c_1$.
        In this case, $\outvertex(b, c_{m-1})\neq \bot$. If $\outvertex(c_1,c_{m-1})=\bot$, then $(b, c_1, c_{m-1})$ is a bad triple where $\dist_T(r,b)<\dist_T(r,c_1)$, a contradiction.  Thus, 
        $\outvertex(c_1,c_{m-1})\neq \bot$. 
        As $\invertex(c_1, c_{m-1})\neq \bot$ and $m\ge 4$, $\outvertex(c_1, c_2)$ and $\invertex(c_1, c_2)$ have the same type with respect to $c_2$. It contradicts the assumption that $\mathcal{T}$ is clean. 

        Thus, $b=c_{m-2}$.

          Let $A$ be a path from $\lambda(c_{m,1})$ to $v$ in 
          $\lambda(c_{m-1, \beta(m-1)})\cup H_{c_{m-1}}\cup H_{c_m}$, and let $B$ be a path from $u$ to $q$ in 
        $\lambda(c_{m,2})\cup H_{c_{m-1}}\cup H_{c_m}$.
        Observe that the head of $B$ is $q$ and the tail of $A$ is a vertex in $\{x,p\}$.
        Let $O$ be the union of the head of $B$ and the tail of $A$.

        Note that when we traverse from $u$, $B$ meets $v$ on $W-w$. Therefore, $A$ and $B$ intersect outside of $O$.
        So, by~\cref{lem:laced2}, there exists a path $B'$ with the same endpoints such that $A$ and $B'$ are laced and intersect outside of $O$. 

    Then $H_{d_2}\cup A\cup B'\cup U[\ell+2,m-1] $ is a mixed chain having larger weight than $H_{d_2}$.
    By applying~\cref{lem:newdecomposition} with $H=A\cup B'\cup U[\ell+2,m-1] $, $F$ contains a subdigraph $F'$ with a chain decomposition $\mathcal{T}'$ where $\mathcal{T}'$ has full height at least $k_2$ and $\mathsf{v}(\mathcal{T}')[1, z]>_{\lex} \mathsf{v}(\mathcal{T})[1, z]$.
    \end{clproof}

    By~\cref{claim:caseunotinW}, we may assume that all edges in $W$ are butterfly contractable. 
    Now, we consider contracting $E(W)$. But when $x=w$, contracting $E(W)$ may turn $H_{c_{m-1}}$ into a subdigraph which is not a mixed chain, as there may exist some rung containing this endpoint. 
    So, we separately deal with the two cases whether $x=w$ or not.

    First, consider the case that $x=w$.
        Let $A$ be a $(p,x)$-path in $H_{c_m}\cup \lambda(c_{m,2})$ and $B$ be a $(x,q)$-path in $H_{c_m}\cup \lambda(c_{m,2})$ such that $A$ and $B$ are laced. 
        As $x=w$, $H_{d_2}\cup A\cup B\cup W\cup U[\ell+2,m] $ is a mixed chain having larger weight than $H_{d_2}$.
        By applying~\cref{lem:newdecomposition} with $H=A\cup B\cup W\cup U[\ell+2,m]$, $F$ contains a subdigraph $F'$ with a chain decomposition $\mathcal{T}'$ where $\mathcal{T}'$ has full height at least $k_2$ and $\mathsf{v}(\mathcal{T}')[1, z]>_{\lex} \mathsf{v}(\mathcal{T})[1, z]$.

    Therefore, we may assume that $x\neq w$.
        In this case, we obtain a new digraph $F'$ by contracting edges of $W$.
        Since $x\neq w$, $H_{c_{m-1}}/E(W)$ is a mixed chain having the same weight as $H_{c_{m-1}}$ by~\cref{lem:reducingmixedchain}.
        We obtain a new chain decomposition $\mathcal{T}'$ by replacing any $\lambda(t)$ containing $W$ with $\lambda(t)/E(W)$.
        Note that $W$ may contain $\invertex(b, c_{m-1})$ for some ancestor $b$ of $c_{m-1}$, and this is identified with $w$.

        It is straightforward to check that $\mathcal{T}'$ has full height at least $k_1+k_2$ and $\mathsf{v}(\mathcal{T}')[1,z]\ge_{\lex} \mathsf{v}(\mathcal{T})[1,z]$.
        We verify that $\mathcal{T}'$ is clean.

        \begin{claim}
            $\mathcal{T}'$ is clean.
        \end{claim}
        \begin{clproof}
        Observe that the type of $\outvertex(d, d')$ or $\invertex(d, d')$ does not change if it is not contained in $W-w$.
        Note that among in-vertices and out-vertices, $W$ may contain only vertices $\invertex(c,c_{m-1})$ for some ancestor $c$ of $c_{m-1}$. 
        These vertices were left with respect to $c_{m-1}$ in $\mathcal{T}$, and they are still left in $\mathcal{T}'$.
        Also, the types of these vertices with respect to the nodes on the path from $c$ to $c_{m-1}$ do not change. 
        Furthermore, their new type with respect to $c_m$ is right, as $w$ is right.

        First show that for internal nodes $d, d'$ where $d'$ is a proper descendant of $d$ in $T$, $d$ does not act upon $d'$ in $\mathcal{T}'$.
        Suppose that there are $d, d'$ where $d$ acts upon $d'$ in $\mathcal{T}'$.
        Then $\outvertex_{\mathcal{T}'}(d,d')\neq\bot$ and $\invertex_{\mathcal{T}'}(d,d')\neq\bot$. 
        So, $\outvertex_{\mathcal{T}}(d,d')\neq\bot$.
        If $\invertex_{\mathcal{T}}(d,d')=\bot$, then $d'=c_m$ and $d=c_{m-1}$, as $\outvertex_{\mathcal{T}}(c, c_m)=\bot$ for every ancestor $c$ of $c_{m-2}$.
        But $\invertex_{\mathcal{T}}(c_{m-1}, c_m)\neq \bot$, a contradiction.  Thus, $\invertex_{\mathcal{T}}(d,d')\neq \bot$.
        
        This implies that $d'$ is a child of $d$, as $\mathcal{T}$ is clean. Then the conditions for $d$ acting upon $d'$ do not hold in $\mathcal{T}'$.
        So, $\mathcal{T}'$ is rinsed.

        Now, to see that $\mathcal{T}'$ is clean, let $d, d'$ be internal nodes where $d'$ is a child of $d$.
        Suppose that $\outvertex(d,d')$ and $\invertex(d,d')$ have the same type with respect to $d'$.
        As $\outvertex_{\mathcal{T}}(d, d')$ and $\outvertex_{\mathcal{T}'}(d, d')$ have the same type and $\mathcal{T}$ is clean, we may assume that 
        $\invertex_{\mathcal{T}}(d, d')$ and $\invertex_{\mathcal{T}'}(d, d')$ have distinct types with respect to $d'$.
        But as observed above, this does not happen.
        Therefore, $\mathcal{T}'$ is clean.
        \end{clproof}
        
        Thus, in this case, we obtain $F'$ and $\mathcal{T}'$ satisfying (3).
\end{proof}

Next, we show that a relaxed tree chain of order $k$ can be extracted from a digraph admitting a spotless chain decomposition with large height. 

Let
$D$ be a digraph with a clean chain decomposition $\mathcal{T}=(T,r,\lambda,\eta)$.
For an internal node $d$ where $\dist_T(r,d)\ge 2$ and $d'$ is the parent of $d$ and $d''$ is the parent of $d'$, we say that 
$d$ \emph{receives incoming paths} if 
$\lambda(d)\cup \inarbor(d)$ contains  $\invertex(d'',d')$, and
$d$ \emph{sends outgoing paths} if $\lambda(d)\cup \outarbor(d)$ contains $\outvertex(d'',d')$.

We first show that for every internal node $d$ in a clean chain decomposition with $\dist_T(r,d)\ge 1$, one of its children receives incoming paths, and the other sends outgoing paths. 
\begin{lemma}\label{lem:receive}
    Let $D$ be a digraph admitting a clean chain decomposition $\mathcal{T}=(T,r,\lambda,\eta)$, and let $d$ be an internal node of $T$ with $\dist(r,d)\ge 1$.
    Then one of the children of $d$ receives incoming paths and the other sends outgoing paths.
\end{lemma}
\begin{proof}
    Let $d_1$ and $d_2$ be the left and right children of $d$, respectively. Let $b$ be the parent of $d$.

    As $\mathcal{T}$ is clean, $\invertex(b,d)$ is left or right with respect to $d$. 
    If $\invertex(b,d)$ is left, then $\outvertex(b,d)$ is right, and
     $\invertex(b,d)$ is contained in $\lambda(d_1)\cup \inarbor(d_1)$ and $\outvertex(b,d)$ is contained in $\lambda(d_2)\cup \outarbor(d_2)$. 
    
     If $\invertex(b,d)$ is right, then $\outvertex(b,d)$ is left, $\invertex(b,d)$ is contained in $\lambda(d_2)\cup \inarbor(d_2)$ and $\outvertex(b,d)$ is contained in $\lambda(d_1)\cup \outarbor(d_1)$. 

 Thus, we obtain the desired result. 
\end{proof}

 Furthermore, if $\mathcal{T}$ is a spotless chain decomposition, then one can obtain the following.
\begin{lemma}\label{lem:uniqueopposite}
    Let $D$ be a digraph admitting a spotless chain decomposition $\mathcal{T}=(T,r,\lambda,\eta)$, and let $d$ be an internal node of $T$ with $\dist(r,d)\ge 2$.
    If $d$ receives incoming paths, then for every proper ancestor $b$ of $d$ with $\dist_T(b,d)\ge 2$, $\outvertex(b,d)=\bot$.
    Similarly, if $d$ sends outgoing paths, then for every proper ancestor $b$ of $d$ with $\dist_T(b,d)\ge 2$, $\invertex(b,d)=\bot$. 
\end{lemma}
\begin{proof}
    Let $d'$ be the parent of $d$, and $d''$ be the parent of $d'$. Let $p \neq d$ be the other child of $d'$.

    Assume that $d$ receives incoming paths.
    Then $\lambda(d)\cup \inarbor(d)$ contains $\invertex(d'',d')$.
    Suppose towards a contradiction that there exists a proper ancestor $b$ of $d$ with $\dist_T(b,d)\ge 2$ such that $\outvertex(b,d)\in V(\lambda(d))$. 

    If $b=d''$, then $\invertex(b,d')$ and $\outvertex(b,d')$ have the same type with respect to $d'$. This contradicts the assumption that $\mathcal{T}$ is clean.  So, we may assume that $b\neq d''$. As $d$ receives incoming paths, by~\cref{lem:receive}, $p$ sends outgoing paths. Thus, $(b,d'', d')$ is a bad triple, a contradiction.  Thus, there is no such node $b$.

    When $d$ sends outgoing paths, we can prove the statement in a similar way. 
\end{proof}

We recall that $\treeChFam_k$ is the set of relaxed tree chains of order $k$.

\begin{lemma}\label{lem:SpotlessToBk}
    If a digraph $D$ admits a spotless chain decomposition $\mathcal{T}=(T,r,\lambda,\eta)$ of full height $k+3$, then $D$ contains a relaxed tree chain of order $k$ as a butterfly minor.
\end{lemma}
\begin{proof}
We may assume that $T$ is a complete binary tree of height $k+3$.

Let $0\le n\le k$ be an integer and let $d$ be a node of $T$ of height $n+1$.
We prove by induction on $n$ that $\lambda(d)$ has a butterfly minor model $\mu$ of $(F, s, t)\in \treeChFam_n$ with 
the decomposition into arborescences $\{(\{r_v\}, I_v, O_v)\}_{v\in V(F)}$,
such that 
\begin{itemize}
    \item[$(\ast)$] for every proper ancestor $b$ of $d$, if $\outvertex(b,d)\neq \bot$, then $\outvertex(b,d)\in O_{s}\cup \{r_{s}\}$, and
    if $\invertex(b,d)\neq \bot$, then $\invertex(b,d)\in I_{t}\cup \{r_{t}\}$.
\end{itemize}
Let $d'$ be the parent of $d$.

Assume that $n=0$. So, $d$ is a leaf node.
Note that for $(F,s,t)\in \treeChFam_0$, we have $V(F)=\{s\}=\{t\}$.

First, assume that $d$ receives incoming paths.
By~\cref{lem:uniqueopposite}, 
for every proper ancestor $b$ of $d$ with $\dist_T(b,d)\ge 2$, we have $\outvertex(b,d)=\bot$.
Let $X$ be a spanning in-arboresence in $\lambda(d)$ whose root is $\outvertex(d',d)$ and set 
\begin{itemize}
    \item $O_s=O_t=\emptyset$, 
    \item $r_s=r_t=\outvertex(d',d)$, and
    \item $I_s=I_t=V(X)\setminus \{r_s\}$.
\end{itemize}
This clearly satisfies the required property~$(\ast)$.
When $d$ sends outgoing paths, we can set $X$ as a spanning out-arborescence in $\lambda(d)$ whose root is $\invertex(d',d)$.
This satisfies the property~$(\ast)$.

Now, suppose that $n \ge 1$.
Let $\eta(d)=(\lambda(d_1), \lambda(d_2), H_d)$.
By exchanging $d_1$ and $d_2$ if necessary, we may assume that 
$d_1$ sends outgoing paths and $d_2$ receives incoming paths.

For each $j\in [2]$, we apply the induction hypothesis to $d_j$ and obtain a butterfly model $\mu_j$ of $(F_j, s_j, t_j)$ with 
a decomposition into arborescences $\{(\{r^j_v\}, I^j_v, O^j_v)\}_{v\in V(F_j)}$, satisfying $(\ast)$, where $(F_1, s_1, t_1), (F_2, s_2, t_2)\in \treeChFam_{n-1}$.

Let $(p,q)$ be the set of left endpoints and $(p',q')$ be the set of right endpoints of $H_d$.
Let $P$ be the unique $(p,q')$-path in $\partial(H)$, and $Q$ be the unique $(p',q)$-path in $\partial(H)$.

Let $U_{\mathrm{out}}$ be the minimal subdigraph of $\outarbor(d_1)$ containing $\outvertex(d,d_1)$ and vertices in 
\[\{\outvertex(b,d):\text{ $b$ is a proper ancestor of $d$ with } \outvertex(b,d)\neq\bot\}\cap V(\outarbor(d_1)).\]
Similarly, let $U_{\mathrm{in}}$ be the minimal subdigraph of $\inarbor(d_2)$ containing $\invertex(d,d_2)$ and vertices in 
\[\{\invertex(b,d):\text{ $b$ is a proper ancestor of $d$ with }\invertex(b,d)\neq\bot \}\cap V(\inarbor(d_2)).\]

Note that $Q$ is internally disjoint from $U_{\mathrm{out}}\cup U_{\mathrm{in}}$. 
We claim that additionally $U_{\mathrm{out}}$ and $U_{\mathrm{in}}$ are vertex-disjoint.

\begin{claim}
    $U_{\mathrm{out}}$ and $U_{\mathrm{in}}$ are vertex-disjoint.
\end{claim}
\begin{clproof}
    Suppose towards a contradiction that $U_{\mathrm{out}}$ and $U_{\mathrm{in}}$ intersect.
    Then, by definitions of $U_{\mathrm{out}}$ and $U_{\mathrm{in}}$, there exist proper ancestors $b_1$ and $b_2$ of $d$ such that 
    \begin{itemize}
        \item  $\outvertex(b_1,d)$ is contained in $\outarbor(d_1)$ and $\invertex(b_2,d)$ is contained in $\inarbor(d_2)$,
        \item the unique path from $p$ to $\outvertex(b_1,d)$ in $\outarbor(d_1)$ meets the unique path from $\invertex(b_2,d)$ to $q'$ in $\inarbor(d_2)$. 
    \end{itemize}
    
     First assume that $\dist_T(b_1, d)\ge 2$ and $\dist_T(b_2, d)\ge 2$. We assume that $b_2$ is a descendant of $d_1$. The case when $b_1$ is a descendant of $d_2$ works similarly.
     Let $b$ be the child of $b_2$, and $b'$ be the child of $b$, where $d$ is a descendant of $b'$.
     Observe that $\outvertex(b_1,b)$ and $\invertex(b_2,b)$ are contained in $V(\lambda(b'))$.
     So, $\outvertex(b_1,b)$ and $\invertex(b_2,b)$ have the same type with respect to $b$.
     As $\mathcal{T}$ is clean, $\outvertex(b_2,b)$ and $\invertex(b_2,b)$ have different types with respect to $b$.
     Thus, $\outvertex(b_1,b)$ and $\outvertex(b_2,b)$ have different types with respect to $b$.
     Therefore, $b_1\neq b_2$ and $(b_1, b_2, b)$ is a bad triple, a contradiction. 

     So, we may assume that $\dist_T(b_1,d)=1$ or $\dist_T(b_2,d)=1$.
     By symmetry, we assume that $\dist_T(b_2,d)=1$.
     If $\dist_T(b_1,d)=1$, then $(\outvertex(b_1,d), \invertex(b_2,d))$ is crossing, which contradicts the assumption that $\mathcal{T}$ is clean. 
     If $\dist_T(b_1,d)\ge 2$, then $(\outvertex(b_1,d), \invertex(b_2,d))$ is crossing, which contradicts the assumption that $\mathcal{T}$ is spotless. 

     We conclude that $U_{\mathrm{out}}$ and $U_{\mathrm{in}}$ are vertex-disjoint.
\end{clproof}

Let $P'$ be the subpath of $P$ from $U_{\mathrm{out}}$ to $U_{\mathrm{in}}$.

Observe that $P'\cup Q$ is a relaxed chain of order $\ell$, for some $\ell$, with tuple \[\left( (P_i:i\in [\ell+1]), (Q_i:i\in [\ell+1]), (R_i:i\in [\ell]) \right),\]
where $P_1\cup Q_1$ meets $\lambda(d_1)\cup U_{\mathrm{out}}$.

Now, we construct a butterfly minor model $\mu$ of $(F,s,t)$ for some $(F,s,t)\in \treeChFam_n$ where $V(F_1)\cup V(F_2)\subseteq V(F)$ and $s_1=s$ and $t_2=t$. For each $v\in V(F_1)\setminus \{s_1\}$, let $\mu(v)\coloneqq \mu_1(v)$, and for each $v\in V(F_2)\setminus \{s_2,t_2\}$, let $\mu(v)\coloneqq \mu_2(v)$. 

First, assume that $\ell=0$. In this case, $P'$ and $Q$ are vertex-disjoint. 
Let 
\begin{align*}
    \mu(s)&=\mu_1(s_1)\cup U_{\mathrm{out}}\cup (P'-V(U_{\mathrm{in}})),\\
    \mu(s_2)&=\mu_2(s_2)\cup (Q-V(\lambda(d_1))), \text{ and}\\
    \mu(t)&=\mu_2(t_2)\cup U_{\mathrm{in}}.
\end{align*}

Also, let $(F,s,t)$ be the digraph obtained from the disjoint union of $(F_1, s_1, t_1)$ and $(F_2, s_2, t_2)$ by adding $(s_1,t_2)$ and $(s_2,t_1)$. So, $\mu$ is a butterfly model of $F$ in $\lambda(d)$ satisfying $(\ast)$.

Now, we assume that $\ell\ge 1$.
Let $W$ be a cycle chain of order $\ell$ on $\{w_1, w_2, \ldots, w_\ell\}$ where for each $j\in [\ell-1]$, both $(w_j, w_{j+1})$ and $(w_{j+1}, w_j)$ are edges of $W$.
Let $F$ be the digraph obtained from the disjoint union of $(F_1, s_1, t_1)$ and $(F_2, s_2, t_2)$ by adding $W$ and the four edges $(s_1,w_1), (w_1,t_1), (s_2,w_\ell)$, and $(w_\ell, t_2)$.

Let $\mu(s)=\mu_1(s_1)\cup U_{\mathrm{out}}$ and $\mu(t)=\mu_2(t_2)\cup U_{\mathrm{in}}$. 
Let $\mu(w_1)$ be the digraph obtained from $P_1\cup P_2\cup R_1\cup Q_1\cup Q_2$ by removing the tails of $P_1, P_2$ and the heads of $Q_1, Q_2$.
For each odd $i\in [\ell]\setminus \{1\}$, let $\mu(w_i)$ be the digraph obtained from $R_i\cup P_{i+1}\cup Q_{i+1}$ by removing the tail of $P_{i+1}$ and the head of $Q_{i+1}$, and for each even $i\in [\ell]$, let $\mu(w_i)$ be the digraph obtained from $R_i\cup P_{i+1}\cup Q_{i+1}$ by removing the head of $P_{i+1}$ and the tail of $Q_{i+1}$. It is straightforward to check that $\mu$ is a butterfly model of $F$ in $\lambda(d)$ satisfying $(\ast)$.

The statement now follows from the case $n=k$.
\end{proof}

\begin{lemma}\label{lemma:ChainLaundry1}
    Let $k,t$ be positive integers.
    Let $F$ be a digraph and $\mathcal{T}$ be a chain decomposition of $F$ of full height at least $f_{\ref{lemma:ChainLaundry1}}(k,t)=2^k(4t^2+t)k$.
    Then $F$ contains
    \begin{enumerate}
        \item a butterfly minor isomorphic to $\cycleCh_t$,
        \item a butterfly minor isomorphic to $\ladder_t$, or
        \item a digraph $F'$ which admits a clean decomposition $\mathcal{T}'$ of full height at least $k$.
    \end{enumerate}
\end{lemma}

\begin{proof}
Note that
\begin{align*}
    f_{\ref{lemma:ChainLaundry1}}(k,t) & = 2^k (4t^2+t) k \geq (2^k-1) (4t^2+t-1) k + k.
\end{align*}
If the width of $\mathcal{T}$ is at least $4t^2+t-1$, then there exists some vertex $d\in V(T)$ such that $H_d$ is a mixed chain of weight at least $4t^2+t-1$.
By~\cref{lemma:ChainsandLaddersinaMix}, this means that $F$ contains a $\cycleCh_t$ or $\ladder_t$ as a butterfly minor.

Let $z\coloneqq 2^k-1$.
Then the observation above implies that $\norm{\mathsf{v}(\mathcal{T}){[1,z]}}_1 < z(4t^2+t-1)$.
Therefore, there exists a non-negative integer $i_0$ such that $\norm{\mathsf{v}(\mathcal{T}){[1,z]}}_1 + i_0 = z(4t^2+t-1)$.

Now, let $F'$ be a subdigraph of $F$ with a chain decomposition $\mathcal{T}'$ of full height at least $ik + k$ where $\norm{\mathsf{v}(\mathcal{T}'){[1,z]}}_1+i=z(4t^2+t-1)$ and $i\geq 0$ is as small as possible.
Such an $F'$ exists by the discussion above.
As before, the case where $i=0$ yields $\cycleCh_t$ or $\ladder_t$ as a butterfly minor of $F'$ (and therefore of $F$) by~\cref{lemma:ChainsandLaddersinaMix}.
Hence, we may assume $i\geq 1$.

We now apply~\cref{lem:directclean} to $F'$ and $\mathcal{T}'$ with $k_1=k$, $k_2=ik$, and $z$.
Let $F''$ be the subdigraph of $F'$ with chain decomposition $\mathcal{T}''$ as guaranteed by~\cref{lem:directclean}.
If $\mathcal{T}''$ satisfies the first outcome, we have found a clean chain decomposition of full height at least $k$ as desired.
Hence, we may assume that $\mathcal{T}''$ satisfies the second outcome.
That is $\mathcal{T}''$ has full height $ik= (i-1)k+k$ and $\mathsf{v}(\mathcal{T}''){[1,z]} >_{\lex} \mathsf{v}(\mathcal{T}'){[1,z]}$.
This, however, means that $\norm{\mathsf{v}(\mathcal{T}''){[1,z]}}_1+i>z(4t^2+t-1)$.
A last application of~\cref{lemma:ChainsandLaddersinaMix} yields that $\norm{\mathsf{v}(\mathcal{T}''){[1,z]}}_1$ is either still smaller than $z(4t^2+t-1)$ or we find $\cycleCh_t$ or $\ladder_t$ as a butterfly minor of $F''$ and therefore of $F$.
Thus, there exists $0 \leq i_1 < i$ such that $\mathcal{T}''$ is of full height at least $i_1k+k$ and $\norm{\mathsf{v}(\mathcal{T}''){[1,z]}}_1+i_1=z(4t^2+t-1)$.
This is a contradiction to our choices of $F'$ and $\mathcal{T}'$ which completes the proof.
\end{proof}

\begin{lemma}\label{lemma:ChainLaundry2}
    Let $k,t$ be positive integers.
    Let $F$ be a digraph and $\mathcal{T}$ be a clean chain decomposition of $F$ of full height at least 
    \[f_{\ref{lemma:ChainLaundry2}}(k,t)=2^k(4t^2+t)k.\]
    Then $F$ contains
    \begin{enumerate}
        \item a butterfly minor isomorphic to $\cycleCh_t$,
        \item a butterfly minor isomorphic to $\ladder_t$, or
        \item a digraph $F'$ which admits a spotless chain decomposition $\mathcal{T}'$ of full height at least $k$.
    \end{enumerate}
\end{lemma}

\begin{proof}
If the width of $\mathcal{T}$ is at least $4t^2+t-1$, then there exists some vertex $d\in V(T)$ such that $H_d$ is a mixed chain of weight at least $4t^2+t-1$.
By~\cref{lemma:ChainsandLaddersinaMix}, this means that $F$ contains a cycle chain of order $t$ or a ladder of order $t$ as a butterfly minor.

Let us take $z=2^k-1$.
Since the width of $\mathcal{T}$ is less than $4t^2+t-1$, we have $\norm{\mathsf{v}(\mathcal{T})}_1 < z(4t^2+t-1)$.
So, there is a positive integer $i_0$ such that $\norm{\mathsf{v}(\mathcal{T})[1,z]}_1 + i_0 = z(4t^2+t-1)$.
Now, we take a butterfly minor $F'$ of $F$ with a clean chain decomposition $\mathcal{T}'$ such that
\begin{enumerate}[(1)]
    \item $\mathcal{T}'$ has full height at least $ik+k$, where $\norm{\mathsf{v}(\mathcal{T}')[1,z]}_1+i=z(4t^2+t-1)$,
    \item subject to (1), $i$ is minimum, and
    \item subject to (2), $\abs{E(F')}$ is minimum.
\end{enumerate}
Note that $i>0$.

Applying~\cref{lem:spotless} to $F'$ and $\mathcal{T}'$ with $k_1=k$, $k_2=ik$, and $z=2^k-1$, we obtain a butterfly minor $F''$ of $F'$ with a clean chain decomposition $\mathcal{T}''$ satisfying one of the three outcomes in~\cref{lem:spotless}.
If $\mathcal{T}''$ satisfies the first outcome, we are done.
Suppose that $\mathcal{T}''$ satisfies the third outcome, that is, there is a non-empty set $J\subseteq E(F')$ of butterfly contractible edges such that $F''=F'/J$.
Then, we get $\abs{E(F'')}<\abs{E(F')}$ which is a contradiction.
So, we may assume that $\mathcal{T}''$ satisfies the second outcome, that is, $\mathcal{T}''$ has full height at least $ik$ and $\mathsf{v}(\mathcal{T}'')[1,z]>_{\lex}\mathsf{v}(\mathcal{T}')[1,z]$.

This implies $\norm{\mathsf{v}(\mathcal{T}'')[1,z]}_1+i>z(4t^2+t-1)$ and thus the width of $\mathcal{T}''$ is at least $4t^2+t-1$.
By~\cref{lemma:ChainsandLaddersinaMix}, we have that either $\norm{\mathsf{v}(\mathcal{T}'')[1,z]}_1<z(4t^2+t-1)$ or $F''$ contains $\cycleCh_t$ or $\ladder_t$ as a butterfly minor.
If the latter holds, we are done.
Otherwise, there exists $0<i_1<i$ such that $\mathcal{T}''$ has full height at least $i_1k+k$ and $\norm{\mathsf{v}(\mathcal{T}'')[1,z]}_1+i_1=z(4t^2+t-1)$, a contradiction.
This gives us the desired result.
\end{proof}

Now, we prove our main theorem.
We restate~\cref{thm:cyclerankmainthm} in the following form.

\begin{theorem}
    For every integer $k\ge 1$, every digraph of cycle rank at least $f_{\ref{thm:inductionmainstar}}\left(k,2^{2^{\mathcal{O}\left(k\right)}}\right)$ contains one of the digraphs $\ladder_k$, $\cycleCh_k$ or $\treeCh_k$ as a butterfly minor.
\end{theorem}

\begin{proof}
For an integer $k\ge 1$, let $k_1\coloneqq 2k+2$, $k_2\coloneqq 2^{k_1}(4k^2+k)k_1$, and $k_3\coloneqq 2^{k_2}(4k^2+k)k_2$.
Let $G$ be a digraph of cycle rank at least $f_{\ref{thm:inductionmainstar}}(k,k_3)$.
We notice that
\[k_3=2^{2^{2k+2}\cdot(2k+2)(4k^2+k)+(2k+2)}\cdot(2k+2)(4k^2+k)^2=2^{2^{\mathcal{O}\left(k\right)}}.\]

By~\cref{thm:inductionmainstar}, $G$ contains either a butterfly minor isomorphic to $\cycleCh_{k}$, or a subdigraph isomorphic to a digraph in $\mathcal{M}_{k_3}$.
If the former holds, we are done. So, we assume that the latter holds.
By construction of the class $\mathcal{M}_{k_3}$, $G$ admits a chain decomposition $\mathcal{T}$ of full height $k_3$.

Since $k_3 = 2^{k_2}(4k^2+k)k_2$, by~\cref{lemma:ChainLaundry1}, $G$ contains
\begin{enumerate}
    \item a butterfly minor isomorphic to $\cycleCh_k$,
    \item a butterfly minor isomorphic to $\ladder_k$, or
    \item a digraph $G'$ which admits a clean decomposition $\mathcal{T}'$ of full height at least $k_2$.
\end{enumerate}
If the first or second one holds, then we are done.
Otherwise, since $k_2= 2^{k_1}(4k^2+k)k_1$, by~\cref{lemma:ChainLaundry2}, $G'$ contains
\begin{enumerate}
    \item a butterfly minor isomorphic to $\cycleCh_k$,
    \item a butterfly minor isomorphic to $\ladder_k$, or
    \item a digraph $G''$ which admits a spotless chain decomposition $\mathcal{T}''$ of full height at least $k_1$.
\end{enumerate}

If the first or second one holds, then we are again done.
So, assume that the third one holds.
Since $k_1=2k+2$, by \cref{lem:SpotlessToBk}, $G''$ contains a butterfly minor of some graph in $\treeChFam_{2k-1}$.
By~\cref{lem:ReduceToBk}, we conclude that $G''$ contains a butterfly minor of $\treeCh_k$.
This proves the theorem.
\end{proof}

\section{Comparison between the three classes}\label{sec:independent}

In this section, we show that for the three classes $\{\ladder_n:n\in \mathbb{N}\}$, $\{\cycleCh_m:m\in \mathbb{N}\}$, and $\{\treeCh_k:k\in \mathbb{N}\}$, the closure of any one class under butterfly minors does not contain any other of the classes.
This shows that these are independent families of obstructions, none of them is redundant.

The \emph{circumference} of a digraph $G$, denoted by $\circum(G)$, is the length of a longest cycle in $G$ if $G$ contains a directed cycle, and $0$ otherwise.
The following statement is useful to show that $\ladder_n$ with $n\ge 2$ and $\treeCh_k$ with $k\ge 3$ are not a butterfly minor of $\cycleCh_m$, as $\cycleCh_m$ has circumference $2$.
\begin{lemma}\label{lem:longestcycle}
    Let $G$ be a digraph and $H$ be a butterfly minor of $G$.
    Then $\circum(H)\le \circum(G)$.
\end{lemma}
\begin{proof}
    As any butterfly minor of $G$ is obtained from $G$ by a sequence of deleting vertices, deleting edges, and butterfly contracting edges, it is sufficient to prove the statement for the case that $H$ is a digraph obtained from $G$ by one of these three operations.
    If $H$ is obtained from $G$ by either deleting a vertex or deleting an edge, then obviously $\circum(H)\le \circum(G)$.
    
    Assume that $H$ is obtained from $G$ by butterfly contracting an edge $(u,v)$.
    Let $x_{u,v}$ be the identified vertex in $H$.
    Let $C$ be a longest cycle in $H$.
    We claim that $|V(C)|\le \circum(G)$.
    If this is true, then the lemma holds. 
    
    If $C$ does not contain $x_{u,v}$, then $C$ is a cycle in $G$.
    So, we may assume that $C$ contains $x_{u,v}$.
    Let $C=v_1v_2\cdots v_i x_{u,v}v_1$.
    
    If $v_i$ is an in-neighbor of $u$ and $v_1$ is an out-neighbor of $v$, then $v_1v_2\cdots v_i u v v_1$ is a cycle in $G$, implying that $\circum(G)\ge |V(C)|$.
    Otherwise, as $(u,v)$ is butterfly contractible, either 
    \begin{itemize}
        \item $(v_i, u)$, $(u, v_1)\in E(G)$, or
        \item $(v_i, v)$, $(v, v_1)\in E(G)$.
    \end{itemize}
    Thus, the digraph obtained from $C$ by replacing $x_{u,v}$ with $u$ or $v$ is also a directed cycle in $G$, and we have $\circum(G)\ge |V(C)|$.
\end{proof}

\begin{lemma}\label{lem:cyclechainladder}
    Let $n\ge 2$ and $m\ge 3$ be integers.
    \begin{enumerate}[(1)]
        \item $\ladder_{n}$ is not a butterfly minor of $\cycleCh_{m}$.
        \item $\cycleCh_{m}$ is not a butterfly minor of $\ladder_{n}$.
    \end{enumerate} 
\end{lemma}
\begin{proof}
    (1)
    The circumference of $\cycleCh_m$ is $2$, but the circumference of $\ladder_n$ is at least $4$.
    By~\cref{lem:longestcycle}, this implies that $\ladder_n$ is not a butterfly minor of $\cycleCh_m$.

  \medskip

    (2) It suffices to show that $\cycleCh_3$ is not a butterfly minor of $\ladder_n$.

    Suppose there is a butterfly minor model $\mu$ of $\cycleCh_3$ in $\ladder_n$.
    Let $\{(\{r_v\}, I_v, O_v)\}_{v\in V(\cycleCh_3)}$ be the decomposition of $\mu(\cycleCh_3)$ into arborescences.
    Since $\cycleCh_3[\{v_1,v_2\}]$ is a cycle, $\mu(\cycleCh_3[\{v_1,v_2\}])$ has a unique cycle $C$ containing $\mu((v_1,v_2))$ and $\mu((v_2,v_1))$. 
    Observe that there exist $i, j\in [n]$ with $i\le j$ such that 
    \[C=p_ip_{i+1} \cdots p_jq_{(n+1)-j}q_{n-j} \cdots q_{(n+1)-i}p_i,\]
    as every cycle in $\ladder_n$ is such a cycle.

    Note that $r_{v_3}\in V(\cycleCh_3) \setminus V(C)$.
    Let $P$ be a $(V(C), r_{v_3})$-path and $Q$ be a $(r_{v_3}, V(C))$-path in $\mu(\cycleCh_3[\{v_2, v_3\}])$.
    Let $w$ be the tail of $P$ and $z$ be the head of $Q$.
    As $C$ contains $r_{v_2}$, we have that $w\in \{r_{v_2}\}\cup O_{v_2}$ and $z\in \{r_{v_2}\}\cup I_{v_2}$.

    On the other hand, depending on the placement of $r_{v_3}$ in $\ladder_n$, the pair $(w,z)$ is 
    either $(q_{(n+1)-i}, p_i)$ or $(p_j, q_{(n+1)-j})$.
    Note that $(w,z)$ is an edge of $C$, and thus, it is contained in $\mu(v_2)$.
    As $w\in \{r_{v_2}\}\cup O_{v_2}$, $z$ must be contained in $O_{v_2}$.
    This contradicts the fact that $z\in \{r_{v_2}\}\cup I_{v_2}$.
\end{proof}

\begin{lemma}\label{lem:propofFak}
    Let $k$ be a positive integer. 
    \begin{enumerate}[(1)]
    \item There is a unique $(t_k,s_k)$-path in $\treeCh_k$ and it contains all vertices in $\treeCh_k$.
    \item If $C$ is a cycle in $\treeCh_{k}$ containing $(s^1_{k-1}, t^2_{k-1})$ or $(s^2_{k-1},t^1_{k-1})$, then $C$ contains all vertices in $\treeCh_{k}$.
    \item $\circum(\treeCh_k)=2^k$.
    \end{enumerate}
\end{lemma}
\begin{proof}
    (1) We prove the statement by induction on $k$.
    It clearly holds when $k=1$, as $\treeCh_{1}$ is a directed cycle of size $2$.
    Thus, assume $k>1$.
    Let $P$ be a $(t_k, s_k)$-path in $\treeCh_{k}$.
    Such a path exists as $\treeCh_{k}$ is strongly connected.

    Observe that there are exactly two edges $(s^1_{k-1},t^2_{k-1})$ and $(s^2_{k-1},t^1_{k-1})$ which meet both $\treeCh^1_{k-1}$ and $\treeCh^2_{k-1}$. So, $P$ contains $(s^2_{k-1},t^1_{k-1})$ and does not contain $(s^1_{k-1},t^2_{k-1})$.
    Thus, $P$ consists of a $(t^2_{k-1},s^2_{k-1})$-path in $\treeCh^2_{k-1}$, the edge $(s^2_{k-1},t^1_{k-1})$, and a $(t^1_{k-1},s^1_{k-1})$-path in $\treeCh^1_{k-1}$. By induction, there is a unique $(t^2_{k-1},s^2_{k-1})$-path in $\treeCh^2_{k-1}$ which contains all vertices in $\treeCh^2_{k-1}$, and there is a unique $(t^1_{k-1},s^1_{k-1})$-path in $\treeCh^1_{k-1}$ which contains all vertices in $\treeCh^1_{k-1}$.
    Therefore, $P$ is a unique $(t_k,s_k)$-path which contains all vertices in $\treeCh_{k}$.

\medskip
    (2) We assume that $C$ contains the edge $(s^1_{k-1}, t^2_{k-1})$.
    Then $C$ contains a $(t^2_{k-1}, s^1_{k-1})$-path in $\treeCh_k$. By (1), such a path contains all vertices in $\treeCh_k$. A similar argument holds when $C$ contains $(s^2_{k-1}, t^1_{k-1})$.

    \medskip
    (3) This follows directly from (2). 
 \end{proof}

\begin{lemma}\label{lem:cyclechainFak}
    Let $k\ge 2$ and $n\ge 3$ be integers.
    \begin{enumerate}[(1)]
        \item $\treeCh_k$ is not a butterfly minor of $\cycleCh_{n}$.
        \item $\cycleCh_{n}$ is not a butterfly minor of $\treeCh_k$. 
    \end{enumerate}
\end{lemma}
\begin{proof}
    (1) The circumference of $\cycleCh_n$ is $2$ but the circumference of $\treeCh_k$ is $2^k>2$.
    By~\cref{lem:longestcycle}, this implies that $\treeCh_k$ is not a butterfly minor of $\cycleCh_{n}$.

\medskip

    (2) It suffices to show that $\cycleCh_{3}$ is not a butterfly minor of $\treeCh_k$.
    Suppose this is not the case. Let us choose $k$ minimum such that $\treeCh_k$ has $\cycleCh_{3}$ as a butterfly minor.
    As $\treeCh_1$ is a directed cycle of size $2$, we may assume that $k\ge 2$.
    
    Let $\mu$ be a minimal butterfly minor model of $\cycleCh_{3}$ in $\treeCh_k$.
    By the minimality of $k$, we know that both $\treeCh^1_{k-1}$ and $\treeCh^2_{k-1}$ do not have $\cycleCh_{3}$ as a butterfly minor.
    As $\cycleCh_{3}$ is strongly connected and $\mu$ is a minimal model, $\mu(\cycleCh_{3})$ is strongly connected by~\cref{lem:minimalmodel}. 
    
    As $\treeCh^1_{k-1}$ and $\treeCh^2_{k-1}$ do not have $\cycleCh_{3}$ as a butterfly minor and $\mu(\cycleCh_3)$ is strongly connected, $\mu(\cycleCh_{3})$ contains both $(s^1_{k-1}, t^2_{k-1})$ and $(s^2_{k-1},t^1_{k-1})$. 
    Thus, there is a directed cycle $C$ in $\mu(\cycleCh_{3})$ containing $(s^1_{k-1}, t^2_{k-1})$.
    This cycle is contained in $\mu(\cycleCh_{3}[\{x,y\}])$ for some two consecutive vertices $x$ and $y$ in $\cycleCh_{3}$.
    By (2) of~\cref{lem:propofFak}, $C$ contains all vertices in $\treeCh_k$.
    This contradicts the fact that $\cycleCh_3$ contains three vertices.
\end{proof}

\begin{figure}[!ht]
    \centering
    \begin{subfigure}[b]{0.5\textwidth}
        \centering
         \resizebox{\textwidth}{!}{%
         \includegraphics[width=\textwidth]{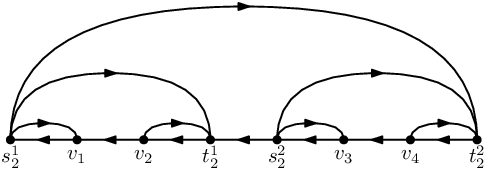}
         }
        \subcaption{A vertex labelling of $\treeCh_3$.}
    \end{subfigure}
    \hfill
    \begin{subfigure}[b]{0.4\textwidth}
        \centering
         \resizebox{\textwidth}{!}{%
        \includegraphics[width=\textwidth]{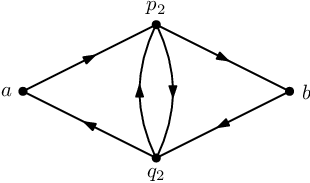}
        }
        \subcaption{The digraph $L'_3$.}
    \end{subfigure}
    \hfill
    \caption{A vertex labelling of $\treeCh_3$ and the digraph $\ladder_3'$ in the proof of \cref{lem:LadderTreeChain}.}
    \label{fig:TC3andL3}
\end{figure}

\begin{lemma}\label{lem:LadderTreeChain}
    Let $k\ge 3$ and $n\ge 3$ be integers.
    \begin{enumerate}[(1)]
        \item $\treeCh_{k}$ is not a butterfly minor of $\ladder_{n}$.
        \item $\ladder_{n}$ is not a butterfly minor of $\treeCh_{k}$.
    \end{enumerate} 
\end{lemma}
\begin{proof} 
    (1) We claim that $\treeCh_{3}$ is not a butterfly minor of $\ladder_n$.
    Suppose it was.
    Let us name all vertices in $\treeCh_3$ as in (a) of~\cref{fig:TC3andL3}.
    Let $\mu_1$ be a butterfly minor model of $\treeCh_3$ in $\ladder_n$.
    In $\treeCh_3$, there are four disjoint cycles $A_1\coloneqq s_3v_1s_3$, $A_2\coloneqq v_2t^1_2v_2$, $A_3\coloneqq s^2_2v_3s^2_2$, and $A_4\coloneqq v_4t_3v_4$.

    For each $i\in [4]$, there exist $a_i, b_i\in [n]$ with $a_i\le b_i$ such that $\mu_1(A_i)$ contains a cycle \[C_i=p_{a_i}p_{a_i+1}\cdots p_{b_i}q_{n+1-b_i}q_{n-b_i}\cdots q_{n+1-a_i}p_{a_i}.\]
    Clearly, $C_1$, $C_2$, $C_3$ and $C_4$ are vertex-disjoint in $\ladder_n$.

    Since $(v_2, v_1)$ is an edge of $\treeCh_3$, $\mu_1((v_2, v_1))$ connects $\mu_1(A_1)$ and $\mu_1(A_2)$.
    Since each of $C_3$ and $C_4$ is vertex-disjoint from $\mu_1(A_1)\cup \mu_1(A_2)$,    
    neither $C_3$ nor $C_4$ appears between $C_1$ and $C_2$. By applying the same argument for $(s^2_2,t^1_2)$ and $(v_4,v_3)$, we know that $C_1, C_2, C_3, C_4$ appear in this order or reverse order in $\ladder_n$.
    Then there is no edge connecting $\mu_1(A_4)$ and $\mu_1(A_1)$, and $\mu_1((s_3,t_3))$ cannot exist, a contradiction.

    \medskip
    
    (2) 
    Let $L'_3$ be the digraph obtained from $\ladder_3$ by deleting edges $(p_1,q_3)$, $(q_1,p_3)$ and contracting $(q_3,p_1)$, $(p_1, p_2)$, $(p_2, p_3)$, $(p_3,q_1)$. See (b) of~\cref{fig:TC3andL3}.
    If $\ladder_3$ is a butterfly minor of $\treeCh_3$ for some $k$, then so is $L'_3$.
    So, it suffices to show that $L'_3$ is not a butterfly minor of $\treeCh_k$.
    
    Suppose that $L'_3$ is a butterfly minor of $\treeCh_k$ for some $k$.
    Let $k$ be the minimum integer such that $\treeCh_k$ has $L'_3$ as a butterfly minor.
    Let $\mu_2$ be a minimal butterfly model of $L'_3$ in $\treeCh_k$.
    Since $\mu_2(L'_3)$ is strongly connected and $k$ is minimal, the digraph $\mu_2(L'_3)$ contains both $(s^1_{k-1}, t^2_{k-1})$ and $(s^2_{k-1},t^1_{k-1})$.
    So, there is a directed cycle $C$ in $\mu_2(L'_3)$ containing $(s^1_{k-1}, t^2_{k-1})$ or $(s^2_{k-1},t^1_{k-1})$.
    By~\cref{lem:propofFak} (2), the directed cycle $C$ contains all vertices in $\treeCh_k$.
    
    One can observe that $C$ contains neither $\mu_2((p_2,q_2))$ nor $\mu_2((q_2,p_2))$; otherwise, $(\{r_a\},I_a, O_a)$ or $(\{r_b\},I_b,O_b)$ does not exist, a contradiction.
    Thus, $C$ contains $\mu_2((a,p_2))$, $\mu_2((p_2,b))$, $\mu_2((b,q_2))$ and $\mu_2((q_2,a))$.

    Now, let $\mu_2((p_2,q_2))=(c,d)$ for some $(c,d)\in E(\treeCh_k)$.
    Since $C$ contains all vertices in $\treeCh_k$, there is no $(d,c)$-path in $\treeCh_k$ that is edge-disjoint with $C$.
    This implies that $\mu_2((q_2,p_2))$ does not exist, a contradiction.
\end{proof}

To show that $\treeCh_k$ is not a butterfly minor of any cylindrical grid, we use a result by Bensmail, Campos, Maia, and Silva~\cite{BensmailCMANA2023Gridembeddable}.
Before stating this result, we introduce some definitions. Let $G$ be a digraph.
For a non-empty proper subset $X$ of $V(G)$, a partition $(X,V(G)\setminus X)$ is a \emph{dicut} of $G$ if there is no edge from $V(G)\setminus X$ to $X$.
The edge set of $(X,V(G)\setminus X)$ is the set of edges from $X$ to $V(G)\setminus X$.
A directed path in $D$ is a \emph{dijoin path} if its edge set intersects the edge set of every dicut of $G$.

For a plane digraph $G$, the dual digraph $G^*$ of $G$ is defined as follows;
\begin{itemize}
    \item $V(G^*)$ is the set of faces of the plane embedding of $G$, and
    \item $e^*\in E(G^*)$ if and only if its head and tail are separated from each other by an edge $e\in E(G)$ and the orientation of $e^*$ is obtained from $e$ by a $90^{\circ}$ clockwise turn. 
\end{itemize}
See~\cref{fig:GwithDual} for an example.

\begin{figure}[!ht]
    \centering
    \includegraphics[scale=0.9]{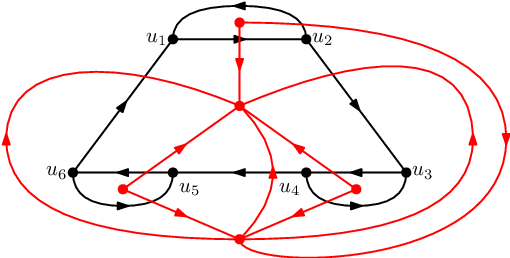}
    \caption{The black digraph is a plane embedding of the digraph $U$ in the proof of \cref{lem:FnotembeddableinGrid} and the red digraph is its dual digraph.}
    \label{fig:GwithDual}
\end{figure}

\begin{theorem}[Corollary 1 in~\cite{BensmailCMANA2023Gridembeddable}]\label{Thm:CyGridiff}
    Let $G$ be a digraph $G$ without sources or sinks. Then $G$ is a butterfly minor of a cylindrical grid if and only if $G$ admits a planar embedding such that $G^*$ has a dijoin path.
\end{theorem}

\begin{lemma}\label{lem:FnotembeddableinGrid}
    Let $k\ge 3$ be an integer. Then $\treeCh_k$ is not a butterfly minor of a cylindrical grid of order $n$ for any $n$.
\end{lemma}
\begin{proof}
Let $U$ be the digraph obtained from $\treeCh_3$ by deleting $(s_3,t^1_2)$, $(v_2,t^1_2)$, $(s^2_2, t_3)$ and contracting $(v_2, v_1)$, $(t^1_2, v_2)$.
See~\cref{fig:GwithDual}.
Clearly, $U$ is a planar digraph without sources or sinks.

It suffices to show that for any positive integer $n$, $U$ is not a butterfly minor of the cylindrical grid of order $n$.
By~\cref{Thm:CyGridiff}, it suffices to show that $U$ does not admit a planar embedding such that its dual digraph has a dijoin path. 

There are two ways to embed the cycle $C\coloneqq u_1u_2u_3u_4u_5u_6u_1$: one in the clockwise direction and the other in the counterclockwise direction.
For either choice, each edge $(u_{2i},u_{2i-1})$ has two options for embedding: one in the interior of $C$, and the other in the exterior.
In any case, three cycles $u_1u_2u_1$, $u_3u_4u_3$, $u_5u_6u_5$ bound distinct faces.

Let $U^*$ be the dual digraph of some plane embedding of $U$.
For each $i\in [3]$, let $c_i$ be the vertex of $U^*$ which corresponds to the face bounded by $u_{2i-1}u_{2i}u_{2i-1}$.
Since each $c_i$ is either a source or a sink, $(\{c_i\},V(U^*)\setminus\{c_i\})$ or $(V(U^*)\setminus\{c_i\},\{c_i\})$ is a dicut.
Thus, $U^*$ has no dijoin path.
\end{proof}

\section{Weak and strong coloring numbers}\label{sec:wcol}

This section considers natural variants of weak and strong coloring numbers for digraphs.
There are two types: a reachability condition where the former vertex is reachable from a latter vertex by a path in a fixed direction, or a condition where the former vertex is reachable from or can reach a latter vertex by a directed path in either direction. We provide formal definitions below.

Let $L$ be a linear ordering of the vertices of a digraph~$G$.
We denote by $\leq_L$ the total order on $V(G)$ induced by the ordering $L$.
For vertices~$v$ and~$w$ of~$G$, we write $v<_Lw$ if $v\leq_Lw$ and $v\neq w$.

Let~$k\in\mathds{N} \cup \{\infty\}$.  
For vertices~$v$ and~$w$ of~$G$, $w$ is \emph{weakly~$(k, \rightarrow)$-reachable} from~$v$ with respect to $L$ if
$w\leq_L v$ and 
there is a directed path $P$ of length at most~$k$ from $v$ to $w$ such that 
for every vertex~$u$ of~$P$, $w\leq_L u$.
We denote by $\onewre_k[G,L,v]$ the set of all weakly~$(k, \rightarrow)$-reachable vertices from~$v$ in $L$.
The \emph{weak~$(k, \rightarrow)$-coloring number} of $L$ is the maximum size of $\onewre_k[G,L,v]$ over all vertices~$v$ of~$G$.
The \emph{weak~$(k, \rightarrow)$-coloring number} of~$G$, denoted by $\onewc_k(G)$, is the minimum weak~$k$-coloring number over all linear orderings of the vertices of~$G$.
For vertices $v$ and $w$ of $G$ we say that $w$ is \emph{strongly $(k,\rightarrow)$-reachable} from $v$ in $L$ if
$w\leq_L v$ and
there exists a directed path $P$ of length at most $k$ from $v$ to $w$ such that
for every internal vertex $u$ of $P$, $v\leq_L u$.
Let $\onesre_k[G,L,v]$ be the set of all vertices that are strongly $(k,\rightarrow)$-reachable from $v$ in $L$.
The \emph{strong $(k,\rightarrow)$-coloring number} of $L$ is the maximum size of $\onesre_k[G,L,v]$ over all vertices $v$ of $G$.
The \emph{strong $(k,\rightarrow)$-coloring number} of $G$, denoted by $\onesc_k(G)$, is the minimum strong $(k,\rightarrow)$-coloring number over all possible linear orderings of the vertices of $G$.

Towards the following observation notice that for any linear ordering $L$ of the vertices of a digraph $G$ and every positive integer $k$, if a vertex $v$ is strongly $(k-1,\rightarrow)$-reachable from a vertex $w$ in $L$, then it is also strongly $(k,\rightarrow)$-reachable from $w$ in $L$ and weakly $(k-1,\rightarrow)$ reachable from $w$ in $L$.
Moreover, if $v$ is weakly $(k-1,\rightarrow)$-reachable from $w$ in $L$, then it must also be weakly $(k,\rightarrow)$-reachable from $w$ in $L$.

\begin{observation}\label{obs:strongandweakrelations}
Let $G$ be a digraph.
Then, for every integer $k\geq 1$ we have
\begin{itemize} 
    \item $\onesc_k(G) \leq \onewc_k(G)$,
    \item $\onesc_{k-1}(G) \leq \onesc_k(G)$, and
    \item $\onewc_{k-1}(G) \leq \onewc_k(G)$.
\end{itemize}
\end{observation}

Meister, Telle, and Vatshelle~\cite[Lemma 3.3]{meister2010recognizing} showed that the Kelly-width~\cite{HUNTER2008206} of a digraph is the same as the strong~$(\infty,\rightarrow)$-coloring number. Hunter and Kreutzer~\cite{HUNTER2008206} gave an equivalent definition of Kelly-width in terms of the so-called elimination width of a linear ordering, and this implies such a relationship. 

For~$k\in\mathds{N}$ and vertices~$v$ and~$w$ of~$G$, $v$ is \emph{weakly~$(k, \leftrightarrow)$-reachable} from~$w$ in $L$ if $v\leq_L w$ and  
\begin{itemize}
    \item there is a directed path $P$ of length at most~$k$ from $w$ to $v$ or
    \item there is a directed path $P$ of length at most~$k$ from $v$ to $w$,
\end{itemize} 
such that for every vertex~$u$ of~$P$, $v\leq_L u$. 
Let $\twowre_k[G,L,v]$ be the set of vertices which are weakly~$(k, \leftrightarrow)$-reachable from~$v$ in $L$.
The \emph{weak~$(k, \leftrightarrow)$-coloring number} of $L$ is the maximum size of $\twowre_k[G,L,v]$ for all vertices~$v$ of~$G$.
The \emph{weak~$(k, \leftrightarrow)$-coloring number} of~$G$, denoted by $\twowc_k(G)$, is the minimum of the weak~$k$-coloring numbers among all possible linear orderings of the vertices of~$G$.
When $k=|V(G)|$, then we replace $k$ with $\infty$.

The parameter $\twowc_k(G)$ was introduced by Kreutzer, Rabinovich, Siebertz, and Weberst\"{a}dt~\cite{KreutzerRSW2017}. Kreutzer, Ossona de Mendez, Rabinovich, and Siebertz~\cite{KORS2017} introduced directed treedepth of a digraph $G$ as $\twowc_{\infty}(G)$.

We show that the cycle rank of a digraph is the same as its weak~$(\infty,\rightarrow)$-coloring number minus one.

\begin{lemma}\label{lem:wcolcrequ}
    For every digraph $G$, $\crank(G)=\onewc_{\infty}(G)-1$.
\end{lemma}
\begin{proof}
    Let $G$ be a digraph. Note that $\onewc_{\infty}(G)$ is the maximum $\onewc_{\infty}(C)$ over all its strongly connected components $C$, and $\crank(G)$ is the maximum $\crank(C)$ over all its strongly connected components $C$. Thus, we may assume that $G$ is strongly connected.
    
    We first claim that $\crank(G)\le \onewc_{\infty}(G)-1$.
    We prove this by induction on $m=\onewc_{\infty}(G)$.
    If $m=1$, then $G$ has one vertex, as $G$ is strongly connected.
    Thus, we have $\crank(G)=0$ and we are done.
    So, we assume that $m\ge 2$.
    
    Let $L=(v_1, v_2, \ldots, v_n)$ be a linear ordering of $V(G)$ such that $\abs{\onewre_{\infty}[G, L, v]}\le m$ for all $v\in V(G)$.
    Let $G_1, \ldots, G_t$ be the strongly connected components of $G-v_1$.

    We verify that for each $i\in [t]$, $\onewc_{\infty}(G_i)\le m-1$.
    Suppose this is not true.
    Then in any linear ordering $L'$ of $V(G_i)$, there is a vertex $v'\in V(G_i)$ such that $\abs{\onewre_{\infty}[G_i, L', v']}\ge m$.
    Applying this to the restriction of $L$ on $G_i$, we have that $\abs{\onewre_{\infty}[G, L, v]}\ge m+1$, as $v_1\in \onewre_{\infty}[G, L, v]$.
    This is a contradiction. Thus, $\onewc_{\infty}(G_i)\le m-1$.

    By the induction hypothesis, $\crank(G_i)\le m-2$ and 
    there is a cycle rank decomposition $(T_i, r_i)$ of $G_i$ of height at most $m-1$.
    Let $T$ be the tree obtained from the disjoint union of $T_1, \ldots, T_t$ by adding $v_1$, adding edges $v_1r_i$ for each $i\in [t]$ and finally, assigning $v_1$ as a root of $T$.
    Then $(T, v_1)$ is a cycle rank decomposition of $G$ of height at most $m$. So, $\crank(G)\le m-1$, as desired. 
    
    \medskip
    
    We show that $\onewc_{\infty}(G)\le \crank(G)+1$.
    We prove this by induction on $m=\crank(G)$. 
    If $m=0$, then $G$ has one vertex because $G$ is strongly connected, and thus, we have $\onewc_{\infty}(G)=1$.
    So, we assume $m\ge 1$.
    
    Let $(T,r)$ be a cycle rank decomposition of $G$ of height $m$.
    Let $T_1, \ldots, T_t$ be the components of $T-r$, and for each $i\in [t]$, let $r_i$ be the root of $T_i$ and $G_i \coloneqq G[V(T_i)]$.
    By the definition of a cycle rank decomposition, each $G_i$ is strongly connected and $(T_i, r_i)$ is a cycle rank decomposition of $G_i$ of height at most $m-1$.
    By induction, there is a linear ordering $L_i$ of $V(T_i)$ where $\abs{\onewre_{\infty}[G_i, L_i, v]}\le (m-1)+1=m$ for every $v\in V(T_i)$.

    We define an auxiliary graph $H$ with $V(H)=\{w_1, \ldots, w_t\}$ and for distinct $p,q\in [t]$, 
    \begin{itemize}
        \item $(w_p,w_q)\in E(H)$ if and only if there are descendants $a$ and $b$ of $r_p$ and $r_q$ in $T$ respectively, such that $(a,b)\in E(G)$.
    \end{itemize} 
    By the definition of a cycle rank decomposition, $H$ is an acyclic digraph. So, there is a linear ordering $L_H$ of $V(H)$ such that for every edge $(x,y)$ in $H$, $x$ appears before $y$ in $L_H$.
    Now, we obtain a linear ordering $L$ from $L_H$ by replacing each $w_i$ with the ordering $L_i$ of $V(G_i)$, and adding~$r$ at the beginning.

    Observe that $\abs{\onewre_{\infty}[G, L, r]}=1$ and for every $v\in V(G_i)$,  
    \[\abs{\onewre_{\infty}[G, L, v]} \le \abs{\onewre_{\infty}[G_i, L_i, v]}+1\le m+1,\]
    because no vertex in any other $G_j$ is weakly reachable from $v$. Thus, $\onewc_{\infty}(G)\le m+1$.
\end{proof}

\section{Conclusion}\label{sec:conclusion}

In this paper, we have established that the classical digraph parameter \textsl{cycle rank} is universally obstructed by three infinite families of digraphs, namely directed ladders, cycle chains, and tree chains.
Here the term \textsl{universal obstruction} means the following:
There exists a function $f\colon\mathds{N}\to\mathds{N}$ such that for every $k\in\mathds{N}$, every digraph with cycle rank at least $f(k)$ contains one of $\ladder_k$, $\cycleCh_k$, or $\treeCh_k$ as a butterfly minor -- see our main result \cref{thm:cyclerankmainthm} -- while additionally every digraph that contains one of $\ladder_k$, $\cycleCh_k$, or $\treeCh_k$ as a butterfly minor must have cycle rank at least $\lfloor\log k\rfloor$ -- see \cref{lem:butterflyminor} as well as  \cref{thm:crankcc,lem:Lk,lem:Fk}.

For minors in undirected graphs, there are many comparable theorems about universal obstructions for a variety of parameters, see~\cite{paul2023universal} for a recent survey and~\cite{PaulPTW2024Obstructions} for a collection of infinitely many such results. 
In the world of digraphs, however, such results are extremely rare.
Even for relatively well-understood parameters such as directed pathwidth (see, for example, \cite{Erde2020Directed}), no analogue of the pathwidth result of Robertson and Seymour~\cite{RobertsonS1983GraphMinorsI} is known.

\begin{question}
What are the universal obstructions for directed pathwidth?
\end{question}

Indeed, finding the unavoidable butterfly minors for digraphs of large directed pathwidth seems to be a big challenge.
A positive answer to the following question might provide some insight.

\begin{question}
Does every digraph of large directed pathwidth contain a long cycle chain or a long directed ladder as a butterfly minor?
\end{question}

As mentioned in the introduction, Giannopoulou, Hunter, and Thilikos~\cite{GIANNOPOULOU2012searchinggame} gave an XP-algorithm for computing the cycle rank of a given digraph.
Gruber~\cite{Gruber2012} showed that there is a polynomial-time approximation algorithm with factor $\mathcal{O}((\log n)^{3/2})$ for cycle rank. Both left the existence of an FPT-algorithm for cycle rank as an open problem and we would like to revive this question here.

\begin{question}[\cite{Gruber2012}]
Is there an FPT algorithm for computing the cycle rank?
\end{question}

Meister, Telle, and Vatshelle~\cite{meister2010recognizing} asked whether there exists an FPT algorithm for computing Kelly-width. This remains an open problem.

Finally, we would like to pose a more open-ended question.
As we discussed, there exist notions for weak and strong coloring numbers for digraphs that correspond to cycle rank and Kelly-width in the same way as their undirected counterparts correspond to treedepth and treewidth.
The study of coloring numbers is closely related to the structure of sparse graph classes through the notion of bounded expansion.
Such graph classes can be characterized in many different ways; it seems, therefore, natural to ask for a directed analogue as follows.

We say a class $\mathcal{C}$ of digraphs has \emph{bounded reachable expansion} if there exists a function $f\colon\mathds{N}\to\mathds{N}$ such that $\onewc_{r}(G)\leq f(r)$ for all $r\in\mathds{N}_{\geq 1}$ and $G\in\mathcal{C}$.

\begin{question}
Are there other natural characterizations of digraph classes of bounded reachable expansion?
\end{question}

\end{document}